\documentclass[11pt,UTF-8,reqno]{amsart}
\usepackage{enumerate, bbm}
\usepackage{amssymb,url,color, booktabs}
\usepackage{mathrsfs}
\usepackage{cite}
\usepackage[utf8]{inputenc}

\usepackage[numbers,sort&compress]{natbib} 

\usepackage{color}
\usepackage[colorlinks=true]{hyperref}
\hypersetup{
    linkcolor=blue,          
    citecolor=red,        
    filecolor=blue,      
    urlcolor=cyan
}

\usepackage{color}
\definecolor{MyDarkBlue}{cmyk}{0.8,0.3,0.8,0.4}
\definecolor{yellow}{rgb}{0.99,0.99,0.70}
\definecolor{white}{rgb}{1.0,1.0,1.0}
\definecolor{black}{rgb}{0.00,0.00,0.00}

\numberwithin{equation}{section}

\newcommand{\be}{\begin{eqnarray}}
\newcommand{\ee}{\end{eqnarray}}
\newcommand{\ce}{\begin{eqnarray*}}
\newcommand{\de}{\end{eqnarray*}}
\newtheorem{theorem}{Theorem}[section]
\newtheorem{lemma}{Lemma}[section]

\newtheorem{definition}{Definition}[section]

\newtheorem{corollary}{Corollary}[section]
\newtheorem{hypothesis}{Hypothesis}[section]
\newtheorem{remark}[theorem]{Remark}

\def\[{{\Big[}}
\def\]{{\Big]}}
\def\<{{\langle}}
\def\>{{\rangle}}
\def\({{\big(}}
\def\){{\big)}}

\def\bb2{{\boldsymbol{2}}}

\def\={&\!\!=\!\!&}

\def\b1{{\mathbbm 1}}

\def\geq{\geqslant}
\def\leq{\leqslant}

\def\[{{\Big[}}
\def\]{{\Big]}}
\def\<{{\langle}}
\def\>{{\rangle}}

\def\={&\!\!=\!\!&}
\def\bt{\begin{theorem}}
\def\et{\end{theorem}}
\def\bl{\begin{lemma}}
\def\el{\end{lemma}}
\def\br{\begin{remark}}
\def\er{\end{remark}}
\def\bd{\begin{definition}}
\def\ed{\end{definition}}
\def\bc{\begin{corollary}}
\def\ec{\end{corollary}}

\def\geq{\geqslant}
\def\leq{\leqslant}

\def\<{\langle} \def\>{\rangle}

\def \eref#1{\hbox{(\ref{#1})}}

\allowdisplaybreaks

\begin{document}

\title[Stochastic 3D Navier-Stokes Equations in $\mathbb{H}^{\frac{1}{2}}$]
{Stochastic Forced 3D Navier-Stokes Equations in $\mathbb{H}^{1/2}$-Space $^\dagger$}


\thanks{$\dagger$
This work is supported by National Key R\&D program of China (No.~2023YFA1010101). The research of W. Hong is supported by  NSFC (No.~12401177) and Basic
Research Program of Jiangsu (No.~BK20241048).  The research of S. Li is supported by NSFC (No.~12371147). The research of W. Liu is supported by NSFC (No.~12571155). }

\maketitle
\centerline{\scshape Wei Hong, Shihu Li,   Wei Liu\footnote{Corresponding author: weiliu@jsnu.edu.cn}}

\vspace{3mm}
 {\footnotesize
\centerline{   School of Mathematics and Statistics, Jiangsu Normal University, Xuzhou 221116, China}}

\begin{abstract}
In the classical work \cite{FK}, Fujita and Kato established  the local existence of solutions to the 3D  Navier-Stokes equations in the critical $\mathbb{H}^{1/2}$-space. In this paper, we  are concerned with the global well-posedness of the stochastic forced 3D  Navier-Stokes equations  in the $\mathbb{H}^{1/2}$-space under general initial  conditions, where the stochastic forcing comprises a transport forcing and a nonlocal turbulent forcing.
  In this setting, the random noise  is shown to provide a  regularization effect on the energy estimates, which we obtain by constructing suitable Lyapunov functions. However, its nonlocality  also brings analytical challenges.  We develop a bootstrap type estimate based on the kinematic viscosity together with a delicate stopping time argument to  prove the  global existence and  uniqueness of solutions, as well as   continuous dependence on the initial value. Furthermore, we also investigated the long-time behavior of the stochastic forced 3D  Navier-Stokes equations.



\end{abstract}
\maketitle

\tableofcontents

\section{Introduction}
\subsection{Background of Navier–Stokes equations}

\subsubsection{Deterministic case}

The classical 3D Navier–Stokes equations describe the time evolution of an
incompressible fluid, which are given by
\begin{equation}\label{DNS}
\begin{cases}
&\!\!\!\!\!\!\partial_tu + (u \cdot \nabla ) u
  -  \nu \Delta u  + \nabla p = f , \\
&\!\!\!\!\!\!\nabla\cdot u =0, \\
&\!\!\!\!\!\!u(0) = x,
\end{cases}
\end{equation}
where $u:=u(t,\xi)$ is the unknown solenoidal velocity field of the fluid, $p:=p(t,\xi)$ denotes the pressure and $f:=f(t,\xi)$ is an external force field acting on the fluid.

It is well known that the system \eref{DNS} possesses an invariant scaling structure; if $u$ solves \eref{DNS}, then the rescaled functions $$u_\lambda(t,\xi):=\lambda u(\lambda^2t,\lambda \xi)$$ also satisfy \eref{DNS} with $f$ replaced by $f_\lambda(t, \xi) := \lambda^3f(\lambda^2t,\lambda \xi)$ for all $\lambda>0$.
Furthermore, we call that the solution space $\mathbb{X}$ is critical if $\|u_\lambda(0,\cdot)\|_{\mathbb{X}}=\|u(0,\cdot)\|_{\mathbb{X}}$, for all $\lambda>0$. Examples of such critical spaces include
$$\mathbb{H}^{1/2}\hookrightarrow L^3\hookrightarrow B^{-1+\frac{3}{p}}_{p,\infty}\hookrightarrow {BMO}^{-1},$$
where $p<\infty$. The critical spaces
play an crucial role for the theory of nonlinear partial differential equations. If the
equation has a class of scaling invariance, then it coincides with the most suitable space
to construct the solution which is expected unique and regular.

In the classical work \cite{FK}, Fujita and Kato proved the local existence of solutions to 3D Navier-Stokes equations in the critical $\mathbb{H}^{1/2}$-space. Kato~\cite{Ka84}, Giga and Miyakawa~\cite{GM85} studied the solutions of $n$-dimensional Navier-Stokes equations with initial data in the critical $L^n$-space. For more works on the study of Navier-Stokes equations in  critical spaces, we also refer interested readers to \cite{CP96,FT25,KK11,KT01} and references therein.

\subsubsection{Stochastic case}
The relationship between  Navier-Stokes equations and hydrodynamic turbulence is regarded as one of the most challenging problems in fluid mechanics. It is commonly believed that the start of turbulence is associated with the randomness of the background fluid movement. In order to capture the chaotic characteristics and intrinsic randomness of fluids, it is highly necessary to study the corresponding stochastic forced Navier-Stokes equations.

Since the pioneering work \cite{BT73} by Bensoussan and Temam,  the stochastic forced Navier-Stokes equations have attracted extensive attention. It is known that the stochastic forced Navier-Stokes equations could be derived from the stochastic flow as well as the Newton's second law. For example,  Mikulevicius and Rozovskii in \cite{MR04} assumed that the dynamics of fluid particles were characterized by the
stochastic Lagrangian flow
\begin{equation}\label{SF}
\begin{cases}
&\!\!\!\!\!\!d \mathcal{X}(t,\xi)  =  u(t,\mathcal{X}(t,\xi))dt+  \zeta(t,\mathcal{X}(t,\xi))\circ\dot{\mathcal{W}}(t) , \\
&\!\!\!\!\!\!\mathcal{X}(0,\xi) = \xi,
\end{cases}
\end{equation}
with undetermined local characteristics $u(t, \xi)$, where $\circ$ denotes the  integration of Stratonovich type. In the flow \eref{SF}, the random field $\zeta (t, \xi)\circ\dot{\mathcal{W}}(t)$ models the turbulent part of the velocity field, while $u(t, \xi)$ models its regular component,
this idea of splitting up the velocity field into a sum of slow oscillating (deterministic) component and fast oscillating (stochastic) component can be traced back to the works of Reynolds in 1880s as mentioned in \cite{MR04}.

By the classical scheme of Newtonian fluid mechanics, Mikulevicius and Rozovskii derived the following stochastic forced Navier–Stokes equations
\begin{equation}\label{sNS01}
\begin{cases}
&\!\!\!\!\!\!\partial_t u + (u \cdot \nabla ) u
  -  \nu \Delta u  + \nabla p =    [(\zeta\cdot\nabla)u-\nabla \tilde{p}+g(u)]\circ\dot{\mathcal{W}}(t) , \\
&\!\!\!\!\!\!\nabla\cdot u =0, \\
&\!\!\!\!\!\!u(0) = x,
\end{cases}
\end{equation}
where $\tilde{p}$ denotes the unknown pressure, see \cite[Section 2]{MR04} for more details.
During recent decades, the research on the well-posedness and asymptotic properties of stochastic forced Navier-Stokes equations (\ref{sNS01}) have been widely investigated, the interested readers can refer to \cite{AKX25,AV24,BBP22,BZPW,CD19,DD03,FL21,FR08,HM06,HZZ,KNS20,Kim10,LR1} and the references
therein.

\subsection{Well-posedness}\label{secmr}

In recent years, there  has been a growing interest in the study of stochastic forced Navier-Stokes equations in the critical space. Several notable works have contributed to this area.  For example, Agresti and Veraar \cite{AV24}  applied the maximal regularity theory to construct local solutions for initial data in a range of Besov spaces and obtained the global existence of solutions with high probability for small initial data in critical spaces. Recently, Aydin, Kukavica and Xu \cite{AKX25} also considered the global existence of solutions to the stochastic forced Navier-Stokes equations with multiplicative noises and with small initial data in the critical $\mathbb{H}^{1/2}$-space, they proved that the solutions
exist globally in time with high probability, which is close to $1$ if the initial data and the noise are sufficiently small. 

In addition, the well-posedness of stochastic forced 3D Navier-Stokes equations with subcritical initial data has also drawn a lot of attention.
For instance, Kim~\cite{Kim10} first obtained the local existence of solutions in 3D
with $\mathbb{H}^s$ ($s > 1/2$) initial data, and derived the global existence of solutions with a high probability for small initial data.
Fernando et al. \cite{FRS} and Kukavica et al. \cite{KXZ22} studied the local existence and uniqueness of solutions to stochastic forced 3D Navier-Stokes equations in $L^p$-spaces.

The above literatures shows that
\begin{itemize}
  \item In the stochastic setting, the local existence and uniqueness of solutions for stochastic forced 3D Navier-Stokes equations with critical initial data have been  well studied. However, the global well-posedness results  have only been established under the small initial data assumption, which coincides with the deterministic case, and with high probability (not probability one). Up to now, the global well-posedness of  the stochastic forced 3D Navier-Stokes equations in the  critical space has not been resolved under general initial conditions.

  \item Due to the lack of the global well-posedness  under general initial conditions, the subsequent and considerably more delicate problem, for example studying the long-time behavior, such as the  ergodicity  of the associated Markov semigroup, of the stochastic forced 3D Navier-Stokes equations seems far out of reach.
\end{itemize}

Motivated by the aforementioned points, our goal here is to address the global well-posedness of the stochastic forced 3D Navier-Stokes equations in the critical $\mathbb{H}^{1/2}$-space, and to further investigate their long-time behavior. To be more precise, we consider the following stochastic Navier-Stokes equations  on the 3D torus $\mathbb{T}^3$,  which involve  a transport noise and, possibly, a nonlocal stochastic turbulent forcing,  i.e.
\begin{equation}\label{sNS}
\begin{cases}
&\!\!\!\!\!\!d u + [(u \cdot \nabla ) u
  -  \nu \Delta u  + \nabla p]dt =  \sum_{i=1}^{\infty}\big[((\zeta_i\cdot\nabla)u-\nabla \tilde{p})d\beta_i(t)+ g_i(t,u) d\hat{\beta}_i(t)\big] , \\
&\!\!\!\!\!\!\nabla\cdot u =0, \\
&\!\!\!\!\!\!u(0) = x\in \mathbb{H}^{1/2},
\end{cases}
\end{equation}
where $(\beta_i)_{i\geq 1}$ and $(\hat{\beta}_i)_{i\geq 1}$ are sequences of independent standard Brownian motions on a complete filtered probability space. The transport noise is now widely serves as a representation for studying the influence of small-scale turbulence on large-scale fluid dynamics (cf.~\cite{MR04}). In the present work, we are also interested in capturing the intrinsic impact for a class of nonlocal stochastic turbulent forcing in fluid dynamics, see Subsection \ref{noise} for more precise statement.

 We first present the global well-posedness of Eq.~(\ref{sNS}) with probability one    under suitable  assumptions concerning the coefficients $\zeta$ and $g$. More precise statement of our results is provided in Theorem \ref{thg} and Theorem \ref{thF}.

\begin{theorem}\label{th01}
Suppose that Hypothesis \ref{h2}-\ref{h3} (see Subsection \ref{sub2.1} below) hold.
\begin{enumerate}[(i)]
  \item For any initial value $x\in\mathbb{H}^{1/2}$, Eq.~$(\ref{sNS})$ has a unique strong solution $u(t)$.
Moreover,  the following energy moment estimates hold
\begin{equation}\label{esap}
\sup_{t\in[0,T]}\mathbf{E}|u(t)|_{1/2}^{2-\gamma}+\mathbf{E}\int_0^T|u(t)|_{3/2}^{2-\gamma}dt<\infty,
\end{equation}
where the constant $\gamma\in(1,2)$.
  \item The solution of Eq.~$(\ref{sNS})$ is continuous with respect to the initial value.
In particular, the corresponding Markov semigroup $(\mathcal{T}_t)_{t\geq 0}$ satisfies the Feller property.
\end{enumerate}

\end{theorem}

In the case of critical initial data, notable progress has been made in the study of stochastic Navier-Stokes  equations, primarily focusing on proving the local well-posedness  and the global well-posedness for small initial data (see e.g. \cite{AV24,AKX25}). To the best of our knowledge, this is the first result concerning the global well-posedness of stochastic forced 3D Navier-Stokes equations in the critical space. A key feature distinguishing our work from previous results is that, by explicitly leveraging the regularization effects induced by  random noise, we are able to remove the commonly adopted assumption of small initial data.

By the way, we also mention that the regularization by noise for stochastic partial differential equations has been achieved significant progress in recent years.
A foundational work by Flandoli, Gubinelli and Priola \cite{FGP10} studied the well-posedness of the linear transport equation with a globally H\"{o}lder continuous and bounded vector field, which is ill-posed in the deterministic case.
  For  studies on the regularization by noise for stochastic forced Navier-Stokes equations, we refer to the works \cite{FL21,RZZ14} for the case of the linear noise and the works \cite{CL24,HLL24,TW22} for the case of the nonlinear noise. However, due to the low regularity of the critical initial data, the framework proposed in  \cite{CL24,HLL24,TW22} cannot be applied to the stochastic forced 3D Navier-Stokes equations in the critical $\mathbb{H}^{1/2}$ space.


\subsection{Long-time behavior}\label{secmr}

Building upon the global well-posedness, we are interested in the asymptotic behavior of Eq.~\eref{sNS}, as $t$ tends to infinity.  Particularly, a natural problem is to consider whether the solution decays and to give an explicit  estimate of the decay rate.

The decay problem for Navier-Stokes equations was first raised by Leray \cite{Le34}, then was affirmatively solved by Kato \cite{Ka84} for the Cauchy problem in $\mathbb{R}^3$.  Borchers and Miyakawa \cite{BM90} deduced an algebraic decay rate for the total kinetic energy of weak solutions for non-stationary Navier-Stokes equations. For the decay problem in critical spaces, Benameur and  Selmi \cite{BS10} proved  the 3D periodic Navier-Stokes with $\mathbb{H}^{1/2}$ initial data decays exponentially fast to zero as time tends to infinity. Recently, Fujii and Tsurumi \cite{FT25} also investigate the decay problem for the forced Navier-Stokes equations in critical Besov spaces. Of course, concerning the decay problem of 3D Navier-Stokes equations in critical spaces, it is usually necessary to assume small initial data conditions. For more results on deterministic 3D Navier-Stokes equations in this direction, we refer to \cite{BBT,OE00,Sc85,Sc91} and the reference therein.

In this work, we further investigate the decay estimate of the stochastic forced 3D Navier-Stokes equations in the $\mathbb{H}^{1/2}$-space, under a stronger assumption compared to Theorem \ref{th01}, and arrive at the following result (for more details, please refer to Theorem \ref{th4}).

\begin{theorem}\label{th002a}
Suppose that  Hypothesis \ref{h2}-\ref{h3}   hold with $(\mathbf{H}_{g}^{2})$ replaced by $(\mathbf{H}_{g}^{2*})$ (see Subsection \ref{subsre2} below).   There exist a constant $\kappa>0$,  and an $\mathbf{P}$-a.s. finite random time $\tau$ such that for any initial value $x\in\mathbb{H}^{1/2}$,
  \begin{equation*}
|u(t)|_{1/2}^{2-\gamma}\leq e^{-\kappa t}|x|_{1/2}^{2-\gamma},~~t\geq \tau,~~~~~~\mathbf{P}\text{-a.s.}.
\end{equation*}
\end{theorem}

In contrast to the deterministic case, which is restricted to small initial values (see e.g.~\cite{BS10,FT25}),  Theorem \ref{th002a} shows that for any $\mathbb{H}^{1/2}$ initial data, the solution of Eq.~$(\ref{sNS})$ decays exponentially fast to zero as time tends to infinity. As a  consequence, we further obtain the ergodicity for the stochastic forced 3D Navier-Stokes equations in  $\mathbb{H}^{1/2}$-space.

\begin{theorem}\label{co01}
Suppose that the assumptions in Theorem \ref{th002a} hold.
 There exists a unique invariant measure to the  Markov semigroup  associated to Eq.~$(\ref{sNS})$.
\end{theorem}

Due to the lack of global well-posedness of the stochastic forced 3D Navier-Stokes equations, the only available result in the existing works is
the ergodicity for every Markov selections. For instance, Da Prato and Debussche \cite{DD03} first proved the existence of  solutions in the mild sense to 3D Navier–Stokes equations with additive noise. Due to the lack of uniqueness, they  constructed a selected transition semigroup and established its ergodicity. In the work \cite{FR08}, Flandoli and Romito also proved the Feller property of the selected Markov semigroup to stochastic 3D Navier–Stokes equations with additive noise. However, due to the lack of continuity of solutions in the space of finite energy, the Markov property holds almost everywhere in time. Note that the ergodicity established in previous works only relates to the Markov process constructed therein.


Theorem \ref{co01} establishes, for the first time, the ergodicity  of the Markov semigroup associated to  the  stochastic forced 3D Navier-Stokes equations (\ref{sNS}), rather than regarding a Markov selection. We point out that even for initial data with higher regularity, such as subcritical or even smooth initial conditions, this result still remains new.

\subsection{Nonlocal stochastic forcing} \label{noise}
Our main results can be applied to a class of {\it{nonlocal}} stochastic forcing.
In the field such as the fluid mechanics, the research on models with nonlocal forcing holds significant physical motivation and importance. For instance, Burgers in the seminal work \cite{Bur39} proposed the  mathematical models of the turbulence, in which a nonlocal turbulent forcing $\Phi(u)=\int_{\mathcal{O}}|u(t,\xi)|^2d\xi$ is involved  to
simulate flow in a channel. As stated in \cite{HO},  the now well-known Burgers equation is in fact  a simplified version of the {\it{nonlocal}} Burgers system, see also \cite{DKL,Bur48} for  more precise physical consideration.

In the well-known work \cite{Lad67}, Ladyzenskaja proposed a nonlocal model, which can be seen as a modification of Navier–Stokes equations, where  a nonlocal viscosity  $\nu(u)=\nu\Big[1+\int_{\mathcal{O}}|\nabla u(t,\xi)|^2d\xi\Big]$ is introduced. This type of nonlocal viscosity forcing  is related to the total dissipation energy of Newtonian fluid. The authors in \cite{CR05} showed that the model introduced in \cite{Lad67} is a limit of systems describing motions of incompressible viscous heat-conducting fluids of Newtonian and Bingham types.

Owing to the impact of random factors such as stochastic environment, a  question arises: how to characterize fluid models with nonlocal stochastic forcing? Inspired by aforementioned works,  we consider the stochastic 3D Navier-Stokes equations perturbed by the following nonlocal   forcing
\begin{equation}\label{ex001}
\sum_{i=1}^{\infty}\int_0^tg_i(s,u(s)) d\hat{\beta}_i(s)=\sum_{i=1}^{\infty}\int_0^t\alpha_i\Big[1+\int_{\mathbb{T}^3}|\nabla u(s,\xi)|^2d\xi\Big]u(s)  d\hat{\beta}_i(s).
\end{equation}
We found that this type of stochastic forcing  shows a certain regularization effect on deterministic Navier-Stokes systems.  In particular, applying our main results, one can obtain the global well-posedness of the corresponding stochastic forced 3D Navier-Stokes equations, as well as analyze their long-time behavior  including the ergodicity. Nevertheless, the  stochastic  forcing (\ref{ex001}) also brings highly nontrivial  challenges due to its nonlocality, in particular, it is not a weakly continuous operator, see the next subsection for more  detailed statements for the difficulties in the proof.

We also mention that the nonlocal forcing also has a wide range of applications in other fields. For instance,  the nonlocal forcing (\ref{ex001}) is  connected with the non-degenerate
 Kirchhoff type forcing in the study of the wave equations (cf.~\cite{K83}), see also \cite{APS,G13} and reference therein. One can see also \cite{WK50}  for a similar axial force introduced by Woinowsky-Krieger to describe the dynamic bucking of a hinged extensible beam.

\subsection{Strategy of the proof}

To study the global well-posedness and long-time behaviour of the stochastic forced 3D Navier-Stokes equations,  we do not follow the arguments in the deterministic case, see e.g.~\cite{FK}. Moreover, different from existing results on stochastic Navier-Stokes equations with critical initial data (cf.~\cite{AV24,AKX25}), our goal is not to extend existing deterministic results to the stochastic setting. Instead, we  capture the impact of the nonlocal stochastic forcing on the system, and thereby characterize the well-posedness and long-time behavior of the stochastic system. However,  in the critical space, it is quite non-trivial to close the energy estimates for the $\mathbb{H}^{1/2}$-norm and establish the continuity of solutions belonging to  $C([0,T];\mathbb{H}^{1/2})$, which is significantly different from the recent result established in the variational framework \cite{HLL24}. The main steps and difficulties in the proof are outlined below.


$\bullet$ \textbf{Lyapunov method for energy estimates}: To construct a (probabilistically) weak solution to Eq.~(\ref{sNS}), we employ a Faedo-Galerkin approximation  and derive energy estimates  for the $\mathbb{H}^{1/2}$-norm by using a suitable Lyapunov function, in which we observe that the nonlocal stochastic forcing (\ref{ex001})  might provide an intrinsic balance for the energy of the  system via a ``second order'' effect. Moreover, by designing a different Lyapunov function we obtain a decay estimate for solutions of  Eq.~(\ref{sNS}) under stronger assumptions on the noise.
Unlike existing works for Leray-Hopf weak solutions in the $L^2$-setting (cf.~e.g.~\cite{BM13}),   we can only obtain the energy moment estimates of order $\alpha$ ($\alpha<1$) in the critical space.

Although the nonlocal stochastic forcing (\ref{ex001}) contributes a certain regularization effect, some highly nontrivial  difficulties are appeared due to its nonlocality and  the lack of finite second moments.

 $\bullet$  \textbf{Difficulties induced by nonlocal stochastic forcing}:
In order to pass to the limit of the Faedo-Galerkin approximation, we want to show the tightness of the approximation sequence. However, due to the natural nonlocality of the stochastic forcing (\ref{ex001}), the equicontinuity  essentially requires the uniform  $\mathbb{H}^1$-estimates. Moreover, the nonlocal operator $g(t,\cdot)$ in stochastic forcing (\ref{ex001}) is not weakly continuous, from this reason, it is necessary to establish the strong convergence in  $\mathbb{H}^1$-space of the approximation sequence. On the other hand, since the finite second moments are not available,  we need to carefully deal with the convergence of all terms in the approximation sequence.

These points are extremely challenging for the stochastic forced 3D Navier-Stokes equations in the critical $\mathbb{H}^{1/2}$-space. Our new ideas, which allow us to overcome the above mentioned technical difficulties,
is to develop a bootstrap type estimate.

\vspace{1mm}
$\bullet$ \textbf{Bootstrap type estimates}:
Owing to the inherent parabolic smoothing  from the kinematic viscosity, we adopt the bootstrap method for iterative regularity enhancement to improve the regularity of the approximating sequence, and derive uniform $\mathbb{H}^1$-estimates  over time intervals where zero point is excluded (i.e.~$(0,T]$).  This  bootstrap type estimate plays a crucial role in employing the stochastic compactness argument and  passing to the limit of  the approximating sequence. However, unlike the results in deterministic  partial differential equations, during the bootstrapping process,
we also need to leverage appropriate  regularization effect induced by random noise, thereby only establishing uniform  logarithmical moment estimates for $\mathbb{H}^1$-norm, see Lemma \ref{lem5} for  details.

$\bullet$ \textbf{Continuity of solutions}:  Based on  the bootstrap type estimates and the cut-off technique, one can get a limiting process $\tilde{u}$ associated with the Faedo-Galerkin approximation, which only lies in
$C((0,T];\mathbb{H}^{-1/2})$. A further problem is to verify that  $\tilde{u}$ is a (probabilistically) weak solution of Eq.~(\ref{sNS}) in the sense of Definition \ref{dew}. More precisely, one need to show that $\tilde{u}$ belongs to $C([0,T];\mathbb{H}^{1/2})\cap L^2([0,T];\mathbb{H}^{3/2})$ and satisfies the equality (\ref{de2}). However, in view of the convergence of the Faedo-Galerkin approximation, we do not know whether $\tilde{u}(t)$ is continuous at $t=0$.
To handle this problem, we will construct a  $\tilde{\mathbf{P}}\otimes dt$-version (denoted by $\bar{u}$) of $\tilde{u}$ with $\bar{u}(0)=x$, which satisfies the original equation. Then we will verify that $\bar{u}$ belongs to   $C([0,T];\mathbb{H}^{1/2})\cap L^2([0,T];\mathbb{H}^{3/2})$.
This will be carried out in two steps: First, by employing a delicate stopping time argument, we  demonstrate that $\bar{u}\in C([0,T];\mathbb{H}^{-1/2})$; Then, by utilizing a localization procedure and employing an important lemma developed in \cite[Proposition 4.2]{GC1}, we  conclude  $\bar{u}\in C([0,T];\mathbb{H}^{1/2})\cap L^2([0,T];\mathbb{H}^{3/2})$.

Building upon the above-mentioned arguments, we are  able to construct a (probabilistically) weak solution of Eq.~(\ref{sNS}).
Then utilizing the commutator estimates and the infinite dimensional version of  Yamada-Watanabe theorem, the existence and uniqueness of (probabilistically) strong solutions are also derived. Furthermore, as a consequence of  the aforementioned decay estimate, we also establish the ergodicity for stochastic forced 3D Navier-Stokes equations.

The rest of manuscript is organized as follows. In Section \ref{sec-preliminaries NSES}, we recall  some basic definitions of function spaces and operators. In Section \ref{mainR}, we state the well-posedness result of  the stochastic forced 3D Navier-Stokes equations and give its proof.  Furthermore, in Section \ref{sec4}, we derive a decay estimate of solutions to the stochastic forced 3D Navier-Stokes equations, and as a consequence, we investigate the ergodicity. Some useful lemmas are left in Section \ref{appendix} as Appendix.
Throughout the work, we use $a\lesssim b$ to denote $a\leq C b$ for some unimportant constant $C>0$. We also use $a\lesssim_{\lambda} b$ to denote $a\leq C_\lambda b$ when we want to emphasize that the implicit constant $C$ depends on $\lambda$.

\section{Preliminaries}  \label{sec-preliminaries NSES}

\subsection{Notations and definitions}
In this subsection, we collect some  definitions of function spaces that are commonly used in the analysis of stochastic Navier-Stokes equations on the periodic domain $\mathbb{T}^3=\mathbb{R}^3/(2\pi\mathbb{Z})^3$. Throughout the paper, we restrict ourselves to deal with  flows which have zero average on $\mathbb{T}^3$, i.e.,
$$\int_{\mathbb{T}^3}u(\xi)d\xi=0,$$
where $d\xi$ denotes the volume measure on $\mathbb{T}^3$.

Let $l^2$ be the Hilbert space consisting of all sequences of square summable real numbers with standard norm $|\cdot|_{l^2}$. We denote by $L^p:=L^p(\mathbb{T}^3)^3$, $p\geq 1$,   the Banach space  of Lebesgue measurable $\mathbb{R}^3 $-valued $p$-th power integrable functions on $\mathbb{T}^3 $, which is equipped with the norm
$$
      |u|_{{L}^{p}}{} := \Bigg( \int_{\mathbb{T}^3 } |u(\xi){|}^{p} \, d\xi \Bigg)^{\frac{1}{p}} , \qquad u \in {L}^{p}.
$$
In particular,  ${L}^{2}$ is a Hilbert space with the inner product given by
$$
  (u,v) := \int_{\mathbb{T}^3 } u(\xi) \cdot v (\xi)^* \, d\xi, \qquad u,v \in {L}^{2},
$$
where $v^*$ denotes the complex conjugate of $v$ and $u\cdot v$ denotes the standard Euclidean scalar product in $\mathbb{R}^3$. Moreover, we denote by ${L}^{\infty }:= {L}^{\infty }(\mathbb{T}^3 )^3$  the Banach space of Lebesgue measurable essentially bounded $\mathbb{R}^3 $-valued functions on $\mathbb{T}^3 $ with  the norm given by
$$
     |u|_{L^\infty }{}:= \mbox{\rm esssup} \,\big\{  |u(\xi)| , \, \, \xi \in \mathbb{T}^3  \big\}  , \qquad u \in {L}^{\infty } .
$$

Note that we work with the periodic boundary condition, one can represent  the
function in Fourier series as
$$u(\xi)=\sum_{k\in\mathbb{Z}_0^3}\hat{u}_ke^{ik\cdot{\xi}},~~\mbox{with}~\hat{u}_k\in\mathbb{C}^3,~\hat{u}_{-k}=\hat{u}^*_k~~\text{for every}~k,$$
 where
$\mathbb{Z}^3_0=\mathbb{Z}^3\setminus\{0\}$ and $\hat{u}^*_k$ denotes the complex conjugate of $\hat{u}_k$.

For $s\in\mathbb{R}$,
the Sobolev space $H^s:=H^s(\mathbb{T}^3)^3$ with vanishing spatial average can be represented as
$$H^s=\left\{u=\sum_{k\in\mathbb{Z}_0^3}\hat{u}_ke^{ik\cdot{\xi}}:\hat{u}_{-k}=\hat{u}^*_k,~\int_{\mathbb{T}^3}u(\xi)d\xi=0,~|u|_s^2:=\sum_{k\in\mathbb{Z}_0^3}|k|^{2s}|\hat{u}_k|^2<\infty\right\}.$$
In the Fourier space, the divergence free condition can be represented as
$$\hat{u}_k\cdot{k}=0~~\mbox{for every}~k.$$
We define the spaces for the divergence free velocity vectors
\begin{equation*}\mathcal{V}:=\big\{u\in\mathcal{C}_c^{\infty}(\mathbb{T}^3)^3:\hat{u}_k\cdot{k}=0 ~~\mbox{for
every}~k \big\}\end{equation*}
and
\begin{equation*}
\mathbb{H}^s:=\big\{ u\in{H}^s:\hat{u}_k\cdot{k}=0 ~~\mbox{for
every}~k \big\}
\end{equation*} which is a Hilbert space with scalar product
$$(u,v)_{s}:=\sum_{k\in\mathbb{Z}_0^3}|k|^{2s}\hat{u}_k\cdot{\hat{v}}_{-k}.$$

We remain denote by $|u|_{L^2}$ the norm in space
$\mathbb{H}^0$  and inner product
$(u,v)=\sum_k \hat{u}_k\cdot\hat{v}_{-k} $. Therefore,
$\mathbb{H}^0$ is the Hilbert space consisting of all $L^2$-integrable $\mathbb{R}^3$-valued functions
on $\mathbb{T}^3$  which are
divergence free and have zero mean.  We
identify the continuous dual space of $\mathbb{H}^s$ as
$\mathbb{H}^{-s}$ with the dual pairing between $\mathbb{H}^s$ and $\mathbb{H}^{-s}$ by  $\langle{u},v\rangle$.

 We define the nonlocal operator
$\Lambda^{s}$  as
$$\Lambda^s u:=\sum_{k\in\mathbb{Z}_0^3}|k|^{s}\hat{u}_ke^{ik\cdot{\xi}},$$
then it is clear that $\Lambda^2=-\Delta$. For sake of simplicity, we denote $\Lambda=\Lambda^1$. Notice that $\Lambda^s$ maps $H^r$ onto
$H^{r-s}$ and
$$|u|_s^2=\sum_{k\in\mathbb{Z}_0^3}|k|^{2s}|\hat{u}_k|^2=|\Lambda^s{u}|_{L^2}^2.$$

Denote by $\mathcal{P}$ the Leray-Helmholtz projection from $H^\beta$
to $\mathbb{H}^\beta$. It's well-known that the operators $\mathcal{P}$
and $\Lambda^s$ are commutative.

\vspace{1mm}
 We have the following Sobolev
embedding theorem (cf.~\cite{Ste70}).
\begin{lemma}\label{lemso}
If $0\leq{}s<\frac{3}{2}$ and
$\frac{1}{p}+\frac{s}{3}=\frac{1}{2}$, then $H^s\subset{L}^p$.
Moreover,
$$|f|_{L^p}\lesssim_{s,p} |{f}|_{s}.$$
\end{lemma}

Define the commutator
$$[\Lambda^s,f]g=\Lambda^s(fg)-f\Lambda^sg.$$

The following commutator estimate is useful for later use (cf. \cite{KP88}).
\begin{lemma}
$($Commutator estimate$)$ Suppose that $s>0$, $p,p_2,p_3\in(1,\infty)$ and $p_1,p_4\in(1,\infty]$ satisfy
\begin{equation*}
\frac{1}{p}\geq\frac{1}{p_1}+\frac{1}{p_2},~~~\frac{1}{p}\geq\frac{1}{p_3}+\frac{1}{p_4}.
\end{equation*}
Then we have
\begin{equation}\label{escom}
|[\Lambda^s,f]g|_{L^p}\lesssim\big(|\nabla{f}|_{L^{p_1}}|\Lambda^{s-1}{g}|_{L^{p_2}}+|\Lambda^s{f}|_{L^{p_3}}|g|_{L^{p_4}}\big).
\end{equation}
\end{lemma}

We also need to use the following interpolation inequality in 3D. 
\begin{lemma}\label{interp1}
$($Interpolation inequality$)$ There exists a constant
$\delta_0>0$
 such that
\begin{equation}\label{inteq2}
|u|_{1}^2\leq\delta_0|u|_{1/2}|u|_{3/2}.
\end{equation}
\end{lemma}

For a Banach space $(\mathbb{X},|\cdot|_{\mathbb{X}})$, we denote by $\mathbb{C}_T(\mathbb{X}):=C([0,T];\mathbb{X})$ the space of all continuous functions from $[0,T]$ to $\mathbb{X}$, which is a Banach space equipped with  the  norm
$$\|u\|_{\mathbb{C}_T(\mathbb{X})}:=\sup_{t\in[0,T]}|u(t)|_{\mathbb{X}},~u\in \mathbb{C}_T(\mathbb{X}).$$
Let $\mathscr{B}_b(\mathbb{X})$ (resp.~$C_b(\mathbb{X})$) be the space of all bounded and Borel measurable (resp.~continuous) real functions on $\mathbb{X}$.
The space $\text{Lip}_b(\mathbb{X})$ consists of  all the bounded and Lipschitz continuous real functions, and its norm is given by
$$\|\varphi\|_{L}:=\|\varphi\|_{\infty}+C_{\text{Lip}}^{\mathbb{X}},~\varphi\in \text{Lip}_b(\mathbb{X}),$$
where $\|\cdot\|_{\infty}$ is the supremum norm and  $C_{\text{Lip}}^{\mathbb{X}}$ is the Lipschitz constant of $\varphi$ on $\mathbb{X}$. Furthermore, let $(\mathbb{X},(\cdot,\cdot)_{\mathbb{X}})$, $(\mathbb{Y},(\cdot,\cdot)_{\mathbb{Y}})$ be two separable Hilbert spaces. The space of all Hilbert-Schmidt operators from $\mathbb{X}$ to $\mathbb{Y}$ is denoted by $\mathcal{L}_2(\mathbb{X};\mathbb{Y})$ equipped with the  Hilbert-Schmidt norm $\|\cdot\|_{\mathcal{L}_2(\mathbb{X};\mathbb{Y})}$.


\subsection{Bilinear operators}

Let us denote
$\mathbb{H}:=\mathbb{H}^0,~\mathbb{V}:=\mathbb{H}^1.$
Consider the trilinear form $b(\cdot,\cdot,\cdot):\mathbb{V}\times\mathbb{V}\times \mathbb{V}\to \mathbb{R}$ by  $$b(u,v,w):=\int_{\mathbb{T}^3}(u\cdot\nabla)v\cdot wd\xi=\sum_{i,i=1}^3\int_{\mathbb{T}^3}u_i\frac{\partial v_j}{\partial x_i}w_j d\xi.$$
By H\"{o}lder's inequality and Lemma \ref{lemso}, we have
\begin{equation}\label{cb1}
|b(u,v,w)|\leq |u|_{L^4}|w|_{L^4}|v|_1
\lesssim  |u|_{1}|v|_{1}|w|_{1},~u,v,w\in \mathbb{V}.
\end{equation}
Hence, $b$ is continuous on $\mathbb{V}$.  In addition, we define a bilinear operator $B$ by $$B(u,v):=b(u,v,\cdot),$$ then in view of (\ref{cb1}) it follows that $B(u,v)\in \mathbb{V}^{*}$ for all $u,v\in \mathbb{V}$ and  that the following estimate holds
$$|B(u,v)|_{-1}\lesssim|u|_1|v|_1,~u,v\in \mathbb{V}.$$
Thus, the mapping $B: \mathbb{V}\times  \mathbb{V}\to \mathbb{V}^{*}$ is bilinear and continuous. For sake of simplicity, we write $$B(u):=B(u,u).$$

We recall an important property of the trilinear form $b$ (cf.\cite{Tem95}).
\begin{lemma}
For any $u,v,w\in \mathbb{V}$,
\begin{equation}\label{cb2}
b(u,v,w)=-b(u,w,v),~b(u,v,v)=0.
\end{equation}
Moreover, $(\ref{cb2})$ holds more generally for any $u,v,w$ giving a meaning to
the trilinear forms, as stated precisely in the following:
\begin{equation}
\langle{B}(u,v),w\rangle\lesssim|u|_{m_1}|v|_{m_2+1}|w|_{m_3},\label{2.2}
\end{equation}
with the nonnegative parameters fulfilling
$$m_1+m_2+m_3\geq\frac{3}{2}~~~~~\text{if}~m_i\neq\frac{3}{2}~\text{for~any~i}$$
or$$m_1+m_2+m_3>\frac{3}{2}~~~~~\text{if}~m_i=\frac{3}{2}~\text{for~some~i}.$$
\end{lemma}

We also need to use the following estimates concerning the mapping $B$, whose proof is placed in Appendix \ref{subsecpr}.
\begin{lemma}\label{Bes12}
There exist positive constant $\delta_1$ and $\delta_2$ such that
\begin{eqnarray}\label{es7}
(B(u),\Lambda u)
\leq\frac{1}{2}| u|_{3/2}^2+\delta_1|u|_{1}^4|u|_{1/2}^2 , ~ ~u\in\mathbb{H}^{3/2}
\end{eqnarray}
and
\begin{eqnarray}\label{es29}
(B(u),\Lambda^2 u)
\leq \frac{1}{2}| u|_{2}^2+\delta_2|u|_{1}^6, ~ ~u\in\mathbb{H}^{2}.
\end{eqnarray}
\end{lemma}


\section{Well-posedness}\label{mainR}
In this section, we establish the global well-posedness of the stochastic forced 3D Navier-Stokes equations. In subsection \ref{sub2.1}, we present the main assumptions on the random noise and state the main results. In subsection \ref{sub2.2}, we make some necessary preparations including several tightness criterions.  In subsections \ref{sub2.3}-\ref{sub2.5}, we construct the Faedo-Galerkin approximation of the stochastic forced 3D Navier-Stokes equations and passage to its limit.  In subsections \ref{sub2.6}-\ref{sub2.7}, we prove the main theorems.
\subsection{Mathematical framework and main results}\label{sub2.1}
Taking into account the Leray-Helmholtz projection $\mathcal{P}$,  the system (\ref{sNS}) can be transferred to  an abstract formulation as follows
\begin{equation}\label{sns00}
\begin{cases}
&\!\!\!\!\!\!du(t)+[\nu\mathcal{A} u(t) +B( u(t))] dt= \sum_{i=1}^{\infty}\big[\mathcal{P}((\zeta_i\cdot\nabla)u(t))d\beta_i(t)+ g_i(t,u(t)) d\hat{\beta}_i(t)\big], \\
&\!\!\!\!\!\!u(0) = x,
\end{cases}
\end{equation}
where $(\beta_i)_{i\geq 1}$ and $(\hat{\beta}_i)_{i\geq 1}$ are sequences of independent standard Brownian motions on a complete filtered probability space $\left(\Omega,\mathscr{F},(\mathscr{F}_t)_{t\geq0},\mathbf{P}\right)$.
 Without loss of generality, throughout the work, we assume $\nu\equiv 1$.

For $i\in\mathbb{N}$, we denote
\begin{equation}\label{sigma1}
\sigma_i(u):= \mathcal{P}((\zeta_i\cdot\nabla)u).
\end{equation}
Then, for any $h\in l^2$  we define
$$\sigma(u)h:=\sum_{i=1}^{\infty}\sigma_i(u)h_i,~g(t,u)h:=\sum_{i=1}^{\infty} g_i(t,u)h_i.$$

We recall that the sequence of independent standard Brownian motions $(\beta_i)_{i\geq 1}$ (resp.~$(\hat{\beta}_i)_{i\geq 1}$) induces uniquely an $l^2$-valued cylindrical Brownian motion  $\mathcal{W}$ (resp.~$\hat{\mathcal{W}}$). In this case, we can rewrite Eq.~(\ref{sns00}) as
\begin{equation}\label{sns1}
\begin{cases}
&\!\!\!\!\!\!du(t)+[\mathcal{A} u(t) +B( u(t))] dt =\sigma(u(t))d\mathcal{W}(t) +    g(t,u(t)) d\hat{\mathcal{W}}(t) , \\
&\!\!\!\!\!\!u(0) = x,
\end{cases}
\end{equation}
where $\mathcal{W}$ and $\hat{\mathcal{W}}$ are independent $l^2$-valued cylindrical Brownian motions.

We first recall the definitions of (probabilistically) weak and strong solutions to Eq.~(\ref{sns1}).

\begin{definition}\label{dew} $($Weak solutions$)$ We say that Eq.~$(\ref{sns1})$ has a weak solution  with initial value $x\in \mathbb{H}^{1/2}$, if there exist a stochastic basis $(\Omega,\mathscr{F},(\mathscr{F}_t)_{t\geq 0},\mathbf{P})$, and  an $(\mathscr{F}_t)$-adapted process $u$  and independent $l^2$-valued cylindrical Brownian motions  $\mathcal{W}$ and $\hat{\mathcal{W}}$  such  that the following holds:

\vspace{2mm}
$(i)$ $u\in \mathbb{C}_T(\mathbb{H}^{1/2})\cap L^2([0,T];\mathbb{H}^{3/2})$\,\,\, $\mathbf{P}$-a.s.;

\vspace{2mm}
$(ii)$ it holds that  $\mathbf{P}$\text{-a.s.}
\begin{eqnarray}\label{de2}
\!\!\!\!\!\!\!\!&&(u(t),\phi)+\int_0^t (B(u(s)),\phi) ds+\int_0^t (\nabla u(s),\nabla\phi)ds
\nonumber \\
\!\!\!\!\!\!\!\!&&=
(x,\phi)+\int_0^t (\sigma(u(s))d \mathcal{W}(s),\phi)+\int_0^t (g(s,u(s))d \hat{\mathcal{W}}(s),\phi)
\end{eqnarray}
for all $t\in[0,T]$ and all $\phi\in \mathcal{V}$.
\end{definition}

\begin{definition}\label{des}
$($Strong solutions$)$ We call that there exists a $($probabilistically$)$ strong solution to stochastic 3D Navier-Stokes equations $(\ref{sns1})$ with initial value $x\in \mathbb{H}^{1/2}$ if for every probability space $(\Omega,\mathscr{F},(\mathscr{F}_t)_{t\geq 0},\mathbf{P})$ with independent $l^2$-valued cylindrical Brownian motions  $\mathcal{W}$ and $\hat{\mathcal{W}}$, there exists
an $(\mathscr{F}_t)$-adapted process $u$ such that properties $(i)$-$(iii)$ in Definition \ref{dew} hold.
\end{definition}

%

For $s\in[0,1]$, we define

\begin{equation}\label{Ms}
\mathcal{N}_s:=\sum_{i=1}^{\infty}|\Lambda^s \zeta_i|_{L^{\infty}}^2.
\end{equation}

The conditions on the vector field $\zeta$ are provided as follows.
\begin{hypothesis}\label{h2}

The measurable mapping $\zeta_i:\mathbb{T}^3\to \mathbb{R}^3$, $i\in\mathbb{N}$, is of $C^{1}$-class.   Moreover, the following assumptions hold.

\vspace{1mm}
\begin{enumerate}
\item [$(\mathbf{H}_{\sigma}^1)$] $\mathcal{N}_1<\infty.$

\vspace{2mm}
\item [$(\mathbf{H}_{\sigma}^2)$]
$\mathcal{N}_0< 1/8.$

\end{enumerate}
\end{hypothesis}

\begin{remark}
Here, the factor $1/8$ in $(\mathbf{H}_{\sigma}^2)$  is not optimal, which is related to the viscosity constant $\nu$ assumed to be $1$. That is to say, the first-order term appearing in the transport noise will be absorbed
by the Laplacian.
\end{remark}

Recall the positive constants $\delta_0$, $\delta_1$ and $\delta_2$ given in Lemma \ref{interp1} and Lemma \ref{Bes12}. The conditions on the mapping $g$ are provided as follows.
\begin{hypothesis}\label{h3}

For any $t\in[0,T]$, we suppose that $g(t,0)=0$ and that the following  assumptions  hold for all $u,v\in \mathbb{H}^{3/2}$.

\vspace{1mm}
\begin{enumerate}

\item [$(\mathbf{H}_{g}^1)$] There are some constants $\beta\geq 2$ and $C>0$ such that
\begin{eqnarray*}
\!\!\!\!\!\!\!\!&&\|g(t,u)-g(t,v)\|_{\mathcal{L}_2(l^2;\mathbb{H}^{1/2})}^2
\nonumber \\
\leq \!\!\!\!\!\!\!\!&&
(C+\rho_1(u)+\rho_2(v))|u-v|^{2}_{1/2}+(C+\eta_1(u)+\eta_2(v))|u-v|^{2}_{1},
\end{eqnarray*}
where $\rho_1,\rho_2,\eta_1,\eta_2: \mathbb{H}^{3/2}\to [0,\infty)$ are  measurable functions satisfying
\begin{eqnarray}
\!\!\!\!\!\!\!\!&&\rho_1(u)+\rho_2(u)\lesssim (1+|u|_{3/2}^{2})(1+|u|_{1/2}^{\beta}),~u\in \mathbb{H}^{3/2},\label{es1}
 \\
\!\!\!\!\!\!\!\!&&\eta_1(u)+\eta_2(u)\lesssim (1+|u|_{3/2})(1+|u|_{1/2}^{\beta}),~u\in \mathbb{H}^{3/2}.\label{es0}
\end{eqnarray}

\item [$(\mathbf{H}_{g}^2)$]   There exist some constants  $C>0$, $\gamma\in(1,2)$ such that
\begin{equation*}
\kappa_1(|u|_{1}^{4}+1)|u|_{1/2}^4+\|g(t,u)\|_{\mathcal{L}_2(l^2;\mathbb{H}^{1/2})}^2|u|_{1/2}^2
\leq C|u|_{1/2}^{\gamma+2}+\gamma\|(g(t,u)\cdot,u)_{1/2}\|_{\mathcal{L}_2(l^2;\mathbb{R})}^2,
\end{equation*}
where $\kappa_1\geq\max\big\{\delta_1,\frac{4\delta_0^2\mathcal{N}_{1/2}^2}{1-8\mathcal{N}_0}\big\}$,  $\|(g(t,u)\cdot,u)_{1/2}\|_{\mathcal{L}_2(l^2;\mathbb{R})}^2=\sum_{i=1}^{\infty}|(g(t,u)h_k,u)_{1/2}|^2$, $\{h_1,h_2,\cdots\}$ is a complete orthonormal basis on  $l^2$.

\vspace{1mm}
\item [$(\mathbf{H}_{g}^3)$]   There exist  some constants  $C>0$ and  $\kappa_2\geq2\delta_2$    such that
\begin{equation*}
\kappa_2|u|_{1}^{6}+\|g(t,u)\|_{\mathcal{L}_2(l^2;\mathbb{H}^{1})}^2
\leq C(1+|u|_{1}^2)+\frac{2\|(g(t,u)\cdot,u)_{1}\|_{\mathcal{L}_2(l^2;\mathbb{R})}^2}{1+|u|_{1}^2}.
\end{equation*}

\item [$(\mathbf{H}_{g}^4)$] There are some constants $\alpha,\beta\geq 2$ such that
\begin{equation}\label{cong1}
\|g(t,u)\|_{\mathcal{L}_2(l^2;\mathbb{H}^{1/2})}^2\lesssim (1+|u|_{3/2}^{2})(1+|u|_{1/2}^{\beta}),
\end{equation}
\begin{equation}\label{cong4}
\|g(t,u)\|_{\mathcal{L}_2(l^2;\mathbb{H}^{1})}^2\lesssim (1+|u|_{2}^{2})(1+|u|_{1}^{\beta}),
\end{equation}
and
\begin{equation}\label{cong2}
\|g(t,u)\|_{\mathcal{L}_2(l^2;\mathbb{H}^{-1/2})}^2\lesssim (1+|u|_{1}^{\alpha})(1+|u|_{1/2}^{\beta}).
\end{equation}

\end{enumerate}

\end{hypothesis}

In the following remark, we present a motivational example for random noise that satisfies Hypothesis \ref{h3}.

\begin{remark}\label{reb2}
(i) In the field of  partial differential equations, the integral-type nonlocal forcing plays an important role since it can be used to characterize the spatial dependence for systems introduced in the engineering,
biology and fluid mechanics, which have been extensively studied  (cf.~the monograph \cite{KS18}). In particular, the now well-known Burgers equation is in fact  a simplified version of the {\it{nonlocal}} Burgers system, in which a nonlocal turbulent forcing $\Phi(u)=\int_{\mathcal{O}}|u(t,\xi)|^2d\xi$ is involved  to simulate flow in a channel. For related introduction for the original {\it{nonlocal}} Burgers system, we refer to the seminal work  \cite{Bur39}, as well as some subsequent follow-up studies \cite{HO,DKL,Bur48}.

On the other hand, in the field of fluid mechanics, the nonlocal viscous forcing have also attracted widespread attention. In the well-known work \cite{Lad67}, Ladyzenskaja considered the Navier-Stokes equations with the following nonlocal viscosity constant
$$\nu(u):=\nu\Big[1+\int_{\mathcal{O}}|\nabla u(t,\xi)|^2d\xi\Big],$$
which is related to the total dissipation energy of Newtonian fluid.  As noticed in \cite{CR05}, such nonlocal viscosity enables us to view the model
as a particular asymptotic limit  of certain nonisothermal flows with very high thermal conductivity in the
fluid.

Motivated by the deterministic case,  we apply our main results to the nonlocal stochastic  forcing of the following form
\begin{equation}\label{exaB1}
\int_0^tg(s, u(s,x)) d\hat{\mathcal{W}}(s):=\sum_{i=1}^{\infty}\int_0^t\alpha_i\Big[1+\int_{\mathbb{T}^3}|\nabla u(s,\xi)|^2d\xi\Big]u(s)  d\hat{\beta}_i(s),
\end{equation}
where the constant $\varrho:=\sum_{i=1}^{\infty}\alpha_i^2\geq\frac{\kappa_1}{\gamma-1}\vee\kappa_2$. The  nonlocal stochastic forcing $(\ref{exaB1})$ can be used to characterize random perturbations with spatial dependence, where the intensity of the noise is related to the total dissipation of energy. It worth noting that the gradient-integral-type nonlocal forcing have extensive applications in many fields, see e.g.~the non-degenerate
 Kirchhoff type forcing in the study of the wave equations (cf.~\cite{K83} and reference therein) and the Woinowsky-Krieger nonlocal  forcing introduced in \cite[Chap 6]{KS18}. The reader can also refer to \cite{WK50,APS,G13} for additional applications.

\vspace{1mm}
(ii) For reader's convenience,  we sketch the proof for which the nonlocal stochastic forcing $(\ref{exaB1})$ satisfies the assumptions  $(\mathbf{H}_{g}^2)$ and $(\mathbf{H}_{g}^3)$, whereas  the remaining  assumptions in Hypothesis \ref{h3} are more straightforward.
\begin{proof}
Recall that $\{h_1,h_2,\cdots\}$ is a complete orthonormal basis on  $l^2$.
Note that
\begin{equation*}
   g(t,u)y=\sum_{i=1}^{\infty}\alpha_i\big[1+|u|_{1}^2\big] u( y,h_i)_{l^2},~y\in l^2.
\end{equation*}
Therefore, it follows that
\begin{eqnarray*}
\!\!\!\!\!\!\!\!&&\|(g(t,u)\cdot,u)_{1/2}\|_{\mathcal{L}_2(l^2;\mathbb{R})}^2
\\
=\!\!\!\!\!\!\!\!&&\|g(t,u)^*u\|_{l^2}^2  \\
=\!\!\!\!\!\!\!\!&&( u,g(t,u)g(t,u)^*u)_{1/2}   \\
=\!\!\!\!\!\!\!\!&&\sum_{i=1}^{\infty}\alpha_i\big[1+|u|_{1}^{2}\big]( u, u)_{1/2}( g(t,u)^*u,h_i)_{l^2}   \\
=\!\!\!\!\!\!\!\!&&\sum_{i=1}^{\infty}\alpha_i\big[1+|u|_{1}^{2}\big]|u|_{1/2}^{2}( u,\sum_{j=1}^{\infty}\alpha_j\big[1+|u|_{1}^{2}\big] u( h_j,h_i)_{l^2})_{1/2}   \\
=\!\!\!\!\!\!\!\!&&\varrho\big[1+|u|_{1}^{2}\big]^2|u|_{1/2}^{4}.
 \end{eqnarray*}
 Similarly, we can also get
 \begin{equation*}
 \|(g(t,u)\cdot,u)_{1}\|_{\mathcal{L}_2(l^2;\mathbb{R})}^2=\varrho\big[1+|u|_{1}^{2}\big]^2|u|_{1}^{4}.
 \end{equation*}

As for the assumption  $(\mathbf{H}_{g}^2)$, we know that there exists a constant $\gamma\in(1,2)$ such that for $\varrho\geq\frac{\kappa_1}{\gamma-1}$,
\begin{eqnarray*}
\!\!\!\!\!\!\!\!&&\kappa_1\big[1+|u|_{1}^{4}\big]|u|_{1/2}^4+\|g(t,u)\|_{\mathcal{L}_2(l^2;\mathbb{H}^{1/2})}^2|u|_{1/2}^2\\
\leq\!\!\!\!\!\!\!\!&&
 \kappa_1\big[1+|u|_{1}^{2}\big]^2|u|_{1/2}^4+\varrho\big[1+|u|_{1}^{2}\big]^2|u|_{1/2}^4
 \\
\leq\!\!\!\!\!\!\!\!&&\gamma\varrho\big[1+|u|_{1}^{2}\big]^2|u|_{1/2}^4 =\gamma\|(g(t,u)\cdot,u)_{1/2}\|_{\mathcal{L}_2(l^2;\mathbb{R})}^2      .
 \end{eqnarray*}
On the other hand, as for the assumption  $(\mathbf{H}_{g}^3)$, we can deduce that for $\varrho\geq\kappa_2$
\begin{eqnarray*}
\!\!\!\!\!\!\!\!&&\kappa_2|u|_{1}^{6}+\|g(t,u)\|_{\mathcal{L}_2(l^2;\mathbb{H}^{1})}^2\\
\leq\!\!\!\!\!\!\!\!&&
 \kappa_2|u|_{1}^{6}+\varrho\big[1+|u|_{1}^{2}\big]^2|u|_{1}^2
 \\
 \leq\!\!\!\!\!\!\!\!&&
  (\kappa_2\vee\varrho)\frac{\big[1+|u|_{1}^{2}\big]^2|u|_{1}^2+\big[1+|u|_{1}^{2}\big]^2|u|_{1}^4}{1+|u|_{1}^2}
  \\
 \leq\!\!\!\!\!\!\!\!&&
  \varrho\frac{\big[1+|u|_{1}^{2}\big]^2(\frac{1}{4}+|u|_{1}^4)+\big[1+|u|_{1}^{2}\big]^2|u|_{1}^4}{1+|u|_{1}^2}
  \\
\leq\!\!\!\!\!\!\!\!&&\frac{\varrho}{4}(1+|u|_{1}^{2})+2 \frac{\|(g(t,u)\cdot,u)_{1}\|_{\mathcal{L}_2(l^2;\mathbb{R})}^2}{1+|u|_{1}^2},
 \end{eqnarray*}
where we used Young's inequality in the third step.
\end{proof}

\end{remark}

Now, we state the first main result concerning the global existence and uniqueness of solutions  to stochastic 3D Navier-Stokes equations (\ref{sns1}).
\begin{theorem}\label{thg}
Suppose that Hypothesis \ref{h2}-\ref{h3} hold. For any initial value $x\in\mathbb{H}^{1/2}$, Eq.~$(\ref{sns1})$ has a unique strong solution $u(t)$.
Moreover,  the following energy moment estimates hold
\begin{equation}\label{esap}
\sup_{t\in[0,T]}\mathbf{E}|u(t)|_{1/2}^{2-\gamma}+\mathbf{E}\int_0^T|u(t)|_{3/2}^{2-\gamma}dt<\infty,
\end{equation}
where the constant $\gamma\in(1,2)$ is the same as in $(\mathbf{H}_{g}^2)$.
\end{theorem}

\begin{remark}
 (i) From the energy estimates $(\ref{esap})$ one observes that, unlike the existing results in the $L^2$-setting (cf.~e.g.~\cite{BM13}), the finite second moments are unavailable here for global solutions in the $\mathbb{H}^{1/2}$-space. This originates from the Lyapunov  approach we employ to capture the regularization effect of random noise.

 (ii) This seems to be the first result in the literature that establishes the global existence and uniqueness of solutions for stochastic forced 3D Navier-Stokes equations in the  critical space under general initial conditions. The contribution also includes developing a bootstrap type estimate  to prove the global solvability of the 3D Navier-Stokes equations perturbed by nonlocal stochastic forcing in the $\mathbb{H}^{1/2}$-space.
\end{remark}


For any $\varphi\in \mathscr{B}_b(\mathbb{H}^{1/2})$, $t\geq 0$, we define a function $\mathcal{T}_t\varphi:\mathbb{H}^{1/2}\to \mathbb{R}$ by
\begin{equation}\label{semig}
\mathcal{T}_t\varphi(x):=\mathbf{E}\varphi(u(t,x)),~x\in \mathbb{H}^{1/2},
\end{equation}
where  $u(t,x)$ is the  solution to (\ref{sns1}) with initial data $x$.

\vspace{1mm}
The following result establishes the continuous dependence on initial values in probability.
\begin{theorem}\label{thF}
Suppose that  Hypothesis \ref{h2}-\ref{h3} hold. Let $\{x_n\}_{n\in\mathbb{N}}$ and $x$ be a sequence with $|x_n-x|_{1/2}\to 0$.
Then as $n\to\infty$,
\begin{equation}\label{apri3}
\|u(\cdot,x_n)-u(\cdot,x)\|_{\mathbb{C}_T(\mathbb{H}^{1/2})}+\|u(\cdot,x_n)-u(\cdot,x)\|_{L^2([0,T];\mathbb{H}^{3/2})} \to 0~~\text{in probability}.
\end{equation}
In particular, if $g$ is independent of $t$ $($i.e.,~ $g(t,u)=g(u)$$)$, Eq.~$(\ref{sns1})$ defines a Feller Markov process, that is, $\mathcal{T}_t:C_b(\mathbb{H}^{1/2})\to C_b(\mathbb{H}^{1/2}) $ and
\begin{equation}\label{markov}
\mathbf{E}[\varphi(u(t+s,x))|\mathscr{F}_t]=(\mathcal{T}_s\varphi)(u(t,x)),~\text{for any}~\varphi\in C_b(\mathbb{H}^{1/2}),x\in\mathbb{H}^{1/2},t,s>0.
\end{equation}
Besides, the semigroup property $\mathcal{T}_{t+s}=\mathcal{T}_{t}\circ \mathcal{T}_{s}$ holds.

\end{theorem}

\begin{remark}
Since we have established the Feller property of the transition semigroup $\mathcal{T}_{t}$, our next goal is to investigate the existence and uniqueness of invariant measures for the stochastic forced 3D Navier-Stokes equations, which will be done in Section \ref{sec4}.

\end{remark}

\subsection{Stochastic compactness criterions}\label{sub2.2}
In this subsection,  we introduce several stochastic compactness criterions and a Jakubowski's beautiful generalization of the Skorokhod theorem, which is pivotal for proving the convergence of sequences on the nonmetric space.

It is clear that the embeddings $\mathbb{H}^{3/2}\subset\mathbb{H}^{1}\subset \mathbb{H}^{1/2}\subset\mathbb{H}$ are compact and dense.
Denote by $\mathbb{H}^*$   the dual space of $\mathbb{H}$. We have the following embedding of spaces
\begin{equation}\label{emb}
\mathbb{H}^{3/2}\subset\mathbb{H}^{1}\subset \mathbb{H}^{1/2}\subset\mathbb{H}(\simeq \mathbb{H}^*)\subset\mathbb{H}^{-1/2}\subset\mathbb{H}^{-1}\subset\mathbb{H}^{-3/2},
\end{equation}
where   $\mathbb{H}^*$ is identified
with $\mathbb{H}$ via  the Riesz isomorphism.

Now we set the following space
\begin{eqnarray*}
\!\!\!\!\!\!\!\!&&\mathcal{X}_1:= L^2_w([0,T];\mathbb{H}^{3/2})\cap L^2([0,T];\mathbb{H}^{1/2})\cap L^{\infty}_{w^*}([0,T];\mathbb{H}^{1/2}),
\nonumber \\
\!\!\!\!\!\!\!\!&&\mathcal{X}_2:=C((0,T];\mathbb{H}^{-1/2}),
\end{eqnarray*}
where $L^2_w([0,T];\mathbb{H}^{3/2})$ is the space   $L^2([0,T];\mathbb{H}^{3/2})$ with  the weak topology, $L^{\infty}_{w^*}([0,T];\mathbb{H}^{1/2})$ is the space   $L^{\infty}([0,T];\mathbb{H}^{1/2})$ with  the weak-$*$ topology, and  $ C((0,T];\mathbb{H}^{-1/2} )$ is the space consisting of all continuous functions from $(0,T]$ to $\mathbb{H}^{-1/2}$ which is equipped with the complete metric
$$ d(u,v):=\sum_{k=1}^{\infty}\frac{1}{2^k}\Bigg(\sup_{t\in[\frac{1}{k},T]}|u(t)-v(t)|_{-1/2}\wedge 1 \Bigg).         $$
Here the intersection space $\mathcal{X}_1$ takes the intersection topology denoted by $\tau_{\mathcal{X}_1}$: the class of
open sets of $\mathcal{X}_1$  are generated by the sets of the form $\mathcal{O}_1\cap\mathcal{O}_2\cap\mathcal{O}_3$, where $\mathcal{O}_1$, $\mathcal{O}_2$ and $\mathcal{O}_3$ are open sets in   $L^2_w([0,T];\mathbb{H}^{3/2})$, $L^2([0,T];\mathbb{H}^{1/2})$, $L^{\infty}_{w^*}([0,T];\mathbb{H}^{1/2})$, respectively. The space $\mathcal{X}_1$ will considered w.r.t.~the Borel $\sigma$-algebra $\mathscr{B}(\tau_{\mathcal{X}_1})$.

The following lemma gives a tightness criterion for the set of measures induced on  $\mathcal{X}_1$.
\begin{lemma}\label{lemt}
Let $(u_n)_{n\in\mathbb{N}}$ be a sequence  such that
\begin{enumerate}[$(i)$]

\item  $$\lim_{R\rightarrow\infty}\sup_{n\in\mathbb{N}}\mathbf{P}\Big(\sup_{ t\in[0, T]}|u_n(t)|_{1/2}>R\Big)=0,$$

\item $$\lim_{R\rightarrow\infty}\sup_{n\in\mathbb{N}}\mathbf{P}\Bigg(\int_0^T|u_n(t)|_{3/2}^2dt>R\Bigg)=0,$$

\item For any $\varepsilon>0$,
\begin{equation}\label{tight2}
\lim_{\Delta\to 0^+}\sup_{n\in\mathbb{N}}\mathbf{P}\Bigg(\sup_{\delta\in[0,\Delta]}\int_0^{T-\delta}|u_n(t+\delta)-u_n(t)|_{1/2}^2dt>\varepsilon   \Bigg)=0.
\end{equation}
\end{enumerate}
Let $\mu_n$ be the law of $u_n$ on the Borel $\sigma$-algebra $\mathscr{B}(\tau_{\mathcal{X}_1})$. Then for every $\varepsilon>0$, there exists a compact subset $\mathcal{K}_{\varepsilon}$ of $\mathcal{X}_1$ such that
$$\sup_{n\in\mathbb{N}}\mu_n(\mathcal{K}_{\varepsilon}^c)\leq \varepsilon.$$
\end{lemma}

\begin{proof}
Due to (i), for any $\varepsilon>0$, there exists $R_1>0$ such that
\begin{equation}\label{comp1}
\sup_{n\in\mathbb{N}}\mathbf{P}\Big(\sup_{ t\in[0, T]}|u_n(t)|_{1/2}>R_1\Big)\leq \frac{\varepsilon}{3},
\end{equation}
then we denote
$$\mathcal{K}_{1}:=\Big\{u_n\in\mathcal{X}_1:\sup_{ t\in[0, T]}|u_n(t)|_{1/2}\leq R_1\Big\}.$$
Similarly, by (ii) for any $\varepsilon>0$, there exists $R_2>0$ such that
\begin{equation}\label{comp2}
\sup_{n\in\mathbb{N}}\mathbf{P}\Bigg(\int_0^T|u_n(t)|_{3/2}^2dt>R_2\Bigg)\leq \frac{\varepsilon}{3}.
\end{equation}
We denote
$$\mathcal{K}_{2}:=\Bigg\{u_n\in\mathcal{X}_1:\int_0^T|u_n(t)|_{3/2}^2dt\leq R_2\Bigg\}.$$
Due to (iii), for any $k\in\mathbb{N}$, there exists $\Delta_k>0$  such that
$$\sup_{n\in\mathbb{N}}\mathbf{P}\Bigg(\sup_{\delta\in[0,\Delta_k]}\int_0^{T-\delta}|u_n(t+\delta)-u_n(t)|_{1/2}^2dt>\frac{1}{k}   \Bigg)\leq \frac{1}{3}\cdot\frac{\varepsilon}{2^{k}}.$$
Denote
$$\Xi_k:=\Bigg\{u_n\in\mathcal{X}_1:\sup_{\delta\in[0,\Delta_k]}\int_0^{T-\delta}|u_n(t+\delta)-u_n(t)|_{1/2}^2dt\leq\frac{1}{k} \Bigg\}.$$

Finally, we denote by $\mathcal{K}_{\varepsilon}$ the closure of the set $\mathcal{K}_{1}\cap \mathcal{K}_{2}\cap \cap_{k=1}^{\infty}\Xi_k$ in $\mathcal{Z}_T$. Due to the compactness criterion presented in Lemma \ref{lemc} in Appendix, we know that $\mathcal{K}_{\varepsilon}$ is a compact set in $\mathcal{Z}_T$.  Then (\ref{comp1}) and (\ref{comp2}) imply
\begin{eqnarray*}
\sup_{n\in\mathbb{N}}\mathbf{P}\big(u_n\in\mathcal{K}_{\varepsilon}^c   \big)\leq\!\!\!\!\!\!\!\!&&\sup_{n\in\mathbb{N}}\mathbf{P}\big(u_n\in\mathcal{K}_{1}^c \big)+\sup_{n\in\mathbb{N}}\mathbf{P}\big(u_n\in\mathcal{K}_{2}^c\big)+\sum_{k=1}^{\infty}\sup_{n\in\mathbb{N}}\mathbf{P}\big(u_n\in\Xi_k^c\big)
\nonumber \\
\leq\!\!\!\!\!\!\!\!&&\frac{2\varepsilon}{3}+\frac{1}{3}\sum_{k=1}^{\infty}\frac{\varepsilon}{2^{k}}\leq    \varepsilon.
\end{eqnarray*}
The proof is complete.
\end{proof}

\vspace{1mm}

Let $k\in\mathbb{N}$. We recall the Aldous condition in $\mathbb{H}^{-1/2}$ with the time interval $[\frac{1}{k},T]$.
\begin{definition}\label{aldous}
A sequence $(u_n)_{n\geq 1}$ is said to satisfy the Aldous condition in $\mathbb{H}^{-1/2}$ with the time interval $[\frac{1}{k},T]$ iff for any $\varepsilon,\eta>0$, there exists $\Delta>0$ such that for every stopping time sequence $(\zeta_n)_{n\in\mathbb{N}}$ with $\zeta_n\in [\frac{1}{k},T]$ one has
\begin{equation*}
\sup_{n\in\mathbb{N}}\sup_{0\leq \theta\leq \Delta}\mathbf{P}(|u_n(\zeta_n+\theta)-u_n(\zeta_n)|_{-1/2}\geq \eta)\leq \varepsilon.
\end{equation*}
\end{definition}

We present the following tightness criterion for the set of measures induced on $ \mathcal{X}_2$.
\begin{lemma}\label{lemt1}
Let $(u_n)_{n\in\mathbb{N}}$ be a sequence  such that
\begin{enumerate}[$(i)$]

\item  For any $k\in\mathbb{N},\varepsilon>0$, there exists a constant $C_{k,\varepsilon}>0$ such that

$$\sup_{n\in\mathbb{N}}\mathbf{P}\Big(\sup_{ t\in[\frac{1}{k}, T]}|u_n(t)|_{1/2}>C_{k,\varepsilon}\Big)\leq \varepsilon,$$

\item For any $\varepsilon,\eta>0$, there exists $\Delta>0$ such that
\begin{equation}\label{stop5}
\sup_{n\in\mathbb{N}}\mathbf{P}\Bigg(\sum_{k=1}^{\infty}\frac{1}{2^k}\bigg(\sup_{|t-s|\leq \Delta, t,s\in[\frac{1}{k},T]}|u_n(t)-u_n(s)|_{-1/2}\wedge 1 \bigg) > \eta   \Bigg)\leq \varepsilon.
\end{equation}
\end{enumerate}
Let $\mu_n$ be the law of $u_n$. Then for every $\varepsilon>0$, there exists a compact subset $\mathcal{K}_{\varepsilon}$ of $ \mathcal{X}_2$ such that
$$\sup_{n\in\mathbb{N}}\mu_n(\mathcal{K}_{\varepsilon}^c)\leq \varepsilon.$$
\end{lemma}

\begin{proof}
Due to (i), for any $k\in\mathbb{N}$, there exists $C_k>0$ such that for any $\varepsilon>0$,
\begin{equation}\label{comp3}
\sup_{n\in\mathbb{N}}\mathbf{P}\Big(\sup_{ t\in[\frac{1}{k}, T]}|u_n(t)|_{1/2}>C_k\Big)\leq \frac{\varepsilon}{2^{k+1}},
\end{equation}
then we denote
$$\Gamma_{k}:=\Big\{u_n\in C([\frac{1}{k},T];\mathbb{H}^{-1/2} ):\sup_{ t\in[\frac{1}{k}, T]}|u_n(t)|_{1/2}\leq C_k\Big\}.$$
Due to (ii), for any $j\in\mathbb{N}$, there exists $\Delta_j>0$  such that
\begin{equation}\label{comp4}
\sup_{n\in\mathbb{N}}\mathbf{P}\Bigg(\sum_{k=1}^{\infty}\frac{1}{2^k}\bigg(\sup_{|t-s|\leq \Delta_j, t,s\in[\frac{1}{k},T]}|u_n(t)-u_n(s)|_{-1/2}\wedge 1 \bigg) > \frac{1}{j}  \Bigg)\leq \frac{\varepsilon}{2^{j+1}}.
\end{equation}
Denote
$$\tilde{\Xi}_j:=\Bigg\{u_n\in C([\frac{1}{k},T];\mathbb{H}^{-1/2}):\sum_{k=1}^{\infty}\frac{1}{2^k}\bigg(\sup_{|t-s|\leq \Delta_j, t,s\in[\frac{1}{k},T]}|u_n(t)-u_n(s)|_{-1/2}\wedge 1 \bigg) \leq \frac{1}{j} \Bigg\}.$$

Finally, we denote by $\mathcal{K}_{\varepsilon}$ the closure of the set $(\cap_{k=1}^{\infty}\Gamma_{k})\cap(\cap_{j=1}^{\infty}\tilde{\Xi}_j)$ in $ \mathcal{X}_2$. Due to the compactness criterion presented in Lemma \ref{lemc} in Appendix, we know that $\mathcal{K}_{\varepsilon}$ is a compact subset in $ \mathcal{X}_2$.  Then (\ref{comp3}) and (\ref{comp4}) imply
\begin{equation*}
\sup_{n\in\mathbb{N}}\mathbf{P}\big(u_n\in\mathcal{K}_{\varepsilon}^c   \big)\leq\sum_{k=1}^{\infty}\sup_{n\in\mathbb{N}}\mathbf{P}\big(u_n\in\Gamma_k^c\big)+\sum_{j=1}^{\infty}\sup_{n\in\mathbb{N}}\mathbf{P}\big(u_n\in\tilde{\Xi}_j^c\big)
\leq    \varepsilon.
\end{equation*}
The proof is complete.
\end{proof}

The path space $\mathcal{X}_1$ is not a Polish space and so our compactness argument is based on
 the Jakubowski-Skorokhod representation theorem instead of the classical Skorokhod
 representation theorem.
\begin{theorem}\label{sko1}$($Theorem A.1 in \cite{BO}$)$
Let $\mathbb{X}$ be a topological space such that there exists a sequence of continuous functions $f_m:\mathbb{X}\to \mathbb{R}$ that separates points of  $\mathbb{X}$. Let us denote by $\mathscr{S}$ the $\sigma$-algebra generated by the maps $f_m$. Then

$(i)$ every compact subset of $\mathbb{X}$ is metrizable;

\vspace{1mm}
$(ii)$ if $(\mu_m)$ is tight sequence of probability measures on $(\mathbb{X},\mathscr{S})$, then there exists a subsequence  denoted also by $(m)$, a probability space $(\Omega,\mathscr{F},\mathbf{P})$ with $\mathbb{X}$-valued Borel measurable variables $\xi_m$, $\xi$ such that $\mu_m$ is the law of $\xi_m$ and $\xi_m$ converges to $\xi$ almost surely on $\Omega$. Moreover, the law of $\xi$ is a Random measure.
\end{theorem}

Now we present the following result that is pivotal in applying the Jakubowski-Skorokhod theorem.
\begin{theorem}\label{th1}
The topological space $(\mathcal{X}_1,\mathscr{B}(\tau_{\mathcal{X}_1}))$ satisfies the assumptions in Theorem \ref{sko1}.

\end{theorem}
\begin{proof}
Since $L^{2}([0,T];\mathbb{H}^{1/2})$ is a Polish space, it is clear that the assumptions in Theorem \ref{sko1} hold. For the space $L^2_w([0,T];\mathbb{H}^{3/2})$, it suffices to put
$$f_m(u):=\int_0^T\langle v_m(t),u(t)\rangle dt\in\mathbb{R},~u\in L^2_w([0,T];\mathbb{H}^{3/2}),~m\in\mathbb{N},$$
where $\{v_m\}_{m\geq 1}$ is a dense subset of $L^2([0,T];\mathbb{H}^{-3/2})$, which separates the points of $L^2([0,T];\mathbb{H}^{3/2})$.

Furthermore, following similar arguments as in the proof of  \cite{HLL24}, we can deduce that $(\mathcal{X}_1,\mathscr{B}(\tau_{\mathcal{X}_1}))$  is a standard Borel space (cf.~Definition \ref{de5}).
Due   to Theorem \ref{th2}, the $\sigma$-algebra generated by the sequence of the above continuous functions
separating the points in $\mathcal{X}_1$ is exactly $\mathscr{B}(\tau_{\mathcal{X}_1})$. Hence, all the conditions in Theorem \ref{sko1} are satisfied for $\mathcal{X}_1$.
\end{proof}

We also prepare the following lemmas for the estimates of the coefficient $\sigma$, which are frequently used in the proof.
\begin{lemma}\label{lemsig}
For any $u\in \mathbb{H}^{3/2}$, we have
\begin{equation}\label{sig1}
\|\sigma(u)\|_{\mathcal{L}_2(l^2;\mathbb{H}^{1/2})}^2\leq \frac{1}{4}|u|_{3/2}^2+\frac{4\delta_0^2\mathcal{N}_{1/2}^2}{1-8\mathcal{N}_0}|u|_{1/2}^2,
\end{equation}
where $\delta_0$, $\mathcal{N}_{1/2}$ and $\mathcal{N}_{0}$ come from Lemma \ref{interp1} as well as $\eref{Ms}$.
Moreover, there exists a constant $C>0$ such that for any $u\in \mathbb{H}^{2}$,
\begin{equation}\label{sig2}
\|\sigma(u)\|_{\mathcal{L}_2(l^2;\mathbb{H}^{1})}^2\leq \frac{1}{4}|u|_{2}^2+C|u|_{1}^2.
\end{equation}
\end{lemma}

\begin{proof}
Let us consider a standard orthonormal basis $(h_i)_{i\in\mathbb{N}}$ in $l^2$. Then, by Hypothesis \ref{h2} we have
\begin{eqnarray*}
\!\!\!\!\!\!\!\!&&\|\sigma(u)\|_{\mathcal{L}_2(l^2;\mathbb{H}^{1/2})}^2
\nonumber \\
=\!\!\!\!\!\!\!\!&&\sum_{i=1}^{\infty}|\Lambda^{1/2}((\zeta_i\cdot \nabla) u)|_{L^2}^2
\nonumber \\
\leq\!\!\!\!\!\!\!\!&&2\sum_{i=1}^{\infty}|\Lambda^{1/2}\zeta_i\cdot \nabla u|_{L^2}^2+2\sum_{i=1}^{\infty}|\zeta_i\cdot \Lambda^{1/2}\nabla u|_{L^2}^2
\nonumber \\
\leq\!\!\!\!\!\!\!\!&&2\mathcal{N}_{1/2}|u|_{1}^2+2\mathcal{N}_{0} | u|_{3/2}^2
\nonumber \\
\leq\!\!\!\!\!\!\!\!&&2\delta_0\mathcal{N}_{1/2}|u|_{1/2}|u|_{3/2}+2\mathcal{N}_{0} | u|_{3/2}^2
\nonumber \\
\leq\!\!\!\!\!\!\!\!&&\frac{1}{4}|u|_{3/2}^2+\frac{4\delta_0^2\mathcal{N}_{1/2}^2}{1-8\mathcal{N}_0}|u|_{1/2}^2,
\end{eqnarray*}
where we have used the interpolation inequality \eref{inteq2} in the fourth step and Young inequality in the last inequality.
Applying similar argument, we can also infer that (\ref{sig2}) holds.
\end{proof}

\begin{lemma}
For any $u\in \mathbb{H}^{1}$, we have
\begin{equation}\label{sig3}
\|\sigma(u)\|_{\mathcal{L}_2(l^2;\mathbb{H}^{-1/2})}^2\lesssim_{\mathcal{N}_1} |u|_{1}^2.
\end{equation}
Moreover, for any $u\in \mathbb{H}$, we have
\begin{equation}\label{sig4}
\|\sigma(u)\|_{\mathcal{L}_2(l^2;\mathbb{H}^{-1})}^2\lesssim_{\mathcal{N}_1} |u|_{L^2}^2.
\end{equation}
\end{lemma}

\begin{proof}
We first prove (\ref{sig3}). For any $v\in \mathbb{H}^{1/2}$, we can use the estimate (\ref{2.2}) to get
\begin{equation*}
|\langle (\zeta_i\cdot \nabla)u,v\rangle|=|b(\zeta_i,u,v)|\lesssim|\zeta_i|_1|u|_{1}|v|_{1/2}.
\end{equation*}
Then it follows that
\begin{equation*}
|(\zeta_i\cdot \nabla)u|_{-1/2}\lesssim|\zeta_i|_1|u|_{1}.
\end{equation*}
Due to $(\mathbf{H}_{\sigma}^1)$, it follows that
\begin{equation*}
\|\sigma(u)\|_{\mathcal{L}_2(l^2;\mathbb{H}^{-1/2})}^2= \sum_{i=1}^{\infty}|(\zeta_i\cdot \nabla)u|_{-1/2}^2\lesssim|u|_{1}^2\sum_{i=1}^{\infty}|\zeta_i|_1^2\lesssim_{\mathcal{N}_1}|u|_{1}^2.
\end{equation*}

Now, we turn to prove (\ref{sig4}). For any $u\in\mathcal{V}$, $i\in\mathbb{N}$,
\begin{equation*}
\sum_{j=1}^3\frac{\partial}{\partial \xi_j}(\zeta_i^ju)=\sum_{j=1}^3\frac{\partial \zeta_i^j}{\partial \xi_j} u+\sum_{j=1}^3\zeta_i^j\frac{\partial u}{\partial \xi_j}=(\nabla\cdot\zeta_i) u+\sum_{j=1}^3\zeta_i^j\frac{\partial u}{\partial \xi_j}.
\end{equation*}
Making use of the integration by parts formula, it follows that for any $u,v\in\mathcal{V}$,
\begin{eqnarray*}
((\zeta_i\cdot\nabla)u,v)=\!\!\!\!\!\!\!\!&&\sum_{j=1}^3(\zeta_i^j\frac{\partial u}{\partial \xi_j}, v)=\sum_{j=1}^3(\frac{\partial}{\partial \xi_j}(\zeta_i^ju), v)-((\nabla\cdot\zeta_i) u,v)
\nonumber \\
=\!\!\!\!\!\!\!\!&&-\sum_{j=1}^3(\zeta_i^ju, \frac{\partial v}{\partial \xi_j} )-((\nabla\cdot\zeta_i) u,v).
\end{eqnarray*}
Then by H\"{o}lder's inequality we obtain
\begin{eqnarray*}
|((\zeta_i\cdot\nabla)u,v)|\leq\!\!\!\!\!\!\!\!&&\sum_{j=1}^3|(\zeta_i^ju, \frac{\partial v}{\partial \xi_j} )|+|((\nabla\cdot\zeta_i) u,v)|
\nonumber \\
\leq\!\!\!\!\!\!\!\!&&\sum_{j=1}^3\Big(|\zeta_i^j|_{L^{\infty}}|u|_{L^2}||\frac{\partial v}{\partial \xi_j}|_{L^2}\Big)+|\nabla\cdot\zeta_i|_{L^{\infty}}|u|_{L^2}|v|_{L^2}
\nonumber \\
\leq\!\!\!\!\!\!\!\!&&|\zeta_i|_{L^{\infty}}|u|_{L^2}|v|_{1}+|\nabla\cdot\zeta_i|_{L^{\infty}}|u|_{L^2}|v|_{1}.
\end{eqnarray*}
Therefore, by we can deduce that
$$|(\zeta_i\cdot\nabla)u|_{-1}^2\lesssim \Big(|\zeta_i|_{L^{\infty}}^2+|\nabla\cdot\zeta_i|_{L^{\infty}}^2\Big)|u|_{L^2}^2,$$
which yields
\begin{eqnarray*}
\|\sigma(u)\|_{\mathcal{L}_2(l^2;\mathbb{H}^{-1})}^2\!\!\!\!\!\!\!\!&&= \sum_{i=1}^{\infty}|(\zeta_i\cdot\nabla)u|_{-1}^2
\nonumber \\
\!\!\!\!\!\!\!\!&&\lesssim
 \Big(\sum_{i=1}^{\infty}|\zeta_i|_{L^{\infty}}^2+\sum_{i=1}^{\infty}|\nabla\cdot\zeta_i|_{L^{\infty}}^2\Big)|u|_{L^2}^2
 \nonumber \\
\!\!\!\!\!\!\!\!&&\lesssim_{\mathcal{N}_1}|u|_{L^2}^2.
\end{eqnarray*}
We complete the proof.
\end{proof}

\subsection{Energy estimates}\label{sub2.3}

We denote by $\Pi_n:\mathbb{H}^{-1/2} \to \mathbb{H}_n(:={\rm span}\{e_k:k\leq{n}\})$  the projection operator, where $\{e_1,e_2,\cdots\}\subset\mathcal{D}(\mathcal{A})$ is a complete orthonormal basis on  $\mathbb{H}$, which is given by
$$\Pi_n x:=\sum\limits_{i=1}^{n}\langle x,e_i\rangle e_i,~x\in \mathbb{H}^{-1/2}.$$

Consider the Faedo-Galerkin approximation of Eq.~(\ref{sns1}), i.e.,
\begin{equation}\label{eqf}
\begin{cases}
&\!\!\!\!\!\!du_{n}(t)+[\mathcal{A} u_{n}(t)+\Pi_nB( u_{n}(t))]dt = \Pi_n\sigma(u_n(t))d \mathcal{W}(t)+\Pi_ng(t,u_n(t))d \hat{\mathcal{W}}(t) , \\
&\!\!\!\!\!\!u_{n}(0)=x_{n}:=\Pi_n x.
\end{cases}
\end{equation}
Under $(\mathbf{H}_{g}^1)$ and $(\mathbf{H}_{g}^4)$, it is clear that there exists a weak solution to Eq.~(\ref{eqf})  up to its life time. Furthermore, under the assumption $(\mathbf{H}_{g}^2)$, the solution is non-explosive,  whose proof is the same as that of the following lemma. Thus, there is a global weak solution to Eq.~(\ref{eqf}).

\vspace{2mm}
We have the following energy moment estimates.
\begin{lemma}\label{lem3.0}
$($Energy estimates for $\mathbb{H}^{1/2}$-norm$)$ For any $x\in\mathbb{H}^{1/2}$, we have
$$\sup_{t\in[0,T]}\mathbf{E}|u_n(t)|_{1/2}^{2-\gamma}+\mathbf{E}\int_0^{T}| u_n(s)|_{3/2}^{2-\gamma}ds\lesssim|x|_{1/2}^{2-\gamma}
+ T.$$
\end{lemma}

\begin{proof}
By It\^{o}'s formula to $|\cdot|_{1/2}^2$, we obtain for all $t\in [0,T]$,
\begin{eqnarray}\label{es4}
\!\!\!\!\!\!\!\!&&|u_n(t)|_{1/2}^2
\nonumber \\
=\!\!\!\!\!\!\!\!&&|x|_{1/2}^2+2\mathcal{M}^1(t)+2\mathcal{M}^2(t)+2\int_0^t(-\mathcal{A}u_n(s),\Lambda u_n(s))ds-2\int_0^t(B(u_n(s)),\Lambda u_n(s))ds
\nonumber \\
\!\!\!\!\!\!\!\!&&+\int_0^t\|\Pi_n\sigma(u_n(s))\|_{\mathcal{L}_2(l^2;\mathbb{H}^{1/2})}^2ds+\int_0^t\|\Pi_n g(s,u_n(s))\|_{\mathcal{L}_2(l^2;\mathbb{H}^{1/2})}^2ds,
\end{eqnarray}
where $\mathcal{M}^1(t)$ and $\mathcal{M}^2(t)$ are  continuous local martingales given by
\begin{eqnarray*}
\!\!\!\!\!\!\!\!&&\mathcal{M}^1(t):=\int_0^t(\sigma(u_n(s))d\mathcal{W}(s),u_n(s))_{1/2},
\nonumber \\
\!\!\!\!\!\!\!\!&&
\mathcal{M}^2(t):=\int_0^t(g(s,u_n(s))d\hat{\mathcal{W}}(s),u_n(s))_{1/2}.
\end{eqnarray*}

Next, applying It\^{o}'s formula to the  auxiliary function $\Phi^{\varepsilon}(x):=(\varepsilon+x)^{\alpha}$, for any $\varepsilon,\alpha\in(0,1)$,  we have
\begin{eqnarray}\label{es5}
\!\!\!\!\!\!\!\!&&\Phi^{\varepsilon}(|u_n(t)|_{1/2}^2)
\nonumber \\
=\!\!\!\!\!\!\!\!&&\Phi^{\varepsilon}(|x|_{1/2}^2)+2\alpha\mathcal{M}^{1,\varepsilon}(t)+2\alpha\mathcal{M}^{2,\varepsilon}(t)
\nonumber \\
\!\!\!\!\!\!\!\!&&
+\alpha\int_0^t\frac{-2|u_n(t)|_{3/2}^2-2(B(u_n(t)),\Lambda u_n(t))_{L^2}}{(\varepsilon+|u_n(t)|_{1/2}^2)^{1-\alpha}}dt
\nonumber \\
\!\!\!\!\!\!\!\!&&
+\alpha\int_0^t\frac{\|\Pi_n\sigma(u_n(s))\|_{\mathcal{L}_2(l^2;\mathbb{H}^{1/2})}^2}{(\varepsilon+|u_n(s)|_{1/2}^2)^{1-\alpha}}ds
\nonumber \\
\!\!\!\!\!\!\!\!&&
+\alpha\int_0^t\frac{\|\Pi_ng(s,u_n(s))\|_{\mathcal{L}_2(l^2;\mathbb{H}^{1/2})}^2}{(\varepsilon+|u_n(s)|_{1/2}^2)^{1-\alpha}}ds
\nonumber \\
\!\!\!\!\!\!\!\!&&
-2\alpha(1-\alpha)\int_0^t\frac{\|(\sigma(u_n(s))\cdot,u_n(s))_{1/2}\|_{\mathcal{L}_2(l^2;\mathbb{R})}^2}{(\varepsilon+|u_n(s)|_{1/2}^2)^{2-\alpha}}ds
\nonumber \\
\!\!\!\!\!\!\!\!&&
-2\alpha(1-\alpha)\int_0^t\frac{\|(g(s,u_n(s))\cdot,u_n(s))_{1/2}\|_{\mathcal{L}_2(l^2;\mathbb{R})}^2}{(\varepsilon+|u_n(s)|_{1/2}^2)^{2-\alpha}}ds
\nonumber \\
=:\!\!\!\!\!\!\!\!&&\Phi^{\varepsilon}(|x|_{1/2}^2)+2\alpha\mathcal{M}^{1,\varepsilon}(t)+2\alpha\mathcal{M}^{2,\varepsilon}(t)+\alpha (\text{I}+\text{II}+\text{III})-2\alpha(1-\alpha)(\text{IV}+\text{V}),~
\end{eqnarray}
where $\mathcal{M}^{1,\varepsilon}(t)$ and $\mathcal{M}^{2,\varepsilon}(t)$ are  continuous local martingales given by
\begin{eqnarray*}
\!\!\!\!\!\!\!\!&&\mathcal{M}^{1,\varepsilon}(t):=\int_0^t\frac{(\sigma(u_n(s))d\mathcal{W}(s),u_n(s))_{1/2}}{(\varepsilon+|u_n(s)|_{1/2}^2)^{2-\alpha}}ds,
\nonumber \\
\!\!\!\!\!\!\!\!&&
\mathcal{M}^{2,\varepsilon}(t):=\int_0^t\frac{(g(s,u_n(s))d\hat{\mathcal{W}}(s),u_n(s))_{1/2}}{(\varepsilon+|u_n(s)|_{1/2}^2)^{2-\alpha}}ds.
\end{eqnarray*}

Thus, by (\ref{es7})  it follows that
\begin{eqnarray}\label{es2}
\text{I}\leq\!\!\!\!\!\!\!\!&&\int_0^t\frac{-| u_n(s)|_{3/2}^2+\delta_1|u_n(s)|_{1}^4|u_n(s)|_{1/2}^2}{(\varepsilon+|u_n(s)|_{1/2}^2)^{1-\alpha}}ds.
\end{eqnarray}
Moreover, using the estimate (\ref{sig1}) yields
\begin{eqnarray}\label{es3}
\text{II}\leq\int_0^t\frac{\frac{1}{2}| u_n(s)|_{3/2}^2+\frac{4\delta_0^2\mathcal{N}_{1/2}^2}{1-8\mathcal{N}_0}| u_n(s)|_{1/2}^2}{(\varepsilon+|u_n(s)|_{1/2}^2)^{1-\alpha}}ds.
\end{eqnarray}

 Combining (\ref{es5})-(\ref{es3}),  for all $t\in[0,T]$,
\begin{eqnarray}\label{es18}
\!\!\!\!\!\!\!\!&&\Phi^{\varepsilon}(|u_n(t)|_{1/2}^2)+\frac{\alpha}{2}\int_0^t\frac{| u_n(s)|_{3/2}^2}{(\varepsilon+|u_n(s)|_{1/2}^2)^{1-\alpha}}ds
\nonumber \\
\leq\!\!\!\!\!\!\!\!&&\Phi^{\varepsilon}(|x|_{1/2}^2)+2\alpha\mathcal{M}^{1,\varepsilon}(t)+2\alpha\mathcal{M}^{2,\varepsilon}(t)+\alpha\int_0^t\frac{\frac{4M_0^2}{1-8N_0}|u_n(s)|_{1/2}^2}{(\varepsilon+|u_n(s)|_{1/2}^2)^{1-\alpha}}ds
\nonumber \\
\!\!\!\!\!\!\!\!&&
+\alpha\int_0^t\Bigg\{\frac{\big(\delta_1|u_n(s)|_{1}^{4}|u_n(s)|_{1/2}^2+\|g(s,u_n(s))\|_{\mathcal{L}_2(l^2;\mathbb{H}^{1/2})}^2\big)(\varepsilon+|u_n(s)|_{1/2}^2)}{(\varepsilon+|u_n(s)|_{1/2}^2)^{2-\alpha}}
\nonumber \\
\!\!\!\!\!\!\!\!&&+\frac{-2(1-\alpha)\|(g(s,u_n(s))\cdot,u_n(s))_{1/2}\|_{\mathcal{L}_2(l^2;\mathbb{R})}^2}{(\varepsilon+|u_n(s)|_{1/2}^2)^{2-\alpha}}\Bigg\}ds
\nonumber \\
\leq\!\!\!\!\!\!\!\!&&\Phi^{\varepsilon}(|x|_{1/2}^2)+2\alpha\mathcal{M}^{1,\varepsilon}(t)+2\alpha\mathcal{M}^{2,\varepsilon}(t)
\nonumber \\
\!\!\!\!\!\!\!\!&&
+\alpha\int_0^t\Bigg\{\frac{\big(\kappa_1(|u_n(s)|_{1}^{4}+1)|u_n(s)|_{1/2}^2+\|g(s,u_n(s))\|_{\mathcal{L}_2(l^2;\mathbb{H}^{1/2})}^2\big)(\varepsilon+|u_n(s)|_{1/2}^2)}{(\varepsilon+|u_n(s)|_{1/2}^2)^{2-\alpha}}
\nonumber \\
\!\!\!\!\!\!\!\!&&+\frac{-2(1-\alpha)\|(g(s,u_n(s))\cdot,u_n(s))_{1/2}\|_{\mathcal{L}_2(l^2;\mathbb{R})}^2}{(\varepsilon+|u_n(s)|_{1/2}^2)^{2-\alpha}}\Bigg\}ds,
\end{eqnarray}
where we have used the assumption $\kappa_1\geq\max\{\delta_1,\frac{4\delta_0^2\mathcal{N}_{1/2}^2}{1-8\mathcal{N}_0}\}$.

Then it follows from (\ref{es18}) and the assumption $g(t,0)=0$ that
\begin{eqnarray}\label{es188}
\!\!\!\!\!\!\!\!&&\Phi^{\varepsilon}(|u_n(t)|_{1/2}^2)+\frac{\alpha}{2}\int_0^t\frac{| u_n(s)|_{3/2}^2}{(1+|u_n(s)|_{1/2}^2)^{1-\alpha}}ds
\nonumber \\
\leq\!\!\!\!\!\!\!\!&&\Phi^{\varepsilon}(|x|_{1/2}^2)+2\alpha\mathcal{M}^{1,\varepsilon}(t)+2\alpha\mathcal{M}^{2,\varepsilon}(t)
\nonumber \\
\!\!\!\!\!\!\!\!&&
+\alpha\int_0^t\Bigg\{\frac{\big(\kappa_1(|u_n(s)|_{1}^{4}+1)|u_n(s)|_{1/2}^2+\|g(s,u_n(s))\|_{\mathcal{L}_2(l^2;\mathbb{H}^{1/2})}^2\big)(\varepsilon+|u_n(s)|_{1/2}^2)}{(\varepsilon+|u_n(s)|_{1/2}^2)^{2-\alpha}}
\nonumber \\
\!\!\!\!\!\!\!\!&&+\frac{-2(1-\alpha)\|(g(s,u_n(s))\cdot,u_n(s))_{1/2}\|_{\mathcal{L}_2(l^2;\mathbb{R})}^2}{(\varepsilon+|u_n(s)|_{1/2}^2)^{2-\alpha}}\Bigg\}\mathbf{1}_{\{|u_n(s)|_{1/2}>0\}}ds.
\end{eqnarray}

For any $R>0$, we define some stopping times
\begin{equation*}
\tau^1_{n,R}:=\inf\Big\{t\geq 0:|u_n(t)|_{1/2}\geq R\Big\},
\end{equation*}
\begin{equation*}
\tau^2_{n,R}:=\inf\Bigg\{t\geq 0:\int_0^t\frac{| u_n(s)|_{3/2}^2}{(1+|u_n(s)|_{1/2}^2)^{1-\alpha}}ds\geq R\Bigg\},
\end{equation*}
and
\begin{equation}\label{stop1}
\tau_{n,R}:=\tau^1_{n,R}\wedge \tau^2_{n,R}.
\end{equation}
Then, it follows from the condition (\ref{cong1}) and the estimate (\ref{sig1}) that $\mathcal{M}^{1,\varepsilon}(t)$ and  $\mathcal{M}^{2,\varepsilon}(t)$ are martingales up to $\tau_{n,R}$.

Considering the stopping time $\tau_{n,R}$ and taking expectation on both sides of (\ref{es188}) gives
\begin{eqnarray}\label{es189}
\!\!\!\!\!\!\!\!&&\mathbf{E}\Phi^{\varepsilon}(|u_n(t\wedge\tau_{n,R})|_{1/2}^2)+\frac{\alpha}{2}\mathbf{E}\int_0^{t\wedge\tau_{n,R}}\frac{| u_n(s)|_{3/2}^2}{(1+|u_n(s)|_{1/2}^2)^{1-\alpha}}ds
\nonumber \\
\leq\!\!\!\!\!\!\!\!&&\Phi^{\varepsilon}(|x|_{1/2}^2)
+\alpha\mathbf{E}\int_0^{t\wedge\tau_{n,R}}\mathbf{1}_{\{|u_n(s)|_{1/2}>0\}}\Bigg\{\frac{-2(1-\alpha)\|(g(s,u_n(s))\cdot,u_n(s))_{1/2}\|_{\mathcal{L}_2(l^2;\mathbb{R})}^2}{(\varepsilon+|u_n(s)|_{1/2}^2)^{2-\alpha}}
\nonumber \\
\!\!\!\!\!\!\!\!&&+\frac{\big(\kappa_1(|u_n(s)|_{1}^{4}+1)|u_n(s)|_{1/2}^2+\|g(s,u_n(s))\|_{\mathcal{L}_2(l^2;\mathbb{H}^{1/2})}^2\big)(\varepsilon+|u_n(s)|_{1/2}^2)}{(\varepsilon+|u_n(s)|_{1/2}^2)^{2-\alpha}}\Bigg\}ds.
\end{eqnarray}
Define the function  $\Phi(x):=x^{\alpha}$. Now letting $\varepsilon\to 0$ on both sides of (\ref{es189}) and taking $\alpha=1-\frac{\gamma}{2}$, we derive
\begin{eqnarray}\label{es10}
\!\!\!\!\!\!\!\!&&\mathbf{E}\Phi(|u_n(t\wedge\tau_{n,R})|_{1/2}^2)+\frac{\alpha}{2}\mathbf{E}\int_0^{t\wedge\tau_{n,R}}\frac{| u_n(s)|_{3/2}^2}{(1+|u_n(s)|_{1/2}^2)^{1-\alpha}}ds
\nonumber \\
\leq\!\!\!\!\!\!\!\!&&\Phi(|x|_{1/2}^2)
+\alpha\mathbf{E}\int_0^{t\wedge\tau_{n,R}}\mathbf{1}_{\{|u_n(s)|_{1/2}>0\}}\Bigg\{\frac{-2(1-\alpha)\|(g(s,u_n(s))\cdot,u_n(s))_{1/2}\|_{\mathcal{L}_2(l^2;\mathbb{R})}^2}{|u_n(s)|_{1/2}^{2(2-\alpha)}}
\nonumber \\
\!\!\!\!\!\!\!\!&&+\frac{\big(\kappa_1(|u_n(s)|_{1}^{4}+1)|u_n(s)|_{1/2}^2+\|g(s,u_n(s))\|_{\mathcal{L}_2(l^2;\mathbb{H}^{1/2})}^2\big)|u_n(s)|_{1/2}^2}{|u_n(s)|_{1/2}^{2(2-\alpha)}}\Bigg\}ds
\nonumber \\
\leq\!\!\!\!\!\!\!\!&&\Phi(|x|_{1/2}^2)
+\alpha C\mathbf{E}\int_0^{t\wedge\tau_{n,R}}\mathbf{1}_{\{|u_n(s)|_{1/2}>0\}}\frac{|u_n(s)|_{1/2}^{2+\gamma}}{|u_n(s)|_{1/2}^{2(2-\alpha)}}ds
\nonumber \\
\leq\!\!\!\!\!\!\!\!&&\Phi(|x|_{1/2}^2)
+\alpha CT,
\end{eqnarray}
where we used the condition $(\mathbf{H}_{g}^2)$ in the first inequality.

Letting $R\to\infty$  in (\ref{es10}), by Fatou's lemma we have
\begin{equation}\label{es20}
\sup_{t\in[0,T]}\mathbf{E}\Phi(|u_n(t)|_{1/2}^2)+\frac{\alpha}{2}\mathbf{E}\int_0^{T}\frac{|u_n(t)|_{3/2}^2}{(1+|u_n(t)|_{1/2}^2)^{1-\alpha}}dt
\leq \Phi(|x|_{1/2}^2)
+\alpha CT.
\end{equation}
On the other hand,
we observe that
\begin{eqnarray}\label{es13}
\mathbf{E}\int_0^{T}|u_n(t)|_{3/2}^{2-\gamma}dt\!\!\!\!\!\!\!\!&&=\mathbf{E}\int_0^{T}\frac{|u_n(t)|_{3/2}^{2-\gamma}(1+|u_n(t)|_{1/2}^2)^{\gamma/2}}{(1+|u_n(t)|_{1/2}^2)^{\gamma/2}}dt
\nonumber \\
\!\!\!\!\!\!\!\!&& \lesssim\mathbf{E}\int_0^{T}\frac{1+|u_n(t)|_{3/2}^2}{(1+|u_n(t)|_{1/2}^2)^{\gamma/2}}dt
\nonumber \\
\!\!\!\!\!\!\!\!&&\lesssim
 \Phi(|x|_{1/2}^2)
+T.
\end{eqnarray}
Combining (\ref{es20}) and (\ref{es13}), we complete the proof of Lemma \ref{lem3.0}.
\end{proof}

\vspace{1mm}
Let $\varepsilon>0$. In the following, we provide a bootstrap type estimate by the  kinematic viscosity, which plays an essential role in proving the tightness and the convergence of the approximating sequence in $\mathcal{X}_2$.

\begin{lemma}\label{lem5}
$($Bootstrap type estimates$)$   For any $x\in\mathbb{H}^{1/2}$, there is a constant $C_{T,x,\varepsilon}>0$,
$$\sup_{t\in[\varepsilon,T]}\mathbf{E}\log(1+|u_n(t)|_{1}^2)+\mathbf{E}\int_{\varepsilon}^{T}\frac{| u_n(s)|_{2}^2}{1+|u_n(s)|_{1}^2}ds\leq C_{T,x,\varepsilon}.$$

\end{lemma}

\begin{proof}
Taking $\varepsilon>0$ and by Sobolev embedding theorem (i.e.~Lemma \ref{lemso}),  it follows from Lemma \ref{lem3.0} that
\begin{equation}\label{es14}
\mathbf{E}\int_{\varepsilon/2}^{\varepsilon}|u_n(t)|_{1}^{2-\gamma}dt\lesssim_T
 1+\Phi(|x|_{1/2}^2).
\end{equation}
Therefore, we can find $\theta_{\varepsilon,n}\in[\frac{\varepsilon}{2},\varepsilon]$ such that
\begin{equation*}
\frac{\varepsilon}{2}\mathbf{E}|u_n(\theta_{\varepsilon,n})|_{1}^{2-\gamma}\lesssim_T
 1+\Phi(|x|_{1/2}^2),
\end{equation*}
which leads to
\begin{equation}\label{es15}
\mathbf{E}\big(1+|u_n(\theta_{\varepsilon,n})|_{1}^{2}\big)^{\alpha}\leq \frac{C_{T,x}}{\varepsilon},
\end{equation}
where we recall $\alpha=1-\frac{\gamma}{2}$.

On the other hand,  applying It\^{o}'s formula to $|\cdot|_{1}^2$ we have
\begin{eqnarray}\label{es16}
|u_n(t)|_{1}^2
=\!\!\!\!\!\!\!\!&&|u_n(\theta_{\varepsilon,n})|_{1}^2+2\mathcal{M}_{\varepsilon}^1(t)+2\mathcal{M}_{\varepsilon}^2(t)
\nonumber \\
\!\!\!\!\!\!\!\!&&
+2\int_{\theta_{\varepsilon,n}}^t(-\mathcal{A}u_n(s),\Lambda^2 u_n(s))ds
\nonumber \\
\!\!\!\!\!\!\!\!&&
-2\int_{\theta_{\varepsilon,n}}^t(B(u_n(s)),\Lambda^2 u_n(s))ds
\nonumber \\
\!\!\!\!\!\!\!\!&&+\int_{\theta_{\varepsilon,n}}^t\|\Pi_n\sigma(u_n(s))\|_{\mathcal{L}_2(l^2;\mathbb{H}^{1})}^2ds
\nonumber \\
\!\!\!\!\!\!\!\!&&
+\int_{\theta_{\varepsilon,n}}^t\|\Pi_n g(s,u_n(s))\|_{\mathcal{L}_2(l^2;\mathbb{H}^{1})}^2ds,
\end{eqnarray}
where $\mathcal{M}_{\varepsilon}^1(t)$ and $\mathcal{M}_{\varepsilon}^2(t)$ are  continuous local martingales given by
\begin{equation*}
\mathcal{M}_{\varepsilon}^1(t):=\int_{\theta_{\varepsilon,n}}^t(\sigma(u_n(s))d\mathcal{W}(s),u_n(s))_{1}
\end{equation*}
and
\begin{equation*}
\mathcal{M}_{\varepsilon}^2(t):=\int_{\theta_{\varepsilon,n}}^t(g(s,u_n(s))d\hat{\mathcal{W}}(s),u_n(s))_{1}.
\end{equation*}
Then, applying It\^{o}'s formula to the  auxiliary function $\Psi(x):=\log(1+x)$,  it follows that
\begin{eqnarray}\label{es17}
\!\!\!\!\!\!\!\!&&\Psi(|u_n(t)|_{1}^2)
\nonumber \\
=\!\!\!\!\!\!\!\!&&\Psi(|u_n(\theta_{\varepsilon,n})|_{1}^2)+2\tilde{\mathcal{M}}_{\varepsilon}^1(t)+2\tilde{\mathcal{M}}_{\varepsilon}^2(t)
\nonumber \\
\!\!\!\!\!\!\!\!&&
+\int_{\theta_{\varepsilon,n}}^t\frac{-2| u_n(s)|_{2}^2-2(B(u_n(s)),\Lambda^2  u_n(s))}{1+|u_n(s)|_{1}^2}ds
\nonumber \\
\!\!\!\!\!\!\!\!&&
+\int_{\theta_{\varepsilon,n}}^t\frac{\|\Pi_n\sigma(u_n(s))\|_{\mathcal{L}_2(l^2;\mathbb{H}^{1})}^2}{1+|u_n(s)|_{1}^2}ds
\nonumber \\
\!\!\!\!\!\!\!\!&&
+\int_{\theta_{\varepsilon,n}}^t\frac{\|\Pi_ng(s,u_n(s))\|_{\mathcal{L}_2(l^2;\mathbb{H}^{1})}^2}{1+|u_n(s)|_{1}^2}ds
\nonumber \\
\!\!\!\!\!\!\!\!&&
-\int_{\theta_{\varepsilon,n}}^t\frac{2\|(\sigma(u_n(s))\cdot,u_n(s))_{1}\|_{\mathcal{L}_2(l^2;\mathbb{R})}^2}{(1+|u_n(s)|_{1}^2)^2}ds
\nonumber \\
\!\!\!\!\!\!\!\!&&
-\int_{\theta_{\varepsilon,n}}^t\frac{2\|(g(s,u_n(s))\cdot,u_n(s))_{1}\|_{\mathcal{L}_2(l^2;\mathbb{R})}^2}{(1+|u_n(s)|_{1}^2)^2}ds
\nonumber \\
=:\!\!\!\!\!\!\!\!&&\Psi(|u_n(\theta_{\varepsilon,n})|_{1}^2)+2\tilde{\mathcal{M}}_{\varepsilon}^1(t)+2\tilde{\mathcal{M}}_{\varepsilon}^2(t)+ \text{I}+\text{II}+\text{III}+\text{IV}+\text{V},
\end{eqnarray}
where $\tilde{\mathcal{M}}_{\varepsilon}^1(t)$ and $\tilde{\mathcal{M}}_{\varepsilon}^2(t)$ are  continuous local martingales given by
$$\tilde{\mathcal{M}}_{\varepsilon}^1(t):=\int_{\theta_{\varepsilon,n}}^t(\sigma(u_n(s))d\mathcal{W}(s),u_n(s))_{1}$$
and
$$\tilde{\mathcal{M}}_{\varepsilon}^2(t):=\int_{\theta_{\varepsilon,n}}^t(g(s,u_n(s))d\hat{\mathcal{W}}(s),u_n(s))_{1}.$$

It follows from (\ref{es29}) we have
\begin{eqnarray}\label{es19}
\text{I}\leq\!\!\!\!\!\!\!\!&&\int_{\theta_{\varepsilon,n}}^t\frac{-| u_n(s)|_{2}^2+2\delta_2|u_n(s)|_{1}^6}{1+|u_n(s)|_{1}^2}ds.
\end{eqnarray}
 Moreover, by the estimate (\ref{sig2}) we infer that
\begin{equation}\label{es30}
\text{II}\leq\frac{1}{2}\int_{\theta_{\varepsilon,n}}^t\frac{| u_n(s)|_{2}^2}{1+|u_n(s)|_{1}^2}ds+CT.
\end{equation}

Combining (\ref{es17}), (\ref{es19}) and (\ref{es30}), it follows that $(\mathbf{H}_{g}^3)$ that
\begin{eqnarray}\label{es31}
\!\!\!\!\!\!\!\!&&\Psi(|u_n(t)|_{1}^2)+\frac{1}{2}\int_{\theta_{\varepsilon,n}}^t\frac{| u_n(s)|_{2}^2}{1+|u_n(s)|_{1}^2}ds
\nonumber \\
\leq\!\!\!\!\!\!\!\!&&\Psi(|u_n(\theta_{\varepsilon,n})|_{1}^2)+CT+2\tilde{\mathcal{M}}_{\varepsilon}^1(t)+2\tilde{\mathcal{M}}_{\varepsilon}^2(t)
\nonumber \\
\!\!\!\!\!\!\!\!&&
+\int_{\theta_{\varepsilon,n}}^t\Bigg\{\frac{\big(\delta_2|u_n(s)|_{1}^6+\|g(s,u_n(s))\|_{\mathcal{L}_2(l^2;\mathbb{H}^{1})}^2\big)(1+|u_n(s)|_{1}^2)}{(1+|u_n(s)|_{1}^2)^2}
\nonumber \\
\!\!\!\!\!\!\!\!&&+\frac{-2\|(g(s,u_n(s))\cdot,u_n(s))_{1}\|_{\mathcal{L}_2(l^2;\mathbb{R})}^2}{(1+|u_n(s)|_{1}^2)^2}\Bigg\}ds
\nonumber \\
\leq\!\!\!\!\!\!\!\!&&\Psi(|u_n(\theta_{\varepsilon,n})|_{1}^2)+CT+2\tilde{\mathcal{M}}_{\varepsilon}^1(t)+2\tilde{\mathcal{M}}_{\varepsilon}^2(t).
\end{eqnarray}

For any $R>0$, we define some stopping times
\begin{equation*}
\tau^1_{\varepsilon,n,R}:=\inf\Big\{t\geq \theta_{\varepsilon,n}:|u_n(t)|_{1}\geq R\Big\},
\end{equation*}
\begin{equation*}
\tau^2_{\varepsilon,n,R}:=\inf\Bigg\{t\geq \theta_{\varepsilon,n}:\int_{\theta_{\varepsilon,n}}^t\frac{| u_n(s)|_{2}^2}{1+|u_n(s)|_{1/2}^2}ds\geq R\Bigg\},
\end{equation*}
and
\begin{equation}\label{stop2}
\tau_{\varepsilon,n,R}:=\tau^1_{\varepsilon,n,R}\wedge \tau^2_{\varepsilon,n,R}.
\end{equation}
Then, it follows from the condition (\ref{cong4}) and the estimate (\ref{sig2}) that $\tilde{\mathcal{M}}_{\varepsilon}^1(t)$ and  $\tilde{\mathcal{M}}_{\varepsilon}^2(t)$ are martingales up to $\tau_{\varepsilon,n,R}$.  From (\ref{es15}) and (\ref{es31}) we can get
\begin{eqnarray}\label{es444}
\!\!\!\!\!\!\!\!&&\mathbf{E}\Psi(|u_n(t\wedge\tau_{\varepsilon,n,R} )|_{1}^2)+\frac{1}{2}\mathbf{E}\int_{\theta_{\varepsilon,n}}^{t\wedge\tau_{\varepsilon,n,R}}\frac{| u_n(s)|_{2}^2}{1+|u_n(s)|_{1}^2}ds
\nonumber \\
\!\!\!\!\!\!\!\!&&\leq\mathbf{E}\Psi(|u_n(\theta_{\varepsilon,n})|_{1}^2)+CT
\leq C_{T,x,\varepsilon}.
\end{eqnarray}
By letting $R\to\infty$  and using Fatou's lemma.
$$\sup_{t\in[\varepsilon,T]}\mathbf{E}\Psi(|u_n(t)|_{1}^2)+\frac{1}{2}\mathbf{E}\int_{\varepsilon}^{T}\frac{| u_n(s)|_{2}^2}{1+|u_n(s)|_{1}^2}ds\leq C_{T,x,\varepsilon}.$$
 We complete the proof.
\end{proof}

Building upon Lemmas \ref{lem3.0} and \ref{lem5}, we also have the following estimates.
\begin{lemma}\label{lem9}
 For any $\eta>0$, there exists  $\mathcal{K}_{\eta}>0$ such that for any $p\geq 2$,
\begin{equation}\label{apri}
\sup_{n\in\mathbb{N}}\mathbf{P}\Bigg(\sup_{t\in[0,T]}|u_{n}(t)|_{1/2}^p+\int_0^T|u_{n}(t)|_{3/2}^{2}dt\geq \mathcal{K}_{\eta}\Bigg)
\leq \eta
\end{equation}
and
\begin{equation}\label{apri4}
\sup_{n\in\mathbb{N}}\mathbf{P}\Bigg(\sup_{t\in[\varepsilon,T]}|u_{n}(t)|_{1}^p\geq \mathcal{K}_{\eta}\Bigg)
\leq \eta,
\end{equation}
where $\varepsilon$ is the same as in Lemma \ref{lem5}.
\end{lemma}

\begin{proof}
\textbf{Proof of (\ref{apri}).} By definition of the stopping time $\tau_{n,R}$ defined in (\ref{stop1}), it is clear that
\begin{equation}\label{pro1}
\mathbf{P}\Bigg(\Big\{|u(\tau_{n,R})|_{1/2}\geq R\Big\}\cup\bigg\{\int_0^{\tau_{n,R}}\frac{|u_n(s)|_{3/2}^2}{(1+|u_n(s)|_{1/2}^2)^{1-\alpha}}ds\geq R\bigg\}\Bigg)=1.
 \end{equation}
Then, it follows from (\ref{es10}) and (\ref{pro1}) that
\begin{eqnarray}\label{es11}
\!\!\!\!\!\!\!\!&&\mathbf{P}\Big(\sup_{t\in[0,T]}|u_n(t)|_{1/2}\geq R\Big)+\mathbf{P}\Bigg(\int_0^{T}\frac{|u_n(t)|_{3/2}^2}{(1+|u_n(t)|_{1/2}^2)^{1-\alpha}}dt\geq R\Bigg)
\nonumber \\
\!\!\!\!\!\!\!\!&&=\mathbf{P}\big(\tau^1_{n,R}\leq T\big)+\mathbf{P}\big(\tau^2_{n,R}\leq T\big)
\nonumber \\
\!\!\!\!\!\!\!\!&&\leq2\mathbf{P}\big(\tau_{n,R}\leq T\big)
\nonumber \\
\!\!\!\!\!\!\!\!&&=2\mathbf{P}\Bigg(\big\{\tau_{n,R}\leq T\big\}\cap\bigg(\Big\{|u(\tau_{n,R})|_{1/2}\geq R\Big\}\cup\bigg\{\int_0^{\tau_{n,R}}\frac{|u_n(t)|_{3/2}^2}{(1+|u_n(t)|_{1/2}^2)^{1-\alpha}}dt\geq R\bigg\}\bigg)\Bigg)
\nonumber \\
\!\!\!\!\!\!\!\!&&\lesssim\mathbf{P}\Big(|u(T\wedge\tau_{n,R})|_{1/2}\geq R\Big)+\mathbf{P}\Bigg(\int_0^{T\wedge\tau_{n,R}}\frac{|u_n(t)|_{3/2}^2}{(1+|u_n(t)|_{1/2}^2)^{1-\alpha}}dt\geq R\Bigg)
\nonumber \\
\!\!\!\!\!\!\!\!&&\lesssim\mathbf{P}\Big(\Phi(|u(T\wedge\tau_{n,R})|_{1/2}^2)\geq \Phi(R^2)\Big)+\mathbf{P}\Bigg(\int_0^{T\wedge\tau_{n,R}}\frac{|u_n(t)|_{3/2}^2}{(1+|u_n(t)|_{1/2}^2)^{1-\alpha}}dt\geq R\Bigg)
\nonumber \\
\!\!\!\!\!\!\!\!&&\lesssim\frac{\mathbf{E}\Phi(|u(T\wedge\tau_{n,R})|_{1/2}^2)}{\Phi(R^2)}+\bigg\{\mathbf{E}\int_0^{T\wedge\tau_{n,R}}\frac{|u_n(t)|_{3/2}^2}{(1+|u_n(t)|_{1/2}^2)^{1-\alpha}}dt\bigg\}\Big/R
\nonumber \\
\!\!\!\!\!\!\!\!&&\lesssim\frac{\Phi(|x|_{1/2}^2)+T}{\Phi(R^2)}+\frac{\Phi(|x|_{1/2}^2)+T}{R}.
\end{eqnarray}

Meanwhile, using (\ref{es11}), for any $M>0$ we also get
\begin{eqnarray}\label{es12}
\!\!\!\!\!\!\!\!&&\mathbf{P}\Bigg(\int_0^T|u_n(t)|_{3/2}^2dt\geq M\Bigg)
\nonumber \\
\leq\!\!\!\!\!\!\!\!&&\mathbf{P}\Bigg(\int_0^T|u_n(t)|_{3/2}^2dt\geq M,\tau_{n,R}\geq T\Bigg)
+\mathbf{P}(\tau_{n,R}< T)
\nonumber \\
\leq\!\!\!\!\!\!\!\!&&\mathbf{P}\Bigg(\int_0^{T\wedge \tau_{n,R}}\frac{| u_n(t)|_{3/2}^2}{(1+|u_n(t)|_{1/2}^2)^{1-\alpha}}\cdot(1+|u_n(t)|_{1/2}^2)^{1-\alpha}dt\geq M\Bigg)
+\mathbf{P}(\tau_{n,R}< T)
\nonumber \\
\leq\!\!\!\!\!\!\!\!&&\mathbf{P}\Bigg(\int_0^{T}\frac{| u_n(t)|_{3/2}^2}{(1+|u_n(t)|_{1/2}^2)^{1-\alpha}}dt\geq \frac{M}{C_R}\Bigg)
+\mathbf{P}(\tau_{n,R}< T)
\nonumber \\
\leq\!\!\!\!\!\!\!\!&&\frac{C_{R,T}\big(1+\Phi(|x|_{1/2}^2)\big)}{M}+\mathbf{P}(\tau_{n,R}< T).
\end{eqnarray}
Combining (\ref{es11}) and (\ref{es12}), taking $M\uparrow\infty$ then  $R\uparrow\infty$, we conclude that (\ref{apri}) holds.

\vspace{1mm}
\noindent\textbf{Proof of (\ref{apri4}).} It follows from (\ref{es444}) that
\begin{eqnarray*}
\!\!\!\!\!\!\!\!&&\mathbf{P}\Big(\sup_{t\in[\varepsilon,T]}|u_n(t)|_{1}\geq R\Big)
\nonumber \\
\!\!\!\!\!\!\!\!&&\leq\mathbf{P}\Big(\sup_{t\in[\theta_{\varepsilon,n},T]}|u_n(t)|_{1}\geq R\Big)
\nonumber \\
\!\!\!\!\!\!\!\!&&\leq\mathbf{P}(\tau_{\varepsilon,n,R}\leq T)
\nonumber \\
\!\!\!\!\!\!\!\!&&\leq\mathbf{P}\Big(|u(T\wedge\tau_{\varepsilon,n,R})|_{1}\geq R\Big)+\mathbf{P}\Bigg(\int_{\theta_{\varepsilon,n}}^{T\wedge\tau_{\varepsilon,n,R}}\frac{|u_n(t)|_{2}^2}{1+|u_n(t)|_{1}^2}dt\geq R\Bigg)
\nonumber \\
\!\!\!\!\!\!\!\!&&\leq\mathbf{P}\Big(\Phi(|u(T\wedge\tau_{\varepsilon,n,R})|_{1}^2)\geq \Phi(R^2)\Big)+\mathbf{P}\Bigg(\int_{\theta_{\varepsilon,n}}^{T\wedge\tau_{\varepsilon,n,R}}\frac{|u_n(t)|_{2}^2}{1+|u_n(t)|_{1}^2}dt\geq R\Bigg)
\nonumber \\
\!\!\!\!\!\!\!\!&&\leq\frac{C_{T,x,\varepsilon}}{\Phi(R^2)}+\frac{C_{T,x,\varepsilon}}{R}.
\end{eqnarray*}
Taking $R\to\infty$, then (\ref{apri4}) follows. We complete the proof of Lemma \ref{lem9}.
\end{proof}

\subsection{Tightness of approximating sequence}\label{sub2.4}
First, the tightness of laws of $(u_n)_{n\geq 1}$ in $\mathcal{X}_1$ is given as follows.
\begin{lemma}\label{lem2}
The sequence  $(\mathcal{L}_{u_n})_{n\geq 1}$ is tight in $\mathcal{X}_1$.
\end{lemma}
\begin{proof}
Combining  Lemmas \ref{lemt} and \ref{lem9}, it is sufficient to show that $(u_n)_{n\geq 1}$ satisfies (\ref{tight2}).

Define a stopping time
\begin{equation}\label{stop3}
\tau_R^{n}:=\inf\Bigg\{t\in[0,T]:|u_n(t)|_{1/2}+\int_0^t|u_n(s)|_{3/2}^{2}ds\geq R\Bigg\}\wedge T,
\end{equation}
with the convention $\inf\emptyset=\infty$.

By Lemma \ref{lem9}, it follows that
\begin{equation}\label{es24}
\lim_{R\rightarrow\infty}\sup_{n\in\mathbb{N}}\mathbf{P}\big(\tau_R^{n}<T\big)=0.
\end{equation}
On the other hand,
\begin{eqnarray}\label{sec3es2}
\!\!\!\!\!\!\!\!&&\mathbf{P}\Bigg(\sup_{\delta\in[0,\Delta]}\int_0^{T-\delta}|u_n(t+\delta)-u_n(t)|_{1/2}^2dt\geq\varepsilon\Bigg)
\nonumber \\
\leq \!\!\!\!\!\!\!\!&&\mathbf{P}\Bigg(\sup_{\delta\in[0,\Delta]}\int_0^{T-\delta}|u_n(t+\delta)-u_n(t)|_{1/2}^2dt\geq\varepsilon,\tau_R^{n}\geq T\Bigg)+\mathbf{P}\big(\tau_R^{n}<T\big)
\nonumber \\
\leq \!\!\!\!\!\!\!\!&&\frac{1}{\varepsilon}\mathbf{E}\Bigg[\sup_{\delta\in[0,\Delta]}\int_0^{T-\delta}|u_n((t+\delta)\wedge\tau_R^{n})-u_n(t\wedge\tau_R^{n})|_{1/2}^2dt\Bigg]+\mathbf{P}\big(\tau_R^{n}<T\big).
\end{eqnarray}

Using It\^{o}'s formula to the process $\big\{u_n(r)-u_n(t\wedge\tau_R^{n}),r\in[t\wedge\tau_R^{n},(t+\Delta)\wedge\tau_R^{n}]\big\}$, we obtain
\begin{eqnarray}\label{es38}
\!\!\!\!\!\!\!\!&&|u_n((t+\delta)\wedge\tau_R^{n})-u_n(t\wedge\tau_R^{n})|_{1/2}^2
\nonumber \\
=\!\!\!\!\!\!\!\!&&2\int_{t\wedge\tau_R^{n}}^{(t+\delta)\wedge\tau_R^{n}}(-\mathcal{A}u_n(s),\Lambda (u_n(s)-u_n(t\wedge\tau_R^{n})))ds
\nonumber \\
\!\!\!\!\!\!\!\!&&
-2\int_{t\wedge\tau_R^{n}}^{(t+\delta)\wedge\tau_R^{n}}(B(u_n(s)),\Lambda (u_n(s)-u_n(t\wedge\tau_R^{n})))ds
\nonumber \\
\!\!\!\!\!\!\!\!&&+\int_{t\wedge\tau_R^{n}}^{(t+\delta)\wedge\tau_R^{n}}\|\Pi_n\sigma(u_n(s))\|_{\mathcal{L}_2(l^2;\mathbb{H}^{1/2})}^2ds
\nonumber \\
\!\!\!\!\!\!\!\!&&
+\int_{t\wedge\tau_R^{n}}^{(t+\delta)\wedge\tau_R^{n}}\|\Pi_n g(s,u_n(s))\|_{\mathcal{L}_2(l^2;\mathbb{H}^{1/2})}^2ds
\nonumber \\
\!\!\!\!\!\!\!\!&&+2\int_{t\wedge\tau_R^{n}}^{(t+\delta)\wedge\tau_R^{n}}(\sigma(u_n(s))d\mathcal{W}(s), u_n(s)-u_n(t\wedge\tau_R^{n}))_{1/2}
\nonumber \\
\!\!\!\!\!\!\!\!&&+2\int_{t\wedge\tau_R^{n}}^{(t+\delta)\wedge\tau_R^{n}}( g(s,u_n(s))d\hat{\mathcal{W}}(s),u_n(s)-u_n(t\wedge\tau_R^{n}))_{1/2}
\nonumber \\
=:\!\!\!\!\!\!\!\!&&\text{I}(t)+\text{II}(t)+\text{III}(t)+\text{IV}(t)+\text{V}(t)+\text{VI}(t).
\end{eqnarray}
As for the term $\text{I}(t)$, using H\"{o}lder's inequality we have
\begin{eqnarray}\label{e9}
\!\!\!\!\!\!\!\!&&\int_0^{T-\delta}\text{I}(t)dt
\nonumber \\
\!\!\!\!\!\!\!\!&&\lesssim   \int_0^{T-\delta}\int_{t\wedge\tau_R^{n}}^{(t+\delta)\wedge\tau_R^{n}}\big(|u_n(s)|_{3/2}^2+|u_n(s)-u_n(t\wedge\tau_R^{n})|_{3/2}^2\big)dsdt
\nonumber \\
\!\!\!\!\!\!\!\!&&\lesssim \int_0^{T-\delta}\int_{t}^{t+\delta}\big(\mathbf{1}_{\{s\leq\tau_R^{n} \}}|u_n(s)|_{3/2}^2+|u_n(t\wedge\tau_R^{n})|_{3/2}^2\big)dsdt
\nonumber \\
\!\!\!\!\!\!\!\!&&=\int_0^{\delta}\int_0^s\big(\mathbf{1}_{\{s\leq\tau_R^{n} \}}|u_n(s)|_{3/2}^2\big)dtds+\int_{\delta}^{T-\delta}\int_{s-\delta}^s\big(\mathbf{1}_{\{s\leq\tau_R^{n} \}}|u_n(s)|_{3/2}^2\big)dtds
\nonumber \\
\!\!\!\!\!\!\!\!&&\,\,\,\,\,\,+\int_{T-\delta}^T\int_{s-\delta}^{T-\delta}\big(\mathbf{1}_{\{s\leq\tau_R^{n} \}}|u_n(s)|_{3/2}^2\big)dtds+ \int_0^{T-\delta}\int_{t}^{t+\delta}|u_n(t\wedge\tau_R^{n})|_{3/2}^2dsdt
\nonumber \\
\!\!\!\!\!\!\!\!&&=\int_0^{\delta}\big(s\mathbf{1}_{\{s\leq\tau_R^{n} \}}|u_n(s)|_{3/2}^2\big)ds+\int_{\delta}^{T-\delta}\big(\delta\mathbf{1}_{\{s\leq\tau_R^{n} \}}|u_n(s)|_{3/2}^2\big)ds
\nonumber \\
\!\!\!\!\!\!\!\!&&\,\,\,\,\,\,+\int_{T-\delta}^T(T-s)\big(\mathbf{1}_{\{s\leq\tau_R^{n} \}}|u_n(s)|_{3/2}^2\big)ds+\int_0^{T-\delta}\big(\delta|u_n(t\wedge\tau_R^{n})|_{3/2}^2\big)dt
\nonumber \\
\!\!\!\!\!\!\!\!&&\lesssim_R\delta.
\end{eqnarray}
As for the term $\text{II}(t)$,  by Sobolev embedding inequality and Young's inequality we have
\begin{eqnarray}\label{esq3}
\!\!\!\!\!\!\!\!&&\int_0^{T-\delta}\text{II}(t)dt
\nonumber \\
\!\!\!\!\!\!\!\!&&\lesssim\int_0^{T-\delta}\int_{t\wedge\tau_R^{n}}^{(t+\delta)\wedge\tau_R^{n}}|u_n(s)|_{L^6}|\nabla u_n(s)|_{L^2}|\Lambda(u_n(s)-u_n(t\wedge\tau_R^{n}))|_{L^3}dsdt
\nonumber \\
\!\!\!\!\!\!\!\!&&\lesssim\int_0^{T-\delta}\int_{t\wedge\tau_R^{n}}^{(t+\delta)\wedge\tau_R^{n}}\big(|u_n(s)|_{1/2}^2|u_n(s)|_{3/2}^2+|u_n(s)-u_n(t\wedge\tau_R^{n})|_{3/2}^2\big)dsdt
\nonumber \\
\!\!\!\!\!\!\!\!&&\lesssim_R  \int_0^{T-\delta}\int_{t\wedge\tau_R^{n}}^{(t+\delta)\wedge\tau_R^{n}}\big(|u_n(s)|_{3/2}^2+|u_n(s)-u_n(t\wedge\tau_R^{n})|_{3/2}^2\big)dsdt
\nonumber \\
\!\!\!\!\!\!\!\!&&
\lesssim_R\delta.
\end{eqnarray}
According to (\ref{sig1}) and $(\mathbf{H}_{g}^4)$, it follows that
\begin{eqnarray}\label{es46}
\!\!\!\!\!\!\!\!&&\int_0^{T-\delta}\big(\text{III}(t)+\text{IV}(t)\big)dt
\nonumber \\
\!\!\!\!\!\!\!\!&&\lesssim\int_0^{T-\delta}\int_{t\wedge\tau_R^{n}}^{(t+\delta)\wedge\tau_R^{n}}\big(1+|u_n(s)|_{1/2}^2\big)\big(1+|u_n(s)|_{3/2}^2\big)dsdt
\nonumber \\
\!\!\!\!\!\!\!\!&&
\lesssim_R\delta.
\end{eqnarray}

Regarding the term $\text{V}(t)$, by B-D-G's inequality and H\"{o}lder's inequality we derive
\begin{eqnarray}\label{es81}
\!\!\!\!\!\!\!\!&&\mathbf{E}\Bigg[\sup_{\delta\in[0,\Delta]}\int_0^{T-\delta}\text{V}(t)dt\Bigg]
\nonumber \\
\!\!\!\!\!\!\!\!&&\lesssim\int_0^{T}\mathbf{E}\Bigg[\sup_{\delta\in[0,\Delta]}\int_{t\wedge\tau_R^{n}}^{(t+\delta)\wedge\tau_R^{n}}(\sigma(u_n(s))d\mathcal{W}(s), u_n(s)-u_n(t\wedge\tau_R^{n}))_{1/2}ds\Bigg]dt
\nonumber \\
\!\!\!\!\!\!\!\!&&\lesssim\int_0^{T}\mathbf{E}\Bigg(\int_{t\wedge\tau_R^{n}}^{(t+\Delta)\wedge\tau_R^{n}}\|\sigma(u_n(s))\|_{\mathcal{L}_2(l^2;\mathbb{H}^{1/2})}^2 |u_n(s)-u_n(t\wedge\tau_R^{n})|_{1/2}^2ds\Bigg)^{\frac{1}{2}}dt
\nonumber \\
\!\!\!\!\!\!\!\!&&\lesssim\Bigg(\mathbf{E}\int_0^{T}\int_{t}^{t+\Delta}\mathbf{1}_{\{s\leq\tau_R^{n} \}}\|\sigma(u_n(s))\|_{\mathcal{L}_2(l^2;\mathbb{H}^{1/2})}^2 |u_n(s)-u_n(t\wedge\tau_R^{n})|_{1/2}^2dsdt\Bigg)^{\frac{1}{2}}
\nonumber \\
\!\!\!\!\!\!\!\!&&
\lesssim_R\Delta,
\end{eqnarray}
where the last inequality follows from similar arguments as in (\ref{e9}).

Similarly, we can also get 
\begin{equation}\label{es82}
\mathbf{E}\Bigg[\sup_{\delta\in[0,\Delta]}\int_0^{T-\delta}\text{VI}(t)dt\Bigg]\lesssim_R\Delta.
\end{equation}

Combining the estimates (\ref{es38})-(\ref{es82}) gives
\begin{eqnarray}\label{esq4}
\lim_{\Delta\to 0+}\sup_{n\in\mathbb{N}}\mathbf{E}\sup_{\delta\in[0,\Delta]}\int_0^{T-\delta}|u_n((t+\delta)\wedge\tau_R^{n})-u_n(t\wedge\tau_R^{n})|_{1/2}^2dt=0.
\end{eqnarray}
Collecting (\ref{es24}), (\ref{sec3es2}) and (\ref{esq4}),  letting $\Delta\to 0$  then $R\to\infty$ in (\ref{sec3es2}), we conclude that (\ref{tight2}) holds.
The proof is completed.
\end{proof}

\begin{remark}
In view of  the assumptions on  the diffusion coefficient $g$, it seems to be difficult to show the tightness of the sequence  $(\mathcal{L}_{u_n})_{n\geq 1}$ in the space of  $\mathbb{X}$-valued continuous functions space on $[0,T]$, even when $\mathbb{X}$ is a large enough Hilbert or Banach space. Nevertheless, by  exploiting the improved regularity obtained  in Lemma \ref{lem5} and $(\ref{apri4})$, one can prove the tightness in the space  of  $\mathbb{H}^{-1/2}$-valued continuous functions on $(0,T]$ as follows, using the criterion presented in Lemma \ref{lemt1}.

\end{remark}

The tightness of laws of $(u_n)_{n\geq 1}$ in $\mathcal{X}_2$ is given as follows.
\begin{lemma}\label{lem1}
The sequence  $(\mathcal{L}_{u_n})_{n\geq 1}$ is tight in $\mathcal{X}_2$.
\end{lemma}
\begin{proof}
\textbf{Step 1.} We first prove that for any $k\in\mathbb{N}$ the sequence  $(\mathcal{L}_{u_n})_{n\geq 1}$ is tight in $C([\frac{1}{k},T];\mathbb{H}^{-1/2})$.

In view of Lemma \ref{lem9}, we only need to prove the Aldous condition in Definition \ref{aldous} holds.
Define a stopping time
\begin{equation*}
\tilde{\tau}_R^{n,k}:=\inf\bigg\{t\in[\frac{1}{k},T]:|u_n(t)|_{1}\geq R\bigg\},
\end{equation*}
with the convention $\inf\emptyset=\infty$. Then we denote
\begin{equation}\label{stop4}
\rho_R^{n,k}:=\tau_R^{n}\wedge\tilde{\tau}_R^{n,k},
\end{equation}
where the stopping time $\tau_R^{n}$ is defined by (\ref{stop3}).

Note that due to Lemmas \ref{lem9}, we deduce that
\begin{equation}\label{es52}
\lim_{R\to\infty}\mathbf{P}\big(\rho_R^{n,k}<T\big)
\leq\lim_{R\to\infty}\Big(\mathbf{P}\big(\tau_R^{n}<T\big)+\mathbf{P}\big(\tilde{\tau}_R^{n,k}<T\big)\Big)=0
\end{equation}
and
\begin{eqnarray}\label{es47}
\!\!\!\!\!\!\!\!&&\mathbf{P}(|u_n(\zeta_n+\theta)-u_n(\zeta_n)|_{-1/2}\geq \eta)
\nonumber \\
\leq\!\!\!\!\!\!\!\!&&\mathbf{P}\Big(|u_n(\zeta_n+\theta)-u_n(\zeta_n)|_{-1/2}\geq\eta,\rho_R^{n,k}\geq T\Big)+\mathbf{P}\big(\rho_R^{n,k}<T\big)
\nonumber \\
\leq \!\!\!\!\!\!\!\!&&\frac{1}{\eta^2}\mathbf{E}|u_n((\zeta_n+\theta)\wedge\rho_R^{n,k})-u_n(\zeta_n\wedge\rho_R^{n,k})|_{-1/2}^2+\mathbf{P}\big(\rho_R^{n,k}<T\big),
\end{eqnarray}
where  $(\zeta_n)_{n\in\mathbb{N}}$ is a stopping time sequence with $\zeta_n\in [\frac{1}{k},T]$.

The first term on the right hand side of (\ref{es47})  is estimated as follows
\begin{eqnarray}\label{es48}
\!\!\!\!\!\!\!\!&&\mathbf{E}|u_n((\zeta_n+\theta)\wedge\rho_R^{n,k})-u_n(\zeta_n\wedge\rho_R^{n,k})|_{-1/2}^2
\nonumber \\
\lesssim\!\!\!\!\!\!\!\!&&\mathbf{E}\Bigg(\int_{\zeta_n\wedge \rho_R^{n,k}}^{(\zeta_n+\theta)\wedge  \rho_R^{n,k}}| \mathcal{A}u_n(s)+B(u_n(s))|_{-1/2}ds\Bigg)^{2}
\nonumber \\
 \!\!\!\!\!\!\!\!&&+\mathbf{E}\int_{\zeta_n\wedge \rho_R^{n,k}}^{(\zeta_n+\theta)\wedge  \rho_R^{n,k}}\|\sigma(u_n(s))\|_{\mathcal{L}_2(l^2; \mathbb{H}^{-1/2})}^2ds
\nonumber \\
 \!\!\!\!\!\!\!\!&&+\mathbf{E}\int_{\zeta_n\wedge \rho_R^{n,k}}^{(\zeta_n+\theta)\wedge  \rho_R^{n,k}}\|g(s,u_n(s))\|_{\mathcal{L}_2(l^2; \mathbb{H}^{-1/2})}^2ds
 \nonumber \\
=: \!\!\!\!\!\!\!\!&&\text{I}+\text{II}+\text{III}.
\end{eqnarray}

As for the term $I$, using H\"{o}lder's inequality we have
\begin{eqnarray}\label{es49}
\text{I}\!\!\!\!\!\!\!\!&&\lesssim   |\theta| \mathbf{E}\int_{\zeta_n\wedge \rho_R^{n,k}}^{(\zeta_n+\theta)\wedge  \rho_R^{n,k}}\big(| \mathcal{A}u_n(s)|_{-1/2}^2+|B(u_n(s))|_{-1/2}^2\big)ds
\nonumber \\
\!\!\!\!\!\!\!\!&&\lesssim  |\theta|\mathbf{E}\int_{0}^{T\wedge \rho_R^{n,k}}\big(|u_n(s)|_{3/2}^2+|u_n(s)|_{1/2}^2|u_n(s)|_{3/2}^2\big)ds
\nonumber \\
\!\!\!\!\!\!\!\!&&\lesssim_{R,T} |\theta|.
\end{eqnarray}
As for the term $\text{II}$,  in view of (\ref{sig3}) we have
\begin{eqnarray}\label{es50}
\text{II}
\lesssim_{\mathcal{N}_1}\mathbf{E}\int_{\zeta_n\wedge \rho_R^{n,k}}^{(\zeta_n+\theta)\wedge  \rho_R^{n,k}}
\big|u_n(s)|_{1}^{2}ds
\lesssim_{R,{\mathcal{N}_1}}|\theta|.
\end{eqnarray}
As for the term $\text{III}$,  by the assumption $(\mathbf{H}_{g}^4)$ we have
\begin{eqnarray}\label{es53}
\text{III}
\lesssim\mathbf{E}\int_{\zeta_n\wedge \rho_R^{n,k}}^{(\zeta_n+\theta)\wedge  \rho_R^{n,k}}
\big(1+|u_n(s)|_{1}^{\alpha}\big)\big(1+|u_n(s)|_{1/2}^{\beta}\big)ds
\lesssim_R|\theta|.
\end{eqnarray}
Combining (\ref{es48})-(\ref{es53}) gives
\begin{eqnarray}\label{es51}
\lim_{\theta\to 0}\sup_{n\in\mathbb{N}}\mathbf{E}| u_n((\zeta_n+\theta)\wedge\rho_R^{n,k})-u_n(\zeta_n\wedge\rho_R^{n,k})|_{-1/2}^2=0.
\end{eqnarray}
Collecting (\ref{es52}), (\ref{es47}) and (\ref{es51}),  letting $\theta\to 0$  then $R\to\infty$ in (\ref{es47}), we conclude that the claim holds.

\vspace{1mm}
\noindent\textbf{Step 2.} In this step, we prove the sequence  $(\mathcal{L}_{u_n})_{n\geq 1}$ is tight in $\mathcal{X}_2$.

In view of  Lemma  \ref{lemt1} and Lemma \ref{lem9},   it is sufficient to prove (\ref{stop5}) holds for $u_n$.  Since for any $k\in\mathbb{N}$ the sequence  $(\mathcal{L}_{u_n})_{n\geq 1}$ is tight in $C([\frac{1}{k},T];\mathbb{H}^{-1/2})$, which yields  that for any $\eta>0$,
$$\lim_{\Delta\to0}\sup_{n\in\mathbb{N}}\mathbf{P}\Bigg(\sup_{|t-s|\leq \Delta, t,s\in[\frac{1}{k},T]}|u_n(t)-u_n(s)|_{-1/2}  > \eta   \Bigg)=0.$$
This implies that for any $\eta>0$,
$$\lim_{\Delta\to0}\sup_{n\in\mathbb{N}}\mathbf{P}\Bigg(\sum_{k=1}^{\infty}\frac{1}{2^k}\bigg(\sup_{|t-s|\leq \Delta, t,s\in[\frac{1}{k},T]}|u_n(t)-u_n(s)|_{-1/2}\wedge 1 \bigg) > \eta   \Bigg)=0,$$
which completes the proof.
\end{proof}

\vspace{1mm}
Building upon Lemma \ref{lem2} and Lemma \ref{lem1}, we state the following result.
\begin{lemma}\label{lem8}
The sequence  $(\mathcal{L}_{u_n})_{n\geq 1}$ is tight in $\mathcal{X}:=\mathcal{X}_1\cap\mathcal{X}_2$.
\end{lemma}

\begin{proof}
According to Lemma \ref{lem2} and Lemma \ref{lem1},  for every $\varepsilon>0$, there exist compact subsets $\mathcal{K}_{1,\varepsilon}$ and $\mathcal{K}_{2,\varepsilon}$ in $\mathcal{X}_1$ and $\mathcal{X}_2$, respectively, such that
$$\sup_{n\in\mathbb{N}}\mu_n(\mathcal{K}_{1,\varepsilon}^c)\leq \frac{\varepsilon}{2},~~\sup_{n\in\mathbb{N}}\mu_n(\mathcal{K}_{2,\varepsilon}^c)\leq \frac{\varepsilon}{2}.$$
Recall that $\mathcal{X}_2$ is a separable Banach space that is a Hausdorff space. In view of the proof of Theorem \ref{th1}, we can also deduce that  $\mathcal{X}_1$ is a Hausdorff space. Then the intersection  $\mathcal{K}_{\varepsilon}:=\mathcal{K}_{1,\varepsilon}\cap \mathcal{K}_{2,\varepsilon}$ is  compact in $\mathcal{X}$. Moreover, we have
$$\sup_{n\in\mathbb{N}}\mathcal{L}_{u_n}(\mathcal{K}_{\varepsilon}^c)\leq \sup_{n\in\mathbb{N}}\mathcal{L}_{u_n}(\mathcal{K}_{1,\varepsilon}^c)+\sup_{n\in\mathbb{N}}\mathcal{L}_{u_n}(\mathcal{K}_{2,\varepsilon}^c)\leq \varepsilon.$$
 The proof is  completed.
\end{proof}

\subsection{Passage to the limit}\label{sub2.5}
In this subsection, we first prove the existence of weak solutions via the stochastic compactness method. To pass to the limit $n\to\infty$ in (\ref{eqf}), several nontrivial  difficulties are needed to  deal with.   In particular, since the second moments are unavailable (see Lemmas \ref{lem3.0} and \ref{lem5}),  we need to carefully treat the convergence of every terms in (\ref{eqf}) (cf.~Lemmas \ref{lem4}-\ref{lem6} below). Furthermore, we need to construct a continuous version of the limit of the approximating sequence $(u_n)_{n\in\mathbb{N}}$, which belongs to $C([0,T];\mathbb{H}^{1/2})$. This is achieved through a localization procedure combined with a delicate stopping time argument.

Let us denote
$$\mathcal{Z}_T:=\mathcal{X}\times \mathbb{C}_T(\mathbb{U})\times \mathbb{C}_T(\mathbb{U}),$$
where $\mathbb{U}$ is a Hilbert space such that the embedding $l^2\subset \mathbb{U}$ is Hilbert-Schmidt.
Building upon Lemma \ref{lem8} and applying the  Jakubowski-Skorokhod representation theorem (i.e.~Theorem \ref{sko1}), there exists a subsequence still denoted by $(n)$, a probability space $(\tilde{\Omega},\tilde{\mathscr{F}},\tilde{\mathbf{P}})$ and  $\mathcal{Z}_T$-valued random variables $(\tilde{u}_n,\tilde{\mathcal{W}}_n,\tilde{\hat{\mathcal{W}}}_n)$, $(\tilde{u},\tilde{\mathcal{W}},\tilde{\hat{\mathcal{W}}})$ such that

\vspace{2mm}
\noindent(i) $\mathcal{L}_{(\tilde{u}_n,\tilde{\mathcal{W}}_n,\tilde{\hat{\mathcal{W}}}_n)}|_{\tilde{\mathbf{P}}}=\mathcal{L}_{(u_n,\mathcal{W},\tilde{\hat{\mathcal{W}}})}|_{\mathbf{P}}$, $n\in\mathbb{N}$;

\vspace{2mm}
\noindent(ii) the following convergence hold
\begin{equation}
\tilde{u}_n\to \tilde{u}~\text{in}~\mathcal{X}\,\,\,\,\,\,\tilde{\mathbf{P}}\text{-a.s.},~\text{as}~n\to\infty;\label{es80}
\end{equation}

\vspace{2mm}
\noindent(iii) $\tilde{\mathcal{W}}_n\to\tilde{\mathcal{W}}$ and $\tilde{\hat{\mathcal{W}}}_n\to\tilde{\hat{\mathcal{W}}}$ in  $\mathbb{C}_T(\mathbb{U})$ $\tilde{\mathbf{P}}$-a.s., as $n\to\infty$.

\vspace{2mm}
 Let $(\tilde{\mathscr{F}}^n_t)_{t\in[0,T]}$ be the filtration satisfying the usual conditions and generated by
 $$\big\{\tilde{u}_n(s),\tilde{\mathcal{W}}_n(s),\tilde{\hat{\mathcal{W}}}_n(s):s\in[0,t]\big\}.$$
Due to the claim (i), we have
 \begin{eqnarray*}
\!\!\!\!\!\!\!\!&&\mathbf{P}(\mathcal{W}(t)-\mathcal{W}(s)\in\cdot|\mathscr{F}_s)=\mathbf{P}(\mathcal{W}(t)-\mathcal{W}(s)\in\cdot)
 \\
\Rightarrow\!\!\!\!\!\!\!\!&&\tilde{\mathbf{P}}(\tilde{\mathcal{W}}_n(t)-\tilde{\mathcal{W}}_n(s)\in\cdot|\tilde{\mathscr{F}}^n_s)=\tilde{\mathbf{P}}(\tilde{\mathcal{W}}_n(t)-\tilde{\mathcal{W}}_n(s)\in\cdot).
\end{eqnarray*}
Thus, $\tilde{\mathcal{W}}_n$ is an $l^2$-valued  $(\tilde{\mathscr{F}}^n_t)$-cylindrical Wiener process. Analogously for $\tilde{\hat{\mathcal{W}}}_n$.  Moreover, in view of claims (i) and (iii) and Definition \ref{dew},  the following identity holds, $\tilde{\mathbf{P}}$-a.s. $t\in[0,T]$,
\begin{eqnarray}\label{eqns1}
\!\!\!\!\!\!\!\!&&(\tilde{u}_n(t),\phi)+\int_0^t( \nabla\tilde{u}_n(s),\nabla\phi)ds+\int_0^t(B( \tilde{u}_n(s)),\phi)ds
\nonumber \\
\!\!\!\!\!\!\!\!&&
 =(x_n,\phi)+ \int_0^t(\sigma(\tilde{u}_n(s))  d\tilde{\mathcal{W}}_n(s),\phi)+\int_0^t(g(s,\tilde{u}_n(s))  d\tilde{\hat{\mathcal{W}}}_n(s),\phi),~\phi\in\mathcal{V}.
\end{eqnarray}
According to Lemma \ref{lem9}, the  same bounds  hold for $ \tilde{u}_n$, that is,  for any $\varepsilon>0$ there exists  $\mathcal{K}>0$ such that for any $p\geq 2$,
\begin{equation}
\sup_{n\in\mathbb{N}}\tilde{\mathbf{P}}\Bigg(\sup_{t\in[0,T]}|\tilde{u}_n(t)|_{1/2}^p+\int_0^T|\tilde{u}_n(t)|_{3/2}^{2}dt\geq \mathcal{K}\Bigg)\leq\varepsilon. \label{apri1}
\end{equation}

Notice that by (\ref{es80}) we know $\tilde{u}_n\to \tilde{u}$ in $C((0,T];\mathbb{H}^{-1/2})$ $\tilde{\mathbf{P}}$-a.s., as $n\to\infty$.
Then, employing the lower semicontinuity of the norm $|\cdot|_{1/2}$ in $\mathbb{H}^{-1/2}$,  the weak lower semicontinuity of the norm $\|\cdot\|_{L^{2}([0,T];\mathbb{H}^{3/2})}$ in $L^{2}_{w}([0,T];\mathbb{H}^{3/2})$, and the weak-$*$ lower semicontinuity of the norm $\|\cdot\|_{L^{\infty}([0,T];\mathbb{H}^{1/2})}$ in $L^{\infty}_{w^*}([0,T];\mathbb{H}^{1/2})$, by Fatou's lemma it follows that  for any $\varepsilon>0$  there exists  $\mathcal{K}>0$  such that
\begin{eqnarray}\label{apri2}
\!\!\!\!\!\!\!\!&&\tilde{\mathbf{P}}\Bigg(\|\tilde{u}\|_{L^{\infty}([0,T];\mathbb{H}^{1/2})}^p+\sup_{t\in(0,T]}|\tilde{u}(t)|_{1/2}^p+\|\tilde{u}\|_{L^{2}([0,T];\mathbb{H}^{3/2})}^{2}dt\geq \mathcal{K}\Bigg)
\nonumber \\
\leq\!\!\!\!\!\!\!\!&&\tilde{\mathbf{P}}\Bigg(\liminf_{n\to\infty}\|\tilde{u}_n\|_{L^{\infty}([0,T];\mathbb{H}^{1/2})}^p+\sup_{t\in(0,T]}\liminf_{n\to\infty}|\tilde{u}_n(t)|_{1/2}^p
\nonumber \\
\!\!\!\!\!\!\!\!&&~~~~~
+\liminf_{n\to\infty}\|\tilde{u}_n\|_{L^{2}([0,T];\mathbb{H}^{3/2})}^{2}\geq \mathcal{K}\Bigg)
\nonumber \\
\leq\!\!\!\!\!\!\!\!&&\tilde{\mathbf{P}}\Bigg(\liminf_{n\to\infty}\bigg\{\|\tilde{u}_n\|_{L^{\infty}([0,T];\mathbb{H}^{1/2})}^p+\sup_{t\in(0,T]}|\tilde{u}_n(t)|_{1/2}^p+\|\tilde{u}_n\|_{L^{2}([0,T];\mathbb{H}^{3/2})}^{2}\bigg\}\geq \mathcal{K}\Bigg)
\nonumber \\
\leq\!\!\!\!\!\!\!\!&&\sup_{n\in\mathbb{N}}\tilde{\mathbf{P}}\Bigg(\|\tilde{u}_n\|_{L^{\infty}([0,T];\mathbb{H}^{1/2})}^p+\sup_{t\in(0,T]}|\tilde{u}_n(t)|_{1/2}^p+\|\tilde{u}_n\|_{L^{2}([0,T];\mathbb{H}^{3/2})}^{2}\geq \mathcal{K}\Bigg)
\nonumber \\
\leq\!\!\!\!\!\!\!\!&&\sup_{n\in\mathbb{N}}\tilde{\mathbf{P}}\Bigg(\sup_{t\in[0,T]}|\tilde{u}_n(t)|_{1/2}^p+\int_0^T|\tilde{u}_n(t)|_{3/2}^{2}dt\geq \frac{\mathcal{K}}{2}\Bigg)
\nonumber \\
\leq\!\!\!\!\!\!\!\!&&\varepsilon.
\end{eqnarray}

In order to pass to the limit in (\ref{eqns1}), we first consider the convergence of the stochastic integrals in (\ref{eqns1}).
 \begin{lemma}\label{lem4}
Along a subsequence still denoted by $\{n\}$,  we have the following convergence
\begin{equation*}
\int_0^{\cdot}(\sigma(\tilde{u}_n(s))  d\tilde{\mathcal{W}}_n(s),\phi)\to\int_0^{\cdot}(\sigma(\tilde{u}(s))  d\tilde{\mathcal{W}}(s),\phi)~\text{in}~L^{\infty}([0,T];\mathbb{R})\,\,\,\,\,\, \tilde{\mathbf{P}}\text{-a.s..}
\end{equation*}
 \end{lemma}
\begin{proof}
 According to Lemma 4.3 in \cite{BMX}, in order to prove Lemma \ref{lem4} it is sufficient  to show that for any $t\in[0,T]$, along a subsequence still denoted by $\{n\}$, we have
\begin{equation}
\int_0^t |\sigma(\tilde{u}_n(s))^*\phi-\sigma(\tilde{u}(s))^*\phi|_{l^2}^2ds\to 0~\text{in probability as}~n\to\infty. \label{con5}
\end{equation}
Since the mapping $\sigma$ is linear, then in terms of (\ref{sig4}) for any $u,v\in\mathbb{H}$, we have
\begin{equation}\label{con6}
\|\sigma(u)-\sigma(v)\|_{\mathcal{L}_2(l^2;\mathbb{H}^{-1})}\lesssim |u-v|_{L^2}.
\end{equation}

\noindent\textbf{Step 1.} In view of (\ref{es80}), we have the convergence $\tilde{u}_n\to \tilde{u}$ in $L^2([0,T];\mathbb{H}^{1/2})$ $\tilde{\mathbf{P}}$-a.s., as $n\to\infty$. Then we can deduce that
 \begin{equation}\label{es42}
 \lim_{n\to\infty}|\tilde{u}_n(\omega,t)-\tilde{u}(\omega,t)|_{1/2}=0
\end{equation}
 for $\tilde{\mathbf{P}}\otimes dt$-almost all $(\omega,t)\in\tilde{\Omega}\times [0,T]$ (here selecting a subsequence if necessary). By (\ref{con6}) and (\ref{es42}) it follows that
 \begin{equation}\label{es43}
 \lim_{n\to\infty}\|\sigma(\tilde{u}_n(\omega,t))-\sigma(\tilde{u}(\omega,t))\|_{\mathcal{L}_2(l^2;\mathbb{H}^{-1})}=0
\end{equation}
 for $\tilde{\mathbf{P}}\otimes dt$-almost all $(\omega,t)\in\tilde{\Omega}\times [0,T]$.

\vspace{1mm}
\noindent\textbf{Step 2.} In this part, we prove the convergence
\begin{equation}\label{es54}
\int_0^t\|\sigma(\tilde{u}_n(\omega,s))\|_{ \mathcal{L}_2(l^2;\mathbb{H}^{-1})}^2ds\to\int_0^t\|\sigma(\tilde{u}(\omega,s))\|_{\mathcal{L}_2(l^2;\mathbb{H}^{-1})}^2ds
\end{equation}
holds for $\tilde{\mathbf{P}}$-almost all $\omega\in\tilde{\Omega}$, as $n\to\infty$.

Let $\chi_M\in C^{\infty}_c(\mathbb{R})$ be a cut-off function with
$$\chi_M(r)=\begin{cases} 1,~~~~|r|\leq M&\quad\\
0,~~~~|r|>2M.&\quad\end{cases}$$
Moreover, we set
\begin{eqnarray*}
\!\!\!\!\!\!\!\!&&\Theta_M(t,w):=\int_0^t\|\sigma(w(s))\|_{\mathcal{L}_2(l^2;\mathbb{H}^{-1})}^2\chi_M(|w(s)|_{L^2})ds,
\nonumber \\
\!\!\!\!\!\!\!\!&& \Theta(t,w):=\int_0^t\|\sigma(w(s))\|_{\mathcal{L}_2(l^2;\mathbb{H}^{-1})}^2ds.
\end{eqnarray*}

First, by (\ref{con6}), (\ref{es42}) and the continuity of $\chi_M$,   we have
\begin{eqnarray*}
\lim_{n\to\infty}\!\!\!\!\!\!\!\!&&\Big|\|\sigma(\tilde{u}_n(\omega,t))\|_{\mathcal{L}_2(l^2;\mathbb{H}^{-1})}^2\chi_M(|\tilde{u}_n(\omega,t)|_{L^2})
\nonumber \\
\!\!\!\!\!\!\!\!&&
-
\|\sigma(\tilde{u}(\omega,t))\|_{\mathcal{L}_2(l^2;\mathbb{H}^{-1})}^2\chi_M(|\tilde{u}(\omega,t)|_{L^2})\Big|= 0
\end{eqnarray*}
holds for $\tilde{\mathbf{P}}\otimes dt$-almost all $(\omega,t)\in\tilde{\Omega}\times [0,T]$.
Thus, using  the dominated convergence theorem, it follows that
\begin{equation}\label{es55}
\Theta_M(t,\tilde{u}_n)\to \Theta_M(t,\tilde{u})~~\tilde{\mathbf{P}}\text{-a.s.},~\text{as}~n\to\infty.
\end{equation}
On the other hand, by the definition of $\chi_R$ we obtain that for any $\varepsilon>0$,
\begin{eqnarray}\label{con7}
\!\!\!\!\!\!\!\!&&\tilde{\mathbf{P}}\Big(|\Theta(t,\tilde{u}_n)-\Theta_M(t,\tilde{u}_n)|>\varepsilon\Big)
\nonumber \\
=\!\!\!\!\!\!\!\!&&\tilde{\mathbf{P}}\Big(|\Theta(t,\tilde{u}_n)-\Theta_M(t,\tilde{u}_n)|>\varepsilon,\sup_{t\in[0,T]}|\tilde{u}_n(t)|_{L^2}\leq M\Big)
\nonumber \\
\!\!\!\!\!\!\!\!&&+ \tilde{\mathbf{P}}\Big(|\Theta(t,\tilde{u}_n)-\Theta_M(t,\tilde{u}_n)|>\varepsilon,\sup_{t\in[0,T]}|\tilde{u}_n(t)|_{L^2}> M\Big)
\nonumber \\
\leq\!\!\!\!\!\!\!\!&& \sup_{n\in\mathbb{N}}\tilde{\mathbf{P}}\Big(\sup_{t\in[0,T]}|\tilde{u}_n(t)|_{1/2}> M\Big).
\end{eqnarray}
Due to  (\ref{apri1}), letting $n\to\infty$ then $M\to \infty$ in (\ref{con7}),  we deduce that
\begin{equation}\label{es56}
|\Theta(t,\tilde{u}_n)-\Theta_M(t,\tilde{u}_n)|\to 0~~\text{in probability}.
\end{equation}
By similar argument, we also obtain
\begin{equation}\label{es57}
|\Theta(t,\tilde{u})-\Theta_M(t,\tilde{u})|\to 0~~\text{in probability}.
\end{equation}
Collecting (\ref{es55}), (\ref{es56}) and  (\ref{es57}), then  (\ref{es54}) follows.

\vspace{1mm}
\noindent\textbf{Step 3.} In this step, we prove (\ref{con5}). First, combining (\ref{es43}) and (\ref{es54}), we have
\begin{equation}\label{es58}
\lim_{n\to\infty}\int_0^t\|\sigma(\tilde{u}_n(\omega,s))-\sigma(\tilde{u}(\omega,s))\|_{\mathcal{L}_2(l^2;\mathbb{H}^{-1})}^2ds=0
\end{equation}
holds for $\tilde{\mathbf{P}}$-almost all $\omega\in\tilde{\Omega}$.

Note that
\begin{eqnarray*}
\!\!\!\!\!\!\!\!&&\int_0^t |\sigma(\tilde{u}_n(s))^*\phi-\sigma(\tilde{u}(s))^*\phi|_{l^2}^2ds
\nonumber \\
=\!\!\!\!\!\!\!\!&&\int_0^t \|\langle(\sigma(\tilde{u}_n(s))-\sigma(\tilde{u}(s)))\cdot,\phi\rangle\|_{\mathcal{L}_2(l^2;\mathbb{R})}^2ds
\nonumber \\
\leq\!\!\!\!\!\!\!\!&&\int_0^t\|\sigma(\tilde{u}_n(s))-\sigma(\tilde{u}(s))\|_{\mathcal{L}_2(l^2;\mathbb{H}^{-1})}^2|\phi|_{1}^2ds,
\end{eqnarray*}
which combining with (\ref{es58}) yields  (\ref{con5}).

We complete the proof.
\end{proof}

\begin{lemma}\label{lem7}
Along a subsequence still denoted by $\{n\}$,  we have the following convergence
\begin{equation*}
\int_0^{\cdot}(g(s,\tilde{u}_n(s))  d\tilde{\hat{\mathcal{W}}}_n(s),\phi)\to\int_0^{\cdot}(g(s,\tilde{u}(s))  d\tilde{\hat{\mathcal{W}}}(s),\phi)~\text{in}~L^{\infty}([0,T];\mathbb{R})\,\,\,\,\,\, \tilde{\mathbf{P}}\text{-a.s..}
\end{equation*}
 \end{lemma}

\begin{proof}
 Similar to Lemma \ref{lem4}, we need to show that for any $t\in[0,T]$,
\begin{equation}
\int_0^t |g(s,\tilde{u}_n(s))^*\phi-g(s,\tilde{u}(s))^*\phi|_{l^2}^2ds\to 0~\text{in probability as}~n\to\infty. \label{con3}
\end{equation}

\noindent\textbf{Step 1.} In this step,  we intend to show that  along  a subsequence still denoted by $\{n\}$,
  \begin{equation}\label{es36}
 \lim_{n\to\infty}\|g(t,\tilde{u}_n(\omega,t))-g(t,\tilde{u}(\omega,t))\|_{\mathcal{L}_2(l^2;\mathbb{H}^{1/2})}=0
\end{equation}
holds for $\tilde{\mathbf{P}}\otimes dt$-almost all $(\omega,t)\in\tilde{\Omega}\times [0,T]$.

\vspace{1mm}
Following from (\ref{es80}), we have $\tilde{u}_n\to \tilde{u}$ in $L^2_{w}([0,T];\mathbb{H}^{3/2})$ $\tilde{\mathbf{P}}$-a.s., as $n\to\infty$. Then,
we have $\tilde{u}(\omega,\cdot)\in L^2([0,T];\mathbb{H}^{3/2})$ and the sequence $(\tilde{u}_n(\omega,\cdot))_{n\geq 1}$ is bounded in $L^2([0,T];\mathbb{H}^{3/2})$, as well, for $\tilde{\mathbf{P}}$-almost all $\omega\in\tilde{\Omega}$.
Thus, it is clear that
 \begin{equation}\label{es33}
 \text{the sequence}~(\tilde{u}_n(\omega,t))_{n\geq 1}~\text{is bounded in}~ \mathbb{H}^{3/2},
 \end{equation}
 for $\tilde{\mathbf{P}}\otimes dt$-almost all $(\omega,t)\in\tilde{\Omega}\times [0,T]$.

According to the assumption $(\mathbf{H}_{g}^1)$, by the interpolation inequality we can deduce
\begin{eqnarray}\label{es59}
\!\!\!\!\!\!\!\!&&\|g(t,u)-g(t,v)\|_{\mathcal{L}_2(l^2;\mathbb{H}^{1/2})}^2
\nonumber \\
 \!\!\!\!\!\!\!\!&&\leq
(C+\rho_1(u)+\rho_2(v))|u-v|^{2}_{1/2}+(C+\eta_1(u)+\eta_2(v))|u-v|_{1/2}|u-v|_{3/2}
\nonumber \\
 \!\!\!\!\!\!\!\!&&\lesssim_{|u|_{3/2},|v|_{3/2}}|u-v|^{2}_{1/2}+|u-v|_{1/2}.
\end{eqnarray}
Collecting (\ref{es42}), (\ref{es33}) and  (\ref{es59}),
the convergence (\ref{es36}) holds.

\vspace{1mm}
\noindent\textbf{Step 2.} In this step, we  intend to show that along  a subsequence still denoted by $\{n\}$, the following convergence
\begin{equation}\label{e26}
\int_0^t\|g(s,\tilde{u}_n(\omega,s))\|_{\mathcal{L}_2(l^2;\mathbb{H}^{1/2})}^2ds\to\int_0^t\|g(s,\tilde{u}(\omega,s))\|_{\mathcal{L}_2(l^2;\mathbb{H}^{1/2})}^2ds
\end{equation}
holds for $\tilde{\mathbf{P}}$-almost all $\omega\in\tilde{\Omega}$, as $n\to\infty$.

We set
\begin{eqnarray*}
\!\!\!\!\!\!\!\!&&\varpi_M(t,w):=\int_0^t\|g(s,w(s))\|_{\mathcal{L}_2(l^2;\mathbb{H}^{1/2})}^2\chi_M(|w(s)|_{1/2})ds,
\nonumber \\
\!\!\!\!\!\!\!\!&& \varpi(t,w):=\int_0^t\|g(s,w(s))\|_{\mathcal{L}_2(l^2;\mathbb{H}^{1/2})}^2ds.
\end{eqnarray*}

First, by (\ref{es42}), (\ref{es36}), (\ref{es33}) and the continuity of $\chi_M$,   we have
\begin{eqnarray*}
\lim_{n\to\infty}\!\!\!\!\!\!\!\!&&\Big|\|g(t,\tilde{u}_n(\omega,t))\|_{\mathcal{L}_2(l^2;\mathbb{H}^{1/2})}^2\chi_M(|\tilde{u}_n(\omega,t)|_{1/2})
\nonumber \\
\!\!\!\!\!\!\!\!&&
-
\|g(t,\tilde{u}(\omega,t))\|_{\mathcal{L}_2(l^2;\mathbb{H}^{1/2})}^2\chi_M(|\tilde{u}(\omega,t)|_{1/2})\Big|= 0
\end{eqnarray*}
holds for $\tilde{\mathbf{P}}\otimes dt$-almost all $(\omega,t)\in\tilde{\Omega}\times [0,T]$.

Then, by assumption $(\mathbf{H}_{g}^4)$ and the dominated convergence theorem, it follows that
\begin{equation}\label{e28}
\varpi_M(t,\tilde{u}_n)\to \varpi_M(t,\tilde{u})\,\,\,\,\,\,\tilde{\mathbf{P}}\text{-a.s.},~\text{as}~n\to\infty.
\end{equation}
On the other hand, by the definition of $\chi_R$ we obtain that for any $\varepsilon>0$,
\begin{eqnarray}\label{es60}
\!\!\!\!\!\!\!\!&&\tilde{\mathbf{P}}\Big(|\varpi(t,\tilde{u}_n)-\varpi_M(t,\tilde{u}_n)|>\varepsilon\Big)
\nonumber \\
=\!\!\!\!\!\!\!\!&&\tilde{\mathbf{P}}\Big(|\varpi(t,\tilde{u}_n)-\varpi_M(t,\tilde{u}_n)|>\varepsilon,\sup_{t\in[0,T]}|\tilde{u}_n(t)|_{1/2}\leq M\Big)
\nonumber \\
\!\!\!\!\!\!\!\!&&+ \tilde{\mathbf{P}}\Big(|\varpi(t,\tilde{u}_n)-\varpi_M(t,\tilde{u}_n)|>\varepsilon,\sup_{t\in[0,T]}|\tilde{u}_n(t)|_{1/2}> M\Big)
\nonumber \\
\leq\!\!\!\!\!\!\!\!&& \sup_{n\in\mathbb{N}}\tilde{\mathbf{P}}\Big(\sup_{t\in[0,T]}|\tilde{u}_n(t)|_{1/2}> M\Big).
\end{eqnarray}
Due to  (\ref{apri1}), letting $n\to\infty$ then $M\to \infty$ in (\ref{es60}),  we deduce that
\begin{equation}\label{e29}
|\varpi(t,\tilde{u}_n)-\varpi_M(t,\tilde{u}_n)|\to 0~~\text{in probability}.
\end{equation}
By similar argument, we also obtain
\begin{equation}\label{e31}
|\varpi(t,\tilde{u})-\varpi_M(t,\tilde{u})|\to 0~~\text{in probability}.
\end{equation}
Collecting (\ref{e28}), (\ref{e29}) and (\ref{e31}), we conclude that (\ref{e26}) follows.

\vspace{1mm}
\noindent\textbf{Step 3.} In this step, we prove (\ref{con3}). First, combining (\ref{es36}) and (\ref{e26}), we have
\begin{equation}\label{es37}
\lim_{n\to\infty}\int_0^t\|g(s,\tilde{u}_n(\omega,s))-g(s,\tilde{u}(\omega,s))\|_{\mathcal{L}_2(l^2;\mathbb{H}^{1/2})}^2ds=0
\end{equation}
holds for $\tilde{\mathbf{P}}$-almost all $\omega\in\tilde{\Omega}$.

Note that
\begin{eqnarray*}
\!\!\!\!\!\!\!\!&&\int_0^t |g(s,\tilde{u}_n(s))^*\phi-g(s,\tilde{u}(s))^*\phi|_{l^2}^2ds
\nonumber \\
=\!\!\!\!\!\!\!\!&&\int_0^t \|((g(s,\tilde{u}_n(s))-g(s,\tilde{u}(s)))\cdot,\phi)\|_{\mathcal{L}_2(l^2;\mathbb{R})}^2ds
\nonumber \\
\leq\!\!\!\!\!\!\!\!&&\int_0^t\|g(s,\tilde{u}_n(s))-g(s,\tilde{u}(s))\|_{\mathcal{L}_2(l^2;\mathbb{H}^{1/2})}^2|\phi|_{L^2}^2ds,
\end{eqnarray*}
which combining with (\ref{es37}) yields  (\ref{con3}).

We complete the proof.
\end{proof}

\vspace{1mm}
In the sequel, we present the convergence of remaining terms in Eq.~(\ref{eqns1}).
 \begin{lemma}\label{lem6}
For any $\varphi\in L^{\infty}([0,T]\times \tilde{\Omega};\mathbb{R})$, the following convergence hold $\tilde{\mathbf{P}}$-a.s.

$(i)$ $$\lim_{n\to\infty}\int_0^T(\tilde{u}_n(t)-\tilde{u}(t),\varphi(t)\phi)dt=0;$$

\vspace{1mm}
$(ii)$ $$\lim_{n\to\infty}\int_0^T( \nabla\tilde{u}_n(t)-\nabla\tilde{u}(t),\varphi(t)\nabla\phi)dt=0;$$

\vspace{1mm}
$(iii)$ $$\lim_{n\to\infty}\int_0^T(B( \tilde{u}_n(t))-B( \tilde{u}(t)),\varphi(t)\phi) dt=0.$$
 \end{lemma}

\begin{proof}
Let us fix $\phi\in\mathcal{V}$.   Since by (\ref{es80}), $\tilde{u}_n\to\tilde{u}$  in $L^2_w([0,T];\mathbb{H}^{3/2})$ $\tilde{\mathbf{P}}$-a.s., which implies that the claims (i) and (ii) hold.

Now we focus on (iii). As stated in the proof of Lemma \ref{lem7}, the sequence $(\tilde{u}_n)_{n\geq 1}$ is bounded in $L^2([0,T];\mathbb{H}^{3/2})$
and by  (\ref{es80}) we have $\tilde{u}_n\to\tilde{u}$  in $L^2([0,T];\mathbb{H}^{1/2})$ ~$\tilde{\mathbf{P}}$-a.s.. Moreover, for all $t\in[0,T]$,
\begin{eqnarray*}
\!\!\!\!\!\!\!\!&&\int_0^t(B( \tilde{u}_n(s))-B( \tilde{u}(s)),\varphi(s)\phi) ds
\nonumber \\
=\!\!\!\!\!\!\!\!&&\int_0^t(B( \tilde{u}_n(s)-\tilde{u}(s),\tilde{u}_n(s)),\varphi(s)\phi)ds+   \int_0^t(B( \tilde{u}(s),\tilde{u}_n(s)-\tilde{u}(s)),\varphi(s)\phi) ds
\nonumber \\
\lesssim\!\!\!\!\!\!\!\!&&\big(\|\tilde{u}_n\|_{L^2([0,T];\mathbb{H})}+\|\tilde{u}\|_{L^2([0,T];\mathbb{H})}\big)\|\tilde{u}_n-\tilde{u}\|_{L^2([0,T];\mathbb{H})}|\nabla\phi|_{L^{\infty}}|\varphi|_{L^{\infty}}
\nonumber \\
\lesssim\!\!\!\!\!\!\!\!&&\big(\|\tilde{u}_n\|_{L^2([0,T];\mathbb{H})}+\|\tilde{u}\|_{L^2([0,T];\mathbb{H})}\big)\|\tilde{u}_n-\tilde{u}\|_{L^2([0,T];\mathbb{H})}|\phi|_{\gamma}|\varphi|_{L^{\infty}},
\end{eqnarray*}
where $\gamma>\frac{5}{2}$. Since $\phi\in\mathcal{V}$, we can conclude that the claim (iii) holds.

We complete the proof.
\end{proof}

Now, we have all ingredients to pass to the limit by the following argument. Collecting Lemmas \ref{lem4}-\ref{lem6}, for any $\phi\in\mathcal{V}$, $\varphi\in L^{\infty}([0,T]\times \tilde{\Omega};\mathbb{R})$ we obtain
\begin{eqnarray*}
\!\!\!\!\!\!\!\!&&\int_0^T( \tilde{u}(t),\varphi(t)\phi) dt
\nonumber \\
=\!\!\!\!\!\!\!\!&&\lim_{n\to\infty}\int_0^T( \tilde{u}_{n}(t),\varphi(t)\phi) dt
\nonumber \\
=\!\!\!\!\!\!\!\!&&\lim_{n\to\infty}\Bigg((x_n,\phi) \int_0^T\varphi(t)dt-
\int_0^T\int_0^t( \nabla\tilde{u}_n(s),\varphi(t)\nabla\phi) dsdt
\nonumber \\
\!\!\!\!\!\!\!\!&&-
\int_0^T\int_0^t(B( \tilde{u}_n(s)),\varphi(t)\phi) dsdt
\nonumber \\
\!\!\!\!\!\!\!\!&&
+\int_0^T   (   \int_0^t   \sigma(\tilde{u}_n(s))  d\tilde{\mathcal{W}}_n(s),\varphi(t) \phi)  dt
\nonumber \\
\!\!\!\!\!\!\!\!&&+\int_0^T   (   \int_0^t   g(s,\tilde{u}_n(s))  d\tilde{\hat{\mathcal{W}}}_n(s),\varphi(t) \phi)  dt\Bigg)
\nonumber \\
=\!\!\!\!\!\!\!\!&&\lim_{n\to\infty}\Bigg((x_n,\phi) \int_0^T\varphi(t)dt-
\int_0^T( \nabla\tilde{u}_n(s),\int_s^T\varphi(t)dt\cdot\nabla\phi) ds
\nonumber \\
\!\!\!\!\!\!\!\!&&-
\int_0^T(B( \tilde{u}_n(s)),\int_s^T\varphi(t) dt\cdot\phi) dsdt
\nonumber \\
\!\!\!\!\!\!\!\!&&
+\int_0^T   (   \int_0^t   \sigma(\tilde{u}_n(s))  d\tilde{\mathcal{W}}_n(s),\varphi(t) \phi)  dt
\nonumber \\
\!\!\!\!\!\!\!\!&&+\int_0^T   (   \int_0^t   g(s,\tilde{u}_n(s))  d\tilde{\hat{\mathcal{W}}}_n(s),\varphi(t) \phi)  dt\Bigg)
\nonumber \\
=\!\!\!\!\!\!\!\!&&\int_0^T(x,\phi)\cdot\varphi(t) dt-
\int_0^T \int_0^t(\nabla\tilde{u}(s),\nabla\phi)ds\cdot\varphi(t) dt
\nonumber \\
\!\!\!\!\!\!\!\!&&-
\int_0^T\int_0^t(B( \tilde{u}(s)),\phi)ds\cdot\varphi(t) dt
\nonumber \\
\!\!\!\!\!\!\!\!&&
+\int_0^T      \int_0^t   (\sigma(\tilde{u}(s))  d\tilde{\mathcal{W}}(s), \phi)\cdot\varphi(t)  dt
\nonumber \\
\!\!\!\!\!\!\!\!&&+\int_0^T      \int_0^t   (g(s,\tilde{u}(s))  d\tilde{\hat{\mathcal{W}}}(s), \phi)\cdot\varphi(t)  dt\,\,\,\,\,\,\tilde{\mathbf{P}}\text{-a.s.}.
\end{eqnarray*}
Hence, we can define
\begin{eqnarray}\label{es61}
\bar{u}(t):=\!\!\!\!\!\!\!\!&& x-\int_0^t \big[\mathcal{A}\tilde{u}(s)+B( \tilde{u}(s))\big]ds+ \int_0^t\sigma(\tilde{u}(s))  d\tilde{\mathcal{W}}(s)
\nonumber \\
\!\!\!\!\!\!\!\!&&
+\int_0^tg(s,\tilde{u}(s))  d\tilde{\hat{\mathcal{W}}}(s),~~t\in[0,T].
\end{eqnarray}
It is clear that
\begin{equation}\label{es62}
\tilde{u}=\bar{u}\,\,\,\,\,\,\tilde{\mathbf{P}}\otimes dt\text{-a.e.}.
\end{equation}


  We also derive the following  continuity result.
\begin{lemma}
 \begin{equation}\label{conti2}
 \bar{u}\in \mathbb{C}_T(\mathbb{H}^{-1/2})\,\,\,\,\,\, \tilde{\mathbf{P}}\text{-a.s.}.
  \end{equation}
  \end{lemma}
  \begin{proof}
We define
$$\tilde{\Phi}(t):=-\mathcal{A}\tilde{u}(t)-B( \tilde{u}(t))  $$
and the stopping times
\begin{eqnarray*}
\!\!\!\!\!\!\!\!&&\tilde{\tau}^{\Phi}_R:=\inf\Bigg\{t\in[0,T]:\int_0^t|\tilde{\Phi}(s)|_{-1/2}^{2}ds\geq R\Bigg\}\wedge T,~~R>0,
\nonumber \\
\!\!\!\!\!\!\!\!&&\tilde{\tau}^{\sigma}_R:=\inf\Bigg\{t\in[0,T]:\int_0^t\|\sigma(\tilde{u}(s))\|_{\mathcal{L}_2(l^2;\mathbb{H}^{1/2})}^{2}ds\geq R\Bigg\}\wedge T,~~R>0,
\nonumber \\
\!\!\!\!\!\!\!\!&&\tilde{\tau}^{g}_R:=\inf\Bigg\{t\in[0,T]:\int_0^t\|g(s,\tilde{u}(s))\|_{\mathcal{L}_2(l^2;\mathbb{H}^{1/2})}^{2}ds\geq R\Bigg\}\wedge T,~~R>0.
\end{eqnarray*}
Notice that
\begin{eqnarray}\label{es63}
\!\!\!\!\!\!\!\!&&\int_0^T|\tilde{\Phi}(t)|_{-1/2}^{2}dt
\nonumber \\
\!\!\!\!\!\!\!\!&&\lesssim \int_0^T(1+|\tilde{u}(t)|_{1/2}^{2})(1+|\tilde{u}(t)|_{3/2}^{2})dt
\nonumber \\
\!\!\!\!\!\!\!\!&&
\lesssim\big(1+\|\tilde{u}\|_{L^{\infty}([0,T];\mathbb{H}^{1/2})}^{2}\big)\int_0^T(1+|\tilde{u}(t)|_{3/2}^{2})dt
\nonumber \\
\!\!\!\!\!\!\!\!&&
\lesssim_T1+\|\tilde{u}\|_{L^{\infty}([0,T];\mathbb{H}^{1/2})}^{4}+\bigg(\int_0^T|\tilde{u}(t)|_{3/2}^{2}dt\bigg)^2.
\end{eqnarray}
Similarly, by Hypothesis \ref{h2} and $(\mathbf{H}_{g}^4)$ we have
\begin{eqnarray}\label{es64}
\!\!\!\!\!\!\!\!&&\int_0^T\|\sigma(\tilde{u}(t))\|_{\mathcal{L}_2(l^2;\mathbb{H}^{1/2})}^{2}dt+\int_0^T\|g(t,\tilde{u}(t))\|_{\mathcal{L}_2(l^2;\mathbb{H}^{1/2})}^{2}dt
\nonumber \\
\!\!\!\!\!\!\!\!&&\lesssim \int_0^T(1+|\tilde{u}(t)|_{1/2}^{\beta})(1+|\tilde{u}(t)|_{3/2}^{2})dt
\nonumber \\
\!\!\!\!\!\!\!\!&&
\lesssim\big(1+\|\tilde{u}\|_{L^{\infty}([0,T];\mathbb{H}^{1/2})}^{\beta}\big)\int_0^T(1+|\tilde{u}(t)|_{3/2}^{2})dt
\nonumber \\
\!\!\!\!\!\!\!\!&&
\lesssim_T1+\|\tilde{u}\|_{L^{\infty}([0,T];\mathbb{H}^{1/2})}^{2\beta}+\bigg(\int_0^T|\tilde{u}(t)|_{3/2}^{2}dt\bigg)^2.
\end{eqnarray}
Combining (\ref{es63})-(\ref{es64}) and  utilizing   the estimate (\ref{apri2}), we can see
\begin{equation}\label{es69}
\tilde{\tau}_R:=\tilde{\tau}^{\Phi}_R\wedge\tilde{\tau}^{\sigma}_R\wedge\tilde{\tau}^{g}_R \xrightarrow{R\to\infty} T\,\,\,\,\,\,\tilde{\mathbf{P}}\text{-a.s.}.
\end{equation}
Now, consider the stopped process
\begin{eqnarray}\label{es65}
\bar{u}(t\wedge\tilde{\tau}_R)=\!\!\!\!\!\!\!\!&& x-\int_0^t \mathbf{1}_{\{s\leq\tilde{\tau}_R\} }\big[\mathcal{A}\tilde{u}(s)+B( \tilde{u}(s))\big]ds
\nonumber \\
\!\!\!\!\!\!\!\!&&
+ \int_0^t\mathbf{1}_{\{s\leq\tilde{\tau}_R\} }\sigma(\tilde{u}(s))  d\tilde{\mathcal{W}}(s)
+\int_0^t\mathbf{1}_{\{s\leq\tilde{\tau}_R\} }g(s,\tilde{u}(s))  d\tilde{\hat{\mathcal{W}}}(s).
\end{eqnarray}
It is clear that $\bar{u}(\cdot\wedge\tilde{\tau}_R)\in C([0,T];\mathbb{H}^{-1/2})$, which implies  $\bar{u}\in C([0,\tilde{\tau}_R];\mathbb{H}^{-1/2})$. Taking $R\to \infty$ and by (\ref{es69}), the claim follows.
\end{proof}

\vspace{1mm}
Note that by (\ref{es62}), for any $t\in[0,T]$, we have
\begin{eqnarray}\label{conti1}
\bar{u}(t)=\!\!\!\!\!\!\!\!&& x-\int_0^t \big[\mathcal{A}\bar{u}(s)+B( \bar{u}(s))\big]ds+ \int_0^t\sigma(\bar{u}(s))  d\tilde{\mathcal{W}}(s)
\nonumber \\
\!\!\!\!\!\!\!\!&&
+\int_0^tg(s,\bar{u}(s))  d\tilde{\hat{\mathcal{W}}}(s)\,\,\,\,\,\,\tilde{\mathbf{P}}\text{-a.s.}.
\end{eqnarray}
Consequently, combining (\ref{conti2}) and (\ref{conti1}), we deduce that
$\bar{u}$ satisfies the equality  (\ref{de2}) in Definition  \ref{dew}, i.e.,  $\tilde{\mathbf{P}}$\text{-a.s.}
\begin{eqnarray}\label{es67}
\!\!\!\!\!\!\!\!&&(\bar{u}(t),\phi)+\int_0^t (B(\bar{u}(s)),\phi) ds+\int_0^t (\nabla \bar{u}(s),\nabla\phi)ds
\nonumber \\
\!\!\!\!\!\!\!\!&&=
(x,\phi)+\int_0^t (\sigma(\bar{u}(s))d \tilde{\mathcal{W}}(s),\phi)+\int_0^t (g(s,\bar{u}(s))d \tilde{\hat{\mathcal{W}}}(s),\phi),~~t\in[0,T],
\end{eqnarray}
holds for  all $\phi\in \mathcal{V}$.

\subsection{Proof of Theorem \ref{thg}}\label{sub2.6}
In this part, we first prove the existence of weak solutions in the sense of Definition \ref{dew} to Eq.~(\ref{sns1}) and then prove the pathwise uniqueness, which yield Theorem \ref{thg} in terms of the infinite-dimensional version of Yamada-Watanabe theorem (cf.~\cite{RSZ}).

To establish the existence of weak solutions, in view of (\ref{es67}) it is sufficient  to show that  $\bar{u}$ satisfies the estimate (\ref{esap})  and $\bar{u}\in \mathbb{C}_T(\mathbb{H}^{1/2})$  $\tilde{\mathbf{P}}$-a.s..

\vspace{1mm}
\noindent\textbf{Proof of existence of weak solutions.}   First, we recall
 \begin{equation}\label{conti3}
\tilde{u}\in C((0,T];\mathbb{H}^{-1/2})\,\,\,\,\,\,\tilde{\mathbf{P}}\text{-a.s.}.
  \end{equation}
Combining (\ref{es62}),  (\ref{conti2}) and (\ref{conti3}),  we derive
\begin{equation}\label{es66}
\tilde{u}(t)=\bar{u}(t),~t\in(0,T],~~\tilde{\mathbf{P}}\text{-a.s.}.
\end{equation}
Using the estimate (\ref{apri2}), the equality (\ref{es62}) and (\ref{es66}), it follows that for any $\varepsilon>0$  there exists  $\mathcal{K}>0$  such that
\begin{eqnarray}\label{es68}
\!\!\!\!\!\!\!\!&&\tilde{\mathbf{P}}\Bigg(\sup_{t\in[0,T]}|\bar{u}(t)|_{1/2}^p+\|\bar{u}\|_{L^{2}([0,T];\mathbb{H}^{3/2})}^{2}dt\geq \mathcal{K}\Bigg)
\nonumber \\
\leq\!\!\!\!\!\!\!\!&&\tilde{\mathbf{P}}\Bigg(|x|_{1/2}^p+\sup_{t\in(0,T]}|\tilde{u}(t)|_{1/2}^p+\|\tilde{u}\|_{L^{2}([0,T];\mathbb{H}^{3/2})}^{2}dt\geq \mathcal{K}\Bigg)
\nonumber \\
=\!\!\!\!\!\!\!\!&&\tilde{\mathbf{P}}\Bigg(\sup_{t\in(0,T]}|\tilde{u}(t)|_{1/2}^p+\|\tilde{u}\|_{L^{2}([0,T];\mathbb{H}^{3/2})}^{2}dt\geq \mathcal{K}-|x|_{1/2}^p\Bigg)
\nonumber \\
\leq\!\!\!\!\!\!\!\!&&\varepsilon.
\end{eqnarray}

Next, we turn  to prove
\begin{equation}\label{es41}
\bar{u}\in \mathbb{C}_T(\mathbb{H}^{1/2})\,\,\,\,\,\,\tilde{\mathbf{P}}\text{-a.s.}.
 \end{equation}
We denote $\bar{\tau}_R$  by
\begin{equation*}
\bar{\tau}_R:=\inf\Bigg\{t\in[0,T]:|\bar{u}(t)|_{1/2}+\int_0^t |\bar{u}(s)|_{3/2}^{2}ds\geq R\Bigg\}\wedge T,~~R>0.
\end{equation*}

\begin{lemma}
$\bar{\tau}_R$ is an $(\tilde{\mathscr{F}}_t)$-stopping time.
\end{lemma}

 \begin{proof}
Since the embedding $\mathbb{H}^{1/2}\subset\mathbb{H}^{-1/2}$ is continuous and dense, by (\ref{conti2})  it is known that $\bar{u}$ is weakly continuous in $\mathbb{H}^{1/2}$, so that $|\bar{u}(\cdot)|_{1/2}$ is lower semicontinuous.

Denote
\begin{equation*}
\rho_R:=\inf\Bigg\{t\geq0:|\bar{u}(t)|_{1/2}+\int_0^t |\bar{u}(s)|_{3/2}^{2}ds\geq R\Bigg\},~~R>0.
\end{equation*}
 We claim that for any $t>0$
 \begin{equation*}
 \{\rho_R\geq t\}=\bigcap_{s\in[0,t]}\{|\bar{u}(s)|_{1/2}\leq R\}=\bigcap_{s\in[0,t]\cap\mathbb{Q}}\{|\bar{u}(s)|_{1/2}\leq R\}\in \tilde{\mathscr{F}}_t.
 \end{equation*}
The first equality is straightforward by the definition of $\rho_R$. As for the second equality, we assume $\omega$ belongs to the right hand side, then for any $s\in[0,t]\backslash \mathbb{Q}$, there exists a sequence $(s_k)_{k\in\mathbb{N}}\subset [0,t]\cap\mathbb{Q}$ with $s_k\to s$ such that
$|\bar{u}(\omega,s_k)|_{1/2}\leq R$. By the lower semicontinuity
of $|\bar{u}(\cdot)|_{1/2}$, we have $|\bar{u}(\omega,s)|_{1/2}\leq R$ and $\omega$ belongs to the left hand side as well. It follows that
$$\{\rho_R\leq t\}=\bigcap_{\varepsilon>0}\{\rho_R< t+\varepsilon\}\in\tilde{\mathscr{F}}_{t+}=\tilde{\mathscr{F}}_t.$$
Consequently, $\bar{\tau}_R=\rho_R\wedge T $ is an $(\tilde{\mathscr{F}}_t)$-stopping time.
 \end{proof}

It is clear that
$$\lim_{R\to\infty}\bar{\tau}_R=T\,\,\,\,\,\,\tilde{\mathbf{P}}\text{-a.s.}.$$
Denote
$$ Y(t):=\mathbf{1}_{\{t\leq\bar{\tau}_R \}}\big[\mathcal{A}\bar{u}(t)+B( \bar{u}(t))\big],$$
$$Z_1(t):=\mathbf{1}_{\{t\leq\bar{\tau}_R \}}\sigma(\bar{u}(t)),~Z_2(t):=\mathbf{1}_{\{t\leq\bar{\tau}_R \}}g(t,\bar{u}(t)) .$$
In order to prove (\ref{es41}), we work on the evolution triple $\mathbb{H}^{3/2}\subset\mathbb{H}^{1/2}\subset \mathbb{H}^{-1/2}$ and rewrite the following equality
\begin{equation}\label{eqns2}
\bar{u}(t\wedge \bar{\tau}_R)=x-\int_0^{t\wedge \bar{\tau}_R}Y(s)ds+\int_0^{t\wedge \bar{\tau}_R}Z_1(s)d\tilde{\mathcal{W}}(s)+\int_0^{t\wedge \bar{\tau}_R}Z_2(s)d\tilde{\hat{\mathcal{W}}}(s),~t\in[0,T].
\end{equation}
By (\ref{es63}) and (\ref{es64}),  it is easy to see that
$$\bar{u}(\cdot\wedge \bar{\tau}_R)\mathbf{1}_{\{\cdot\leq\bar{\tau}_R \}}\in L^{2}([0,T]\times \Omega;\mathbb{H}^{3/2}),~Y(\cdot)\in L^{2}([0,T]\times \Omega;\mathbb{H}^{-1/2}),$$
and
$$Z_1(\cdot)\in L^{2}([0,T]\times \Omega;\mathbb{H}^{1/2}),~Z_2(\cdot)\in L^{2}([0,T]\times \Omega;\mathbb{H}^{1/2})$$
Thanks to Proposition 4.2 in \cite{GC1}, we can  deduce that
$\bar{u}\in C([0,\bar{\tau}_R];\mathbb{H}^{1/2})$  $\tilde{\mathbf{P}}$-a.s..
Since $\lim_{R\to\infty}\bar{\tau}_R=T$, it implies that (\ref{es41}) holds.

Hence, we can conclude that $\bar{u}$ is a weak solution of Eq.~(\ref{sns1}). Furthermore,  following from same argument as in Lemma \ref{lem3.0}, it follows that
$$\sup_{t\in[0,T]}\mathbf{E}|\bar{u}(t)|_{1/2}^{2-\gamma}+\mathbf{E}\int_0^{T}|\bar{u}(t)|_{3/2}^{2-\gamma}dt<\infty.$$
We complete the proof of the existence of weak solutions.
\hspace{\fill}$\Box$

\vspace{2mm}
Next, the pathwise uniqueness of Eq.~(\ref{sns1}) is derive by the following.

\vspace{1mm}

\noindent\textbf{Proof of  pathwise uniqueness.}  Let $u,v$ be two solutions of Eq.~(\ref{sns1}) with initial value $u(0)=x\in\mathbb{H}^{1/2},v(0)=y\in\mathbb{H}^{1/2}$. Set $z:=u-v$ that satisfies the following equation
\begin{eqnarray*}
\!\!\!\!\!\!\!\!&&z(t)+\int_0^t \big(   B(u(s))-B(v(s))   \big)ds
+\int_0^t \mathcal{A} z(s)ds
\nonumber \\
=\!\!\!\!\!\!\!\!&&(x-y)+
\int_0^t \big( \sigma(u(s))-\sigma(v(s)) \big)  d \mathcal{W}(s)+\int_0^t \big( g(s,u(s))-g(s,v(s))\big)d \hat{\mathcal{W}}(s).
\end{eqnarray*}
For any $R>0$, we denote the following stopping times
\begin{eqnarray*}
\!\!\!\!\!\!\!\!&&\tau^u_R:=\inf\Bigg\{t\in[0,T]:|u(t)|_{1/2}+\int_0^t |u(s)|_{3/2}^{2}ds\geq R\Bigg\}\wedge T,
\nonumber \\
\!\!\!\!\!\!\!\!&&\tau^v_R:=\inf\Bigg\{t\in[0,T]:|v(t)|_{1/2}+\int_0^t |v(s)|_{3/2}^{2}ds\geq R\Bigg\}\wedge T.
\end{eqnarray*}
Let $\tau_R:=\tau^u_R\wedge\tau^v_R$, it follows that $\lim_{R\to\infty}\tau_R=T$.

\vspace{1mm}
Applying It\^{o}'s formula and using (\ref{sig1}) and $(\mathbf{H}_{g}^1)$, we have
\begin{eqnarray}\label{es70}
\!\!\!\!\!\!\!\!&&\mathbf{E}\Big[\sup_{t\in[0,T\wedge\tau_R]}|z(t)|_{1/2}^2\Big]
\nonumber \\
= \!\!\!\!\!\!\!\!&&|x-y|_{1/2}^2+\mathbf{E}\int_0^{T\wedge\tau_{R} }(-\mathcal{A}z(t),\Lambda z(t))dt
\nonumber \\
\!\!\!\!\!\!\!\!&&
-\mathbf{E}\int_0^{T\wedge\tau_R }(B(u(t))-B(v(t)),\Lambda z(t))dt
\nonumber \\
\!\!\!\!\!\!\!\!&&+\mathbf{E}\int_0^{T\wedge\tau_R }\|\sigma(u(t))-\sigma(v(t))\|_{\mathcal{L}_2(l^2;\mathbb{H}^{1/2})}^2dt
\nonumber \\
\!\!\!\!\!\!\!\!&&+\mathbf{E}\int_0^{T\wedge\tau_R }\|g(t,u(t))-g(t,v(t))\|_{\mathcal{L}_2(l^2;\mathbb{H}^{1/2})}^2dt
\nonumber \\
\!\!\!\!\!\!\!\!&&
+\mathbf{E}\Big[\sup_{t\in[0,T\wedge\tau_R]}\mathcal{M}_1(t)\Big]+\mathbf{E}\Big[\sup_{t\in[0,T\wedge\tau_R]}\mathcal{M}_2(t)\Big]
\nonumber \\
\leq \!\!\!\!\!\!\!\!&&|x-y|_{1/2}^2-\mathbf{E}\int_0^{T\wedge\tau_R }|z(t)|_{3/2}^2dt+\frac{1}{4}\mathbf{E}\int_0^{T\wedge\tau_R }|z(t)|_{3/2}^2dt
\nonumber \\
\!\!\!\!\!\!\!\!&&
+\mathbf{E}\int_0^{T\wedge\tau_R }(B(u(t))-B(v(t)),\Lambda z(t))dt
\nonumber \\
\!\!\!\!\!\!\!\!&&+\mathbf{E}\int_0^{T\wedge\tau_R }\big(C+\rho_1(u(t))+\rho_2(v(t))\big)|z(t)|_{1/2}^2dt
\nonumber \\
\!\!\!\!\!\!\!\!&&+\mathbf{E}\int_0^{T\wedge\tau_R }\big(C+\eta_1(u(t))+\eta_2(v(t))\big)|z(t)|_{1}^2dt
\nonumber \\
\!\!\!\!\!\!\!\!&&
+\mathbf{E}\Big[\sup_{t\in[0,T\wedge\tau_R]}|\mathcal{M}_1(t)|\Big]+\mathbf{E}\Big[\sup_{t\in[0,T\wedge\tau_R]}|\mathcal{M}_2(t)|\Big]
\nonumber \\
=:\!\!\!\!\!\!\!\!&&|x-y|_{1/2}^2-\frac{3}{4}\mathbf{E}\int_0^{T\wedge\tau_R }|z(t)|_{3/2}^2dt+\text{I}+\text{II}+\text{III}+\text{IV}+\text{V},
\end{eqnarray}
where $\mathcal{M}_1(t),\mathcal{M}_2(t) $ are continuous local martingales given by
\begin{eqnarray*}
\!\!\!\!\!\!\!\!&&\mathcal{M}_1(t):=\int_0^{t }(\big(\sigma(u(s))-\sigma(v(s))\big)d\mathcal{W}(s),z(s))_{1/2}ds,
\nonumber \\
\!\!\!\!\!\!\!\!&&
\mathcal{M}_2(t):=\int_0^{t }(\big(g(s,u(s))-g(s,v(s))\big)d\hat{\mathcal{W}}(s),z(s))_{1/2}ds.
\end{eqnarray*}

Note that  using H\"{o}lder's inequality and Sobolev embedding theorem,  we have
\begin{eqnarray}\label{es71}
\text{I}
\leq\!\!\!\!\!\!\!\!&&\mathbf{E}\int_0^{T\wedge\tau_R }|\langle\Lambda^{1/2}\big((u(t)\cdot\nabla)z(t)\big),\Lambda^{1/2} z(t)\rangle|dt
\nonumber \\
\!\!\!\!\!\!\!\!&&
+\mathbf{E}\int_0^{T\wedge\tau_R }|(B(z(t),v(t)),\Lambda z(t))|dt
\nonumber \\
=: \!\!\!\!\!\!\!\!&&\text{I}_1+\text{I}_2.
\end{eqnarray}
For the term $\text{I}_1$,  due to $\langle u\cdot\nabla\Lambda^{1/2}z,\Lambda^{1/2} z\rangle=0 $, making use of the  commutator estimate (\ref{escom}) with $p_1=3,p_2=6,p_3=6,p_4=3$ we deduce that
\begin{eqnarray}\label{es72}
\text{I}_1
\leq \!\!\!\!\!\!\!\!&&
\mathbf{E}\int_0^{T\wedge\tau_R }|\langle\Lambda^{1/2}\big((u(t)\cdot\nabla)z(t)\big),\Lambda^{1/2} z(t)\rangle-\langle (u(t)\cdot\nabla)\Lambda^{1/2}z(t),\Lambda^{1/2} z(t)\rangle|dt
\nonumber \\
=\!\!\!\!\!\!\!\!&&\mathbf{E}\int_0^{T\wedge\tau_R }|\langle[\Lambda^{1/2},u(t)]\cdot\nabla z(t),\Lambda^{1/2} z(t)\rangle|dt
\nonumber \\
\leq \!\!\!\!\!\!\!\!&&\mathbf{E}\int_0^{T\wedge\tau_R }|[\Lambda^{1/2},u(t)]\cdot\nabla z(t)|_{L^{2}}|\Lambda^{1/2} z(t)|_{L^2}dt
\nonumber \\
\leq\!\!\!\!\!\!\!\!&&\mathbf{E}\int_0^{T\wedge\tau_R }\big(|\nabla u(t)|_{L^3}|  \Lambda^{-1/2}\nabla z(t)|_{L^6}+|\Lambda^{1/2}u(t)|_{L^6}|\nabla z(t)|_{L^3}\big) |\Lambda^{1/2} z(t)|_{L^2}dt
\nonumber \\
\leq \!\!\!\!\!\!\!\!&&\epsilon_0\mathbf{E}\int_0^{T\wedge\tau_R }| z(t)|_{3/2}^2dt+ C_{\epsilon_0}\mathbf{E}\int_0^{T\wedge\tau_R }| u(t)|_{3/2}^2  |z(t)|_{1/2}^2dt,
\end{eqnarray}
where $\epsilon_0>0$ is a small constant that will be chosen later, and we used Sobolev embedding inequality and Young' inequality in the last step.

For the term $\text{I}_2$, we use the estimate (\ref{2.2}) to derive
\begin{eqnarray}\label{es73}
\text{I}_2\leq\!\!\!\!\!\!\!\!&&\mathbf{E}\int_0^{T\wedge\tau_R }|z(t)|_{1/2}|v(t)|_{3/2}|\Lambda z(t)|_{1/2}dt
\nonumber \\
 \leq\!\!\!\!\!\!\!\!&&\epsilon_0\mathbf{E}\int_0^{T\wedge\tau_R }|z(t)|_{3/2}^2dt+C_{\epsilon_0}\mathbf{E}\int_0^{T\wedge\tau_R }|v(t)|_{3/2}^2| z(t)|_{1/2}^2dt      .
\end{eqnarray}
For the term $\text{III}$, by (\ref{es0}), (\ref{inteq2}) and Young's inequality it follows that
\begin{eqnarray}\label{es74}
\text{III}
\lesssim \!\!\!\!\!\!\!\!&&\mathbf{E}\int_0^{T\wedge\tau_R }\big(C+\eta_1(u(t))+\eta_2(v(t))\big)|z(t)|_{1/2}|z(t)|_{3/2}dt
 \nonumber \\
\leq \!\!\!\!\!\!\!\!&&C_{R,\epsilon_0}\mathbf{E}\int_0^{T\wedge\tau_R }\big(1+|u(t)|_{3/2}^2+|v(t)|_{3/2}^2\big)|z(t)|_{1/2}^2dt
\nonumber \\
 \!\!\!\!\!\!\!\!&&
+ \epsilon_0\mathbf{E}\int_0^{T\wedge\tau_R }|z(t)|_{3/2}^2dt.
\end{eqnarray}

Moreover, using B-D-G's inequality leads to
\begin{eqnarray}\label{es75}
\text{IV}
\leq\!\!\!\!\!\!\!\!&&
\mathbf{E}\Bigg(\int_0^{T\wedge\tau_R}\|\sigma(u(t))-\sigma(v(t))\|_{\mathcal{L}_2(l^2;\mathbb{H}^{1/2})}^2
|z(t)|_{1/2}^2dt\Bigg)^{\frac{1}{2}}
\nonumber \\
 \leq\!\!\!\!\!\!\!\!&&\frac{1}{4}\mathbf{E}\Big[\sup_{t\in[0,T\wedge\tau_R]}|z(t)|_{1/2}^2\Big]
 +\frac{1}{4}\mathbf{E}\int_0^{T\wedge\tau_R }|z(t)|_{3/2}^2dt
 \nonumber \\
\!\!\!\!\!\!\!\!&&+C\mathbf{E}\int_0^{T\wedge\tau_R }|z(t)|_{1/2}^2dt
\end{eqnarray}
and
\begin{eqnarray}\label{es755}
\text{V}
\leq\!\!\!\!\!\!\!\!&&
\mathbf{E}\Bigg(\int_0^{T\wedge\tau_R}\|g(t,u(t))-g(t,v(t))\|_{\mathcal{L}_2(l^2;\mathbb{H}^{1/2})}^2
|z(t)|_{1/2}^2dt\Bigg)^{\frac{1}{2}}
\nonumber \\
 \leq\!\!\!\!\!\!\!\!&&\frac{1}{2}\mathbf{E}\Big[\sup_{t\in[0,T\wedge\tau_R]}|z(t)|_{1/2}^2\Big]
\nonumber \\
\!\!\!\!\!\!\!\!&&
+\mathbf{E}\int_0^{T\wedge\tau_R }\big(C+\rho_1(u(t))+\rho_2(v(t))\big)|z(t)|_{1/2}^2dt
\nonumber \\
\!\!\!\!\!\!\!\!&&+\mathbf{E}\int_0^{T\wedge\tau_R }\big(C+\eta_1(u(t))+\eta_2(v(t))\big)|z(t)|_{1}^2dt
\nonumber \\
 \leq\!\!\!\!\!\!\!\!&&\frac{1}{4}\mathbf{E}\Big[\sup_{t\in[0,T\wedge\tau_R]}|z(t)|_{1/2}^2\Big]+ \epsilon_0\mathbf{E}\int_0^{T\wedge\tau_R }|z(t)|_{3/2}^2dt
\nonumber \\
\!\!\!\!\!\!\!\!&&+
C_{R,\epsilon_0}\mathbf{E}\int_0^{T\wedge\tau_R }\big(1+|u(t)|_{3/2}^2+|v(t)|_{3/2}^2\big)|z(t)|_{1/2}^2dt  .
\end{eqnarray}

Collecting estimates (\ref{es70})-(\ref{es75}) and  taking $\epsilon_0<\frac{1}{8}$, we can get
\begin{eqnarray*}
\!\!\!\!\!\!\!\!&&\mathbf{E}\Big[\sup_{t\in[0,T\wedge\tau_R]}|z(t)|_{1/2}^2\Big]
+\mathbf{E}\int_0^{T\wedge\tau_R }|z(t)|_{3/2}^2dt
\nonumber \\
\leq\!\!\!\!\!\!\!\!&&C|x-y|_{1/2}^2+C_R\mathbf{E}\int_0^{T\wedge\tau_R }\big(1+|u(t)|_{3/2}^2+|v(t)|_{3/2}^2\big)|z(t)|_{1/2}^2dt.
\end{eqnarray*}
By the definition of stopping time $\tau_R$ and the  stochastic Gronwall's lemma (cf.~Lemma \ref{appen2} in Appendix), we obtain
\begin{equation}\label{es76}
\mathbf{E}\Big[\sup_{t\in[0,T\wedge\tau_R]}|z(t)|_{1/2}^2\Big]+\mathbf{E}\int_0^{T\wedge\tau_R }|z(t)|_{3/2}^2dt
\lesssim_R|x-y|_{1/2}^2.
\end{equation}
Consequently, taking $x=y$ and applying Fatou's lemma, we derive
\begin{equation*}
\mathbf{E}\Big[\sup_{t\in[0,T]}|z(t)|_{1/2}^2\Big]\leq \liminf_{R\to\infty}\mathbf{E}\Big[\sup_{t\in[0,T\wedge\tau_R]}|z(t)|_{1/2}^2\Big]
\leq 0,
\end{equation*}
which yields the pathwise uniqueness of Eq.~(\ref{sns1}). The proof is complete. \hspace{\fill}$\Box$

\subsection{Proof of Theorem \ref{thF}}\label{sub2.7}
Define the stopping time
\begin{eqnarray*}
\tau_{n,R}:=\!\!\!\!\!\!\!\!&&\inf\Bigg\{t\geq0:|u(t,x_n)|_{1/2}+|u(t,x)|_{1/2}
\nonumber \\
\!\!\!\!\!\!\!\!&&
+\int_0^t |u(s,x_n)|_{3/2}^{2}ds+\int_0^t |u(s,x)|_{3/2}^{2}ds\geq R\Bigg\},~~R>0,
\end{eqnarray*}
with the convention $\inf\emptyset=\infty$. In light of the estimate (\ref{esap}) and the convergence of $x_n$ by the assumption, we can deduce that
\begin{equation}\label{con4}
\lim_{R\to\infty}\sup_{n\in\mathbb{N}}\mathbf{P}(\tau_{n,R}<T)=0.
\end{equation}

Following same argument as in the proof of (\ref{es76}),  we also derive
\begin{eqnarray}\label{es27}
\!\!\!\!\!\!\!\!&&\mathbf{E}\Big[\sup_{t\in[0,T\wedge\tau_{n,R}]}|u(t,x_n)-u(t,x)|_{1/2}^2\Big]+\mathbf{E}\int_0^{T\wedge\tau_{n,R} }|u(t,x_n)-u(t,x)|_{3/2}^2dt
\nonumber \\
\!\!\!\!\!\!\!\!&&\lesssim_R|x_n-x|_{1/2}^2.
\end{eqnarray}
Therefore, for any $\varepsilon>0$,
\begin{eqnarray}\label{es23}
\!\!\!\!\!\!\!\!&&\mathbf{P}\Big(\sup_{t\in[0,T]}|u(t,x_n)-u(t,x)|_{1/2}>\varepsilon\Big)
\nonumber \\
\leq\!\!\!\!\!\!\!\!&&\mathbf{P}\Big(\sup_{t\in[0,T\wedge\tau_{n,R}]}|u(t,x_n)-u(t,x)|_{1/2}>\varepsilon\Big)+\mathbf{P}(\tau_{n,R}<T)
\nonumber \\
\leq\!\!\!\!\!\!\!\!&&\frac{C_R|x_n-x|_{1/2}^2}{\varepsilon^2}+\mathbf{P}(\tau_{n,R}<T).
\end{eqnarray}
In view of (\ref{es27}) we also have
\begin{equation}\label{es28}
\mathbf{P}\Bigg(\int_0^{T}|u(t,x_n)-u(t,x)|_{3/2}^2dt>\varepsilon\Bigg)
\leq\frac{C_R|x_n-x|_{1/2}^2}{\varepsilon}+\mathbf{P}(\tau_{n,R}<T).
\end{equation}
Hence, combining (\ref{con4}) and (\ref{es23})-(\ref{es28}) and letting $n\uparrow\infty$ then $R\uparrow\infty$, we conclude that (\ref{apri3}) follows.

\vspace{1mm}
From now on, we assume that the map $g(t,u)$ is independent of $t$. We turn to prove that the function $(\mathcal{T}_t)_{t\geq 0}$ is Feller. More precisely, for any $t\geq 0$ and $C_b(\mathbb{H}^{1/2})$ we shall show
\begin{equation}\label{feller}
\mathcal{T}_t\varphi(x_n)=\mathbf{E}\varphi(u(t,x_n))\to \mathbf{E}\varphi(u(t,x))=\mathcal{T}_t\varphi(x)~~ \text{if}~ x_n\to x~\text{in}~ \mathbb{H}^{1/2}.
\end{equation}
We only need to consider $\varphi\in \text{Lip}_b(\mathbb{H}^{1/2})$ due to the fact
 $\text{Lip}_b(\mathbb{H}^{1/2})\subset C_b(\mathbb{H}^{1/2})$ densely. By (\ref{apri3}), it follows that for any $\varepsilon>0$,
\begin{equation*}
\lim_{n\to\infty}\mathbf{P}\Big(|\varphi(u(t,x_n))-\varphi(u(t,x))|>\varepsilon\Big)
\leq\lim_{n\to\infty}\mathbf{P}\Big(C_{\text{Lip}}|u(t,x_n)-u(t,x)|_{1/2}>\varepsilon\Big)=0,
\end{equation*}
where $C_{\text{Lip}}$ is the Lipschitz constant of function $\varphi$.
Hence,
$$\varphi(u(t,x_n))\to \varphi(u(t,x))~\text{in probability},~\text{as}~n\to\infty.$$
Since the function $\varphi$ is bounded, then (\ref{feller}) follows directly from the Lebesgue dominated convergence theorem.

In order to prove the Markov property (\ref{markov}), we shall prove
\begin{equation}\label{es78}
\mathbf{E}\big[\varphi(u(t+s,x))\phi\big]=\mathbf{E}\big[\mathcal{T}_s\varphi(u(t,x))\phi\big]~~~\text{for any}~\phi\in\mathscr{F}_t.
\end{equation}
By the uniqueness of solutions to (\ref{sns1}),  we know
$$u(t+s,x)=u(t,t+s,u(t,x))\,\,\,\,\,\,\mathbf{P}\text{-a.s.},$$
where we denote by $u(t,t+s,u(t,x))$ the solution of (\ref{sns1}) with initial time $t$ and initial value $u(t,x)$.

Thus, in order to prove (\ref{es78}), it suffices to show
\begin{equation}\label{es79}
\mathbf{E}\big[\varphi(u(t,t+s,Z))\phi\big]=\mathbf{E}\big[\mathcal{T}_s\varphi(Z)\phi\big]
\end{equation}
holds for every $(\mathscr{F}_t)$-measurable random variable $Z$. Note that by standard  approximation procedure, it is enough to prove (\ref{es79}) for random variables $Z=\sum_{i=1}^nZ_i\mathbf{1}_{A_i}$, where  $Z_i\in\mathbb{H}^{1/2}$ is deterministic and $(A_i)\subset\mathscr{F}_t$
 is a collection of disjoint sets such that $\bigcup_{i}A_i=\Omega$.  Then it is enough to prove (\ref{es79}) for every deterministic $Z\in\mathbb{H}^{1/2} $.  Note that, in this case, $\varphi(u(t,t+s,Z))$ depends only on  the increments of the Wiener process between $t$ and $t+s$, which is independent of $\mathscr{F}_t$. It follows that
\begin{equation*}
\mathbf{E}\big[\varphi(u(t,t+s,Z))\phi\big]=\mathbf{E}\big[ \varphi(u(t,t+s,Z))\big]\mathbf{E}\big[\phi\big].
\end{equation*}
Since $u(t,t+s,Z)$ coincides in law with $u(s,Z)$ by uniqueness,  then we deduce that
\begin{equation*}
\mathbf{E}\big[\varphi(u(t,t+s,Z))\phi\big]=\mathbf{E}\big[ \varphi(u(s,Z))\big]\mathbf{E}\big[\phi\big]=\mathcal{T}_s\varphi(Z)\mathbf{E}\big[\phi\big]=\mathbf{E}\big[\mathcal{T}_s\varphi(Z)\phi\big],
\end{equation*}
which completes the proof of (\ref{es79}).

Finally, taking expectation on both sides of (\ref{markov}), we have
$$\mathbf{E}\big[\mathbf{E}[\varphi(u(t+s,x))|\mathscr{F}_t]\big]=\mathbf{E}[\varphi(u(t+s,x))]=\mathcal{T}_{t+s}\varphi(x)$$
and on the other hand,
$$\mathbf{E}\big[(\mathcal{T}_s\varphi)(u(t,x))\big]=(\mathcal{T}_{t}(\mathcal{T}_s\varphi))(x).$$
Therefore, the semigroup property follows. \hspace{\fill}$\Box$

\section{Long-time behaviour}\label{sec4}

\subsection{Main results}\label{subsre2}

 In this section, building upon the global well-posedness result of Eq.~(\ref{sns1}), we intend to investigate the long-time behaviour  of  stochastic forced 3D Navier-Stokes equations.

To be more precise, we consider the following  3D Navier-Stokes systems perturbed by the autonomous stochastic forcing
\begin{equation}\label{sns2}
\begin{cases}
&\!\!\!\!\!\!du(t)+[\mathcal{A} u(t) +B( u(t))] dt =\sigma(u(t))d\mathcal{W}(t) +    g(u(t)) d\hat{\mathcal{W}}(t) , \\
&\!\!\!\!\!\!u(0) = x.
\end{cases}
\end{equation}

We first present the following assumption, which is stronger than $(\mathbf{H}_{g}^{2})$, in order to get the decay estimates associated to the  stochastic forced  3D Navier-Stokes equations.
\begin{enumerate}
\item [$(\mathbf{H}_{g}^{2*})$]   There exists a constant  $\gamma\in(1,2)$  such that for any $u\in\mathbb{H}^{3/2}$,
\begin{equation}\label{cong3}
\kappa_1(|u|_{1}^{4}+1)|u|_{1/2}^4+\|g(u)\|_{\mathcal{L}_2(l^2;\mathbb{H}^{1/2})}^2|u|_{1/2}^2
\leq \gamma\|(g(u)\cdot,u)_{1/2}\|_{\mathcal{L}_2(l^2;\mathbb{R})}^2,
\end{equation}
where $\kappa_1$ is the same as in $(\mathbf{H}_{g}^{2})$.
\end{enumerate}

\begin{remark}\label{reb3}
We point out that following from the proof in Remark \ref{reb2},  the nonlocal stochastic forcing $(\ref{exaB1})$ also satisfies the assumption $(\mathbf{H}_{g}^{2*})$.

\end{remark}

We derive a crucial result  characterizing the decay of solutions to Eq.~(\ref{sns2}).
\begin{theorem}\label{th4}
Suppose that  Hypothesis \ref{h2}-\ref{h3}   hold with $(\mathbf{H}_{g}^{2})$ replaced by $(\mathbf{H}_{g}^{2*})$.  There exist a constant $\kappa\in (0,\lambda^*(2-\gamma))$, where $\lambda^*$ is a positive constant from the Poincar\'{e} inequality,  and an $\mathbf{P}$-a.s. finite random time $\tau$ such that
  \begin{equation*}
|u(t)|_{1/2}^{2-\gamma}\leq e^{-\kappa t}|x|_{1/2}^{2-\gamma},~~t\geq \tau,~~~~~~\mathbf{P}\text{-a.s.}.
\end{equation*}

%

\end{theorem}

\begin{remark}
Regarding Theorem \ref{th4}, we observe that the solutions of the 3D Navier-Stokes equations perturbed by transport noise and nonlocal stochastic forcing decay exponentially to zero as time tends to infinity. This seems to be  the first characterization in the literature concerning the decay rate of solutions for the stochastic forced 3D Navier-Stokes equations.
\end{remark}

With the help of Theorem \ref{th4}, we can investigate the ergodicity of stochastic 3D NS system.   We recall the definition of the invariant measures associated to  $(\mathcal{T}_t)_{t\geq 0}$.
\begin{definition}
A probability measure $\mu$ on $\mathbb{H}^{1/2}$ is called an invariant measure associated to  $(\mathcal{T}_t)_{t\geq 0}$,  if
$$\int_{\mathbb{H}^{1/2}}\varphi(x)\mu(dx)=\int_{\mathbb{H}^{1/2}}\mathcal{T}_t\varphi(x)\mu(dx),~t\geq 0,~\varphi\in\mathscr{B}_b(\mathbb{H}^{1/2}).$$
\end{definition}

Let us  state  the existence, uniqueness and the concentration property of invariant measures  to stochastic forced 3D Navier-Stokes equations (\ref{sns2}).
\begin{theorem}\label{th3}
Suppose that Hypothesis \ref{h2}-\ref{h3}   hold.
Then
there exists an  invariant measure $\mu$ associated to the  transition semigroup $(\mathcal{T}_t)_{t\geq 0}$ of Eq.~$(\ref{sns2})$, which satisfies the following concentration property
\begin{equation}\label{sup1}
\int_{\mathbb{H}^{3/2}}|x|_{3/2}^{2-\gamma}\mu(dx)<\infty,
\end{equation}
where the constant $\gamma\in(1,2)$ is the same as in $(\mathbf{H}_{g}^2)$.

Furthermore, if the assumption $(\mathbf{H}_{g}^{2})$ is replaced by $(\mathbf{H}_{g}^{2*})$, then there exists a unique invariant measure.
\end{theorem}

\begin{remark}
(i) As far as we know, there are few results regarding the ergodicity for the stochastic forced 3D Navier-Stokes equations. An important work was presented by  Da Prato and Debussche in \cite{DD03}, where they investigated the asymptotic properties of the 3D  Navier-Stokes equations perturbed by additive noise. Owing to the lack of uniqueness of solutions, they characterized the ergodicity for the Markov selection semigroup.

(ii) Theorem \ref{th3} reveals that, with a suitable multiplicative noise, one can establish the ergodicity for the stochastic forced 3D Navier-Stokes equations, rather than relying on Markov selection.

\end{remark}

\subsection{Proof of decay estimates}
The proof of Theorem \ref{th4} is divided into the following two steps.

\vspace{1mm}
\noindent\textbf{Step 1.} We claim that  the process
$$\big\{e^{\lambda^*(1-\frac{\gamma}{2}) t}|u(t)|_{1/2}^{2-\gamma}\big\}_{t\geq 0}$$
is a non-negative supermartingale, i.e.
$$\mathbf{E}\Big[e^{\lambda^*(1-\frac{\gamma}{2})t}|u(t)|_{1/2}^{2-\gamma}|\mathscr{F}_{r}\Big]\leq e^{\lambda^*(1-\frac{\gamma}{2})r}|u(r)|_{1/2}^{2-\gamma},~r<t.$$

Recall the equality (\ref{es5}), we deduce that
\begin{eqnarray*}
\!\!\!\!\!\!\!\!&&d\Phi^{\varepsilon}(|u(t)|_{1/2}^2)
\nonumber \\
=\!\!\!\!\!\!\!\!&&2\alpha (\varepsilon+|u(t)|_{1/2}^2)^{\alpha-1}\big[(u(t),\sigma(u(t))d\mathcal{W}(t))_{1/2}+(u(t),g(u(t))d\hat{\mathcal{W}}(t))_{1/2}\big]
\nonumber \\
\!\!\!\!\!\!\!\!&&
+\alpha(\varepsilon+|u(t)|_{1/2}^2)^{\alpha-1}\Big[-2|u(t)|_{3/2}^2-2(B(u(t)),\Lambda u(t))_{L^2}
\nonumber \\
\!\!\!\!\!\!\!\!&&
+\|\sigma(u(t))\|_{\mathcal{L}_2(l^2;\mathbb{H}^{1/2})}^2+\|g(u(t))\|_{\mathcal{L}_2(l^2;\mathbb{H}^{1/2})}^2\Big]dt
\nonumber \\
\!\!\!\!\!\!\!\!&&
-2\alpha(1-\alpha)(\varepsilon+|u(t)|_{1/2}^2)^{\alpha-2}\|((\sigma(u(t)))\cdot,u(t))_{1/2}\|_{\mathcal{L}_2(l^2;\mathbb{R})}^2dt
\nonumber \\
\!\!\!\!\!\!\!\!&&
-2\alpha(1-\alpha)(\varepsilon+|u(t)|_{1/2}^2)^{\alpha-2}\|((g(u(t)))\cdot,u(t))_{1/2}\|_{\mathcal{L}_2(l^2;\mathbb{R})}^2dt.
\end{eqnarray*}
Furthermore, applying the product rule   to the  function $\tilde{\Phi}^{\varepsilon}(t,x):=e^{c_0t}\Phi^{\varepsilon}(x)$, where the positive constant $c_0$ will be chosen later, we derive
\begin{eqnarray*}
\!\!\!\!\!\!\!\!&&d\tilde{\Phi}^{\varepsilon}(t,|u(t)|_{1/2}^2)
\nonumber \\
=\!\!\!\!\!\!\!\!&&c_0\tilde{\Phi}^{\varepsilon}(t,|u(t)|_{1/2}^2)   dt+   2\alpha e^{c_0t} (\varepsilon+|u(t)|_{1/2}^2)^{\alpha-1}\big[(u(t),\sigma(u(t))d\mathcal{W}(t))_{1/2}
\nonumber \\
\!\!\!\!\!\!\!\!&&
+(u(t),g(u(t))d\hat{\mathcal{W}}(t))_{1/2}\big]
\nonumber \\
\!\!\!\!\!\!\!\!&&
+\alpha e^{c_0t}(\varepsilon+|u(t)|_{1/2}^2)^{\alpha-1}\Big[-2|u(t)|_{3/2}^2-2(B(u(t)),\Lambda u(t))_{L^2}
\nonumber \\
\!\!\!\!\!\!\!\!&&
+\|\sigma(u(t))\|_{\mathcal{L}_2(l^2;\mathbb{H}^{1/2})}^2+\|g(u(t))\|_{\mathcal{L}_2(l^2;\mathbb{H}^{1/2})}^2\Big]dt
\nonumber \\
\!\!\!\!\!\!\!\!&&
-2\alpha(1-\alpha)e^{c_0t}(\varepsilon+|u(t)|_{1/2}^2)^{\alpha-2}\|((\sigma(u(t)))\cdot,u(t))_{1/2}\|_{\mathcal{L}_2(l^2;\mathbb{R})}^2dt
\nonumber \\
\!\!\!\!\!\!\!\!&&
-2\alpha(1-\alpha)e^{c_0t}(\varepsilon+|u(t)|_{1/2}^2)^{\alpha-2}\|((g(u(t)))\cdot,u(t))_{1/2}\|_{\mathcal{L}_2(l^2;\mathbb{R})}^2dt
.
\end{eqnarray*}
Following same calculations in the proof of the inequality (\ref{es18}) and applying the Poincar\'{e} inequality,  due to the assumption $g(0)=0$, there exists  a positive constant $\lambda^*$ such that
\begin{eqnarray}\label{es577}
\!\!\!\!\!\!\!\!&&\tilde{\Phi}^{\varepsilon}(t,|u(t)|_{1/2}^2)
\nonumber \\
\leq\!\!\!\!\!\!\!\!&&\tilde{\Phi}^{\varepsilon}(r,|u(r)|_{1/2}^2) +2\alpha\big[\mathcal{M}_{\varepsilon}^1(r,t)+\mathcal{M}_{\varepsilon}^2(r,t)\big]+c_0\int_r^t\tilde{\Phi}^{\varepsilon}(s,|u(s)|_{1/2}^2) ds
\nonumber \\
\!\!\!\!\!\!\!\!&&
-\alpha\int_r^te^{c_0s}\frac{\lambda^*|u(s)|_{1/2}^2}{(\varepsilon+|u(s)|_{1/2}^2)^{1-\alpha}}\mathbf{1}_{\{|u(s)|_{1/2}>0\}}ds
\nonumber \\
\!\!\!\!\!\!\!\!&&
-\alpha\int_r^te^{c_0s}\frac{|u(s)|_{3/2}^2}{(\varepsilon+|u(s)|_{1/2}^2)^{1-\alpha}}ds+\alpha\int_r^te^{c_0s}\frac{\|\sigma(u(s))\|_{\mathcal{L}_2(l^2;\mathbb{H}^{1/2})}^2}{(\varepsilon+|u(s)|_{1/2}^2)^{1-\alpha}}ds
\nonumber \\
\!\!\!\!\!\!\!\!&&
+\alpha\int_r^te^{c_0s}\Bigg\{\frac{\big(\delta_1|u(s)|_{1}^{4}|u(s)|_{1/2}^{2}+\|g(u(s))\|_{\mathcal{L}_2(l^2;\mathbb{H}^{1/2})}^2\big)(\varepsilon+|u(s)|_{1/2}^2)}{(\varepsilon+|u(s)|_{1/2}^2)^{2-\alpha}}\Bigg\}
\mathbf{1}_{\{|u(s)|_{1/2}>0\}}ds
\nonumber \\
\!\!\!\!\!\!\!\!&&
-2\alpha(1-\alpha)\int_r^te^{c_0s}\frac{\|((g(u(t)))\cdot,u(t))_{1/2}\|_{\mathcal{L}_2(l^2;\mathbb{R})}^2}{(\varepsilon+|u(s)|_{1/2}^2)^{2-\alpha}}\mathbf{1}_{\{|u(s)|_{1/2}>0\}}ds
\nonumber \\
\leq\!\!\!\!\!\!\!\!&&\tilde{\Phi}^{\varepsilon}(r,|u(r)|_{1/2}^2) +2\alpha\big[\mathcal{M}_{\varepsilon}^1(r,t)+\mathcal{M}_{\varepsilon}^2(r,t)\big]+c_0\int_r^t\tilde{\Phi}^{\varepsilon}(s,|u(s)|_{1/2}^2) ds
\nonumber \\
\!\!\!\!\!\!\!\!&&
-\alpha\int_r^te^{c_0s}\frac{\lambda^*|u(s)|_{1/2}^2}{(\varepsilon+|u(s)|_{1/2}^2)^{1-\alpha}}\mathbf{1}_{\{|u(s)|_{1/2}>0\}}ds
\nonumber \\
\!\!\!\!\!\!\!\!&&
+\alpha\int_r^te^{c_0s}\Bigg\{\frac{\kappa_1\big((|u(s)|_{1}^{4}+1)|u(s)|_{1/2}^{2}+\|g(u(s))\|_{\mathcal{L}_2(l^2;\mathbb{H}^{1/2})}^2\big)(\varepsilon+|u(s)|_{1/2}^2)}{(\varepsilon+|u(s)|_{1/2}^2)^{2-\alpha}}
\nonumber \\
\!\!\!\!\!\!\!\!&&
-\frac{2(1-\alpha)\|((g(u(t)))\cdot,u(t))_{1/2}\|_{\mathcal{L}_2(l^2;\mathbb{R})}^2}{(\varepsilon+|u(s)|_{1/2}^2)^{2-\alpha}}\Bigg\}\mathbf{1}_{\{|u(s)|_{1/2}>0\}}ds
,~0\leq r<t,
\end{eqnarray}
where $\lambda^*$ is a positive constant from the Poincar\'{e} inequality.
Here $\mathcal{M}_{\varepsilon}^1(r,t)$, $\mathcal{M}_{\varepsilon}^2(r,t)$ are  continuous local martingales given by
\begin{equation*}
\mathcal{M}_{\varepsilon}^1(r,t):=\int_r^t\mathbf{1}_{\{|u(s)|_{1/2}>0\}}e^{c_0s}\frac{(\sigma(u(s))d\mathcal{W}(s),u(s))_{1/2}}{(\varepsilon+|u(s)|_{1/2}^2)^{1-\alpha}}
\end{equation*}
and
\begin{equation*}
\mathcal{M}_{\varepsilon}^2(r,t):=\int_r^t\mathbf{1}_{\{|u(s)|_{1/2}>0\}}e^{c_0s}\frac{(g(u(s))d\hat{\mathcal{W}}(s),u(s))_{1/2}}{(\varepsilon+|u(s)|_{1/2}^2)^{1-\alpha}}.
\end{equation*}

Let us denote $\tilde{\Phi}(t,x):=e^{c_0t}x^{2-\gamma}$. Taking  $\varepsilon\to0$ in (\ref{es577}) and $\alpha=1-\frac{\gamma}{2}$, by condition $(\mathbf{H}_{g}^{2*})$  we have for all $0\leq r<t$,
\begin{eqnarray}\label{es588}
\!\!\!\!\!\!\!\!&&\tilde{\Phi}(t,|u(t)|_{1/2})
\nonumber \\
\leq\!\!\!\!\!\!\!\!&&\tilde{\Phi}(r,|u(r)|_{1/2}) +2\alpha\big[\tilde{\mathcal{M}}^1(r,t)+\tilde{\mathcal{M}}^2(r,t)\big]+c_0\int_r^t\tilde{\Phi}(s,|u(s)|_{1/2})\mathbf{1}_{\{|u(s)|_{1/2}>0\}}  ds
\nonumber \\
\!\!\!\!\!\!\!\!&&
-\alpha\int_r^te^{c_0s}\frac{\lambda^*|u(s)|_{1/2}^2}{|u(s)|_{1/2}^{2(1-\alpha)}}\mathbf{1}_{\{|u(s)|_{1/2}>0\}}ds
\nonumber \\
\!\!\!\!\!\!\!\!&&
+\alpha\int_r^te^{c_0s}\Bigg\{\frac{\kappa_1\big((|u(s)|_{1}^{4}+1)|u(s)|_{1/2}^{2}+\|g(u(s))\|_{\mathcal{L}_2(l^2;\mathbb{H}^{1/2})}^2\big)|u(s)|_{1/2}^2}{|u(s)|_{1/2}^{2(2-\alpha)}}
\nonumber \\
\!\!\!\!\!\!\!\!&&
-\frac{2(1-\alpha)\|((g(u(t)))\cdot,u(t))_{1/2}\|_{\mathcal{L}_2(l^2;\mathbb{R})}^2}{|u(s)|_{1/2}^{2(2-\alpha)}}\Bigg\}\mathbf{1}_{\{|u(s)|_{1/2}>0\}}ds
\nonumber \\
\leq\!\!\!\!\!\!\!\!&&\tilde{\Phi}(r,|u(r)|_{1/2}) +2\alpha\big[\tilde{\mathcal{M}}^1(r,t)+\tilde{\mathcal{M}}^2(r,t)\big]+c_0\int_r^t\tilde{\Phi}(s,|u(s)|_{1/2})\mathbf{1}_{\{|u(s)|_{1/2}>0\}}  ds
\nonumber \\
\!\!\!\!\!\!\!\!&&
-\lambda^*\alpha\int_r^t\tilde{\Phi}(s,|u(s)|_{1/2})\mathbf{1}_{\{|u(s)|_{1/2}>0\}}ds,
\end{eqnarray}
where $\tilde{\mathcal{M}}^1(r,t),\tilde{\mathcal{M}}^2(r,t)$ are  continuous local martingales given by
\begin{equation*}
\tilde{\mathcal{M}}^1(r,t):=\int_r^t\mathbf{1}_{\{|u(s)|_{1/2}>0\}}\frac{(\sigma(u(s))d\mathcal{W}(s),u(s))_{1/2}}{|u(s)|_{1/2}^{2(1-\alpha)}}
\end{equation*}
and
\begin{equation*}
\tilde{\mathcal{M}}^2(r,t):=\int_r^t\mathbf{1}_{\{|u(s)|_{1/2}>0\}}\frac{(g(u(s))d\hat{\mathcal{W}}(s),u(s))_{1/2}}{|u(s)|_{1/2}^{2(1-\alpha)}}.
\end{equation*}
Choosing $c_0=\lambda^*\alpha$ and taking the conditional expectation on both side of inequality (\ref{es588}), we derive
$$\mathbf{E}\Big[e^{\lambda^*(1-\frac{\gamma}{2})t}|u(t)|_{1/2}^{2-\gamma}|\mathscr{F}_{r}\Big]\leq e^{\lambda^*(1-\frac{\gamma}{2})r}|u(r)|_{1/2}^{2-\gamma},~r<t,$$
 which completes the desired result.

\vspace{1mm}
\noindent\textbf{Step 2.} According to Step 1, we can get
  \begin{equation}\label{es599}
\mathbf{E}|u(t)|_{1/2}^{2-\gamma}\leq e^{-\lambda^*(1-\frac{\gamma}{2})t}|x|_{1/2}^{2-\gamma},~t\geq 0.
\end{equation}

Without loss of generality, we only consider $ |x|_{1/2}>0$ since, in terms of (\ref{es599}), $ u(t)\equiv 0$ $\mathbf{P}\text{-a.s.}$ if $x=0$.
Let $\kappa\in (0,\lambda^*(1-\frac{\gamma}{2}))$. Then we obtain
\begin{eqnarray*}
\!\!\!\!\!\!\!\!&&\mathbf{P}\bigg(\sup_{t\in[k,k+1]}e^{\kappa t}|u(t)|_{1/2}^{2-\gamma}\geq |x|_{1/2}^{2-\gamma}   \bigg)
\nonumber \\
\leq\!\!\!\!\!\!\!\!&&\mathbf{P}\bigg(\sup_{t\in[k,k+1]}e^{\lambda^*(1-\frac{\gamma}{2}) t}|u(t)|_{1/2}^{2-\gamma}\geq e^{(\lambda^*(1-\frac{\gamma}{2})-\kappa) k}|x|_{1/2}^{2-\gamma}   \bigg)
\nonumber \\
\leq\!\!\!\!\!\!\!\!&&\mathbf{E}\Big[e^{\lambda^*(1-\frac{\gamma}{2}) k}|u(k)|_{1/2}^{2-\gamma}\Big]\Big/\Big\{e^{(\lambda^*(1-\frac{\gamma}{2})-\kappa) k}|x|_{1/2}^{2-\gamma}\Big\}
\nonumber \\
\leq\!\!\!\!\!\!\!\!&&e^{-(\lambda^*(1-\frac{\gamma}{2})-\kappa) k},
\end{eqnarray*}
where we used the maximal supermartingale inequality in the third step, which implies
\begin{equation}\label{es600}
\sum_{k=0}^{\infty}\mathbf{P}\bigg(\sup_{t\in[k,k+1]}e^{\kappa t}|u(t)|_{1/2}^{2-\gamma}\geq |x|_{1/2}^{2-\gamma}   \bigg)<\infty.
\end{equation}
Consequently,  (\ref{es600}) and Borel-Cantelli's lemma imply that for
almost all $\omega\in \Omega$ there exists a   finite random time $\tau=\tau(\omega)$ such that
$$|u(t,\omega)|_{1/2}^{2-\gamma}\leq e^{-\kappa t}|x|_{1/2}^{2-\gamma},~~t\geq \tau(\omega).$$
We complete the proof.  \hspace{\fill}$\Box$

\subsection{Proof of ergodicity}
We first present the following lemma concerning the time-average estimates.
\begin{lemma}\label{lemta}
There exists a positive constant $C$  such that for any $T\geq 1$,
 \begin{equation}\label{e116}
\frac{1}{T}\int_0^T\mathbf{E}|u(t)|_{3/2}^{2-\gamma}dt\leq C+\frac{C|x|_{1/2}^{2-\gamma}}{T}.
\end{equation}

\end{lemma}

\begin{proof}
The proof follows from the same argument as in (\ref{es13}), we omit it.
\end{proof}

Now we are in the position to prove the existence and uniqueness of invariant measures  to Eq.~(\ref{sns2}).

\vspace{1mm}
\noindent\textbf{Proof of Theorem \ref{th3}.} In what follows, we will prove Theorem \ref{th3} in the following two steps.\\
\textbf{Step 1.} (Proof of existence  of invariant measures).
Since we have proved in Theorem \ref{thF} that the transition semigroup $(\mathcal{T}_t)_{t\geq 0}$ defined by $(\ref{semig})$ is  Feller, from the method of Krylov-Bogoliubov
 we define the occupation measure
$$\mu_n:=\frac{1}{n}\int_0^n \delta_0 \mathcal{T}_t dt,~n\geq 1,$$
where $\delta_0$ is Dirac measure at $0$. It is well-known  that for the existence of invariant measures, one only needs to verify the tightness of $\{\mu_n:n\in\mathbb{N}\}$ on $\mathbb{H}^{1/2}$.

Due to Lemma \ref{lemta}, we can get that there exists a  constant $C>0$ independent of $n$,
\begin{equation}\label{es44}
\mu_n(|\cdot|_{3/2}^{2-\gamma})=\frac{1}{n}\int_0^n\mathbf{E}| u(t,0)|_{3/2}^{2-\gamma}dt\leq C.
\end{equation}
Note that the embedding $\mathbb{H}^{3/2}\subset\mathbb{H}^{1/2}$ is compact, then, for any positive constant $K$, the set $\{u\in \mathbb{H}^{1/2}:|u|_{3/2}\leq K\}$ is relatively compact in
$\mathbb{H}^{1/2}$. Hence, the estimate (\ref{es44}) implies the tightness of $\{\mu_n:n\in\mathbb{N}\}$ on $\mathbb{H}^{1/2}$. Consequently, the limit of a convergent
subsequence provides an invariant measure $\mu$ associated to the transition semigroup $\mathcal{T}_t$.

We proceed to prove the concentration property (\ref{sup1}). Following from the above argument, there exists a subsequence, still denoted by $(\mu_n)$, such that
$$\mu_n\xrightarrow{n \to \infty}\mu~\text{in}~ \mathscr{P}(\mathbb{H}^{1/2}).$$
Using the  lower semi-continuity of norm $|\cdot|_{3/2}$ in $\mathbb{H}^{1/2}$ and the estimate (\ref{es44}), we can deduce that
\begin{eqnarray*}
\int |x|_{3/2}^{2-\gamma}\mu(dx)\leq \!\!\!\!\!\!\!\!&& \liminf_{M\to\infty}\int (|x|_{3/2}^{2-\gamma}\wedge M)\mu(dx)
\nonumber \\
\leq \!\!\!\!\!\!\!\!&& \liminf_{M\to\infty}\liminf_{n\to\infty}\int (|x|_{3/2}^{2-\gamma}\wedge M)\mu_n(dx)
\nonumber \\
\leq\!\!\!\!\!\!\!\!&& \liminf_{n\to\infty}\frac{1}{n}\int_0^n\mathbf{E}|u(t,0)|_{3/2}^{2-\gamma}dt
\nonumber \\
<\!\!\!\!\!\!\!\!&&\infty.
\end{eqnarray*}

\noindent\textbf{Step 2.} (Proof of uniqueness  of invariant measures).  It follows from Theorem \ref{th4}  that, as $t\to\infty$,  almost all sample paths of the solution to Eq.~(\ref{sns2}) will tend to the equilibrium state $u(\infty)=0$. Therefore, for any $\varphi\in C_b(\mathbb{H}^{1/2})$ and initial value $x\in \mathbb{H}^{1/2}$,
$$\varphi(u(t,x))\xrightarrow{t \to \infty} \varphi(0)~~\mathbf{P}\text{-a.s.},$$
which together with the dominated convergence theorem gives
\begin{equation}\label{es611}
\mathcal{T}_t\varphi(x)\xrightarrow{t \to \infty} \varphi(0).
\end{equation}

We now prove that $\delta_0$ is the unique invariant measure associated to the semigroup $\mathcal{T}_t$. Indeed, let $\mu$ be any invariant measure, i.e.
\begin{equation}\label{es622}
\int_{\mathbb{H}^{1/2}}\varphi(x)\mu(dx)=\int_{\mathbb{H}^{1/2}}\mathcal{T}_t\varphi(x)\mu(dx),~t\geq 0,~\varphi\in\mathscr{B}_b(\mathbb{H}^{1/2}).
\end{equation}
Let $\varphi\in C_b(\mathbb{H}^{1/2})$.  Taking $t\to\infty$ on both sides of (\ref{es622}) and using  the dominated convergence theorem,  due to (\ref{es611}) we have
\begin{equation*}
\int_{\mathbb{H}^{1/2}}\varphi(x)\mu(dx)=\varphi(0)=\int_{\mathbb{H}^{1/2}}\varphi(x)\delta_0(dx),~\varphi\in C_b(\mathbb{H}^{1/2}),
\end{equation*}
which completes the proof of the claim.
 \hspace{\fill}$\Box$

\section{Appendix}\label{appendix}

\subsection{Proof of Lemma \ref{Bes12}}\label{subsecpr}

We first prove \eref{es7}. Using H\"{o}lder's inequality and Lemma \ref{lemso}, we have
\begin{eqnarray*}
(B(u),\Lambda u)
\leq\!\!\!\!\!\!\!\!&&|B(u)|_{L^{3/2}}|\nabla u|_{L^3}
\nonumber \\
\lesssim\!\!\!\!\!\!\!\!&&|u|_{L^6} |u|_{1}|u|_{3/2}
\nonumber \\
\leq \!\!\!\!\!\!\!\!&&\frac{1}{4}| u|_{3/2}^2+C_1|u|_{1}^4
\nonumber \\
\leq \!\!\!\!\!\!\!\!&&\frac{1}{4}| u|_{3/2}^2+C_2|u|_{1}^2|u|_{1/2}|u|_{3/2}
\nonumber \\
\leq \!\!\!\!\!\!\!\!&&\frac{1}{2}| u|_{3/2}^2+\frac{3C_2}{4}|u|_{1}^4|u|_{1/2}^2 ,
\end{eqnarray*}
where we used Lemma \ref{lemso} in the third inequality, the interpolation inequality (\ref{inteq2})
in the fourth step, and Young's inequality  in the third and fifth inequalities. Thus \eref{es7} holds by taking $\delta_1=\frac{3C_2}{4}$, which depends on the constants derived from the Sobolev embedding theorem and the interpolation inequality.

For \eref{es7}, using Agmon's inequality (cf.~(A.29) in \cite{FMRT}) and Young's inequality, there exists $\delta_2>0$
\begin{eqnarray*}
(B(u),\Lambda^2 u)
\leq|u|_{L^{\infty}} |u|_{1}|u|_{2}
\lesssim|u|_{1}^{\frac{3}{2}} |u|_{2}^{\frac{3}{2}}
\leq \frac{1}{2}| u|_{2}^2+\delta_2|u|_{1}^6 .
\end{eqnarray*}

\subsection{Deterministic compactness criterions}
The following two lemmas give the compactness criterions associated to the spaces  $\mathcal{X}_1$ and  $\mathcal{X}_2$.
\begin{lemma}\label{lemc}
Let  $\mathcal{K}$ be a subset of $\mathcal{X}_1$ such that the following conditions hold:
\begin{enumerate}[$(i)$]

\item  $\sup_{u\in\mathcal{K}}\sup_{ t\in[0, T]}|u(t)|_{1/2}<\infty$;

\item $\sup_{u\in\mathcal{K}}\int_0^T|u(t)|_{3/2}^2 dt<\infty$;

\item
$$\lim_{\Delta\to 0^+}\sup_{u\in\mathcal{K}}\int_0^{T-\Delta}|u(t+\Delta)-u(t)|_{1/2}^2dt=0.$$
\end{enumerate}
Then, $\mathcal{K}$ is relatively compact in $\mathcal{X}_1$.
\end{lemma}

\begin{proof}
We sketch the proof for reader's convenience.  First, due to the compactness criterion presented in Theorem 5 of \cite{S3}, it follows from the conditions  (ii) and (iii) that $\mathcal{K}$ is relatively compact in $L^2([0,T];\mathbb{H}^{1/2})$.  Furthermore, by the Banach-Alaoglu theorem and the conditions (i) and (ii) it is clear that $\mathcal{K}$ is relatively compact in $L^2_w([0,T];\mathbb{H}^{3/2})\cap L^{\infty}_{w^*}([0,T];\mathbb{H}^{1/2})$. Consequently, we can deduce that $\mathcal{K}$ is relatively compact in $\mathcal{X}_1$.
\end{proof}

\begin{lemma}\label{lemc1}
Let  $\mathcal{K}$ be a subset of $\mathcal{X}_2$ such that the following conditions hold:
\begin{enumerate}[$(i)$]

\item  For any $k\in\mathbb{N}$ there exists a constant $C_k>0$ such that

$$\sup_{u\in\mathcal{K}}\sup_{ t\in[\frac{1}{k}, T]}|u(t)|_{1/2}\leq C_k;$$

\item
$$\lim_{\Delta\to 0^+}\sup_{u\in\mathcal{K}}\sum_{k=1}^{\infty}\frac{1}{2^k}\bigg(\sup_{|t-s|\leq \Delta, t,s\in[\frac{1}{k},T]}|u(t)-u(s)|_{-1/2}\wedge 1 \bigg)=0.$$
\end{enumerate}
Then, $\mathcal{K}$ is relatively compact in $\mathcal{X}_2$.
\end{lemma}
\begin{proof}
Let $\{u_n\}$ be a sequence in $\mathcal{K}$. For any $k\in\mathbb{N}$,    using the Arzela-Ascoli theorem  and the conditions (i) and (ii),  we can find a subsequence $\{u_n^{k}\}$ of the sequence $\{u_n^{k-1}\}$,
which is convergent in $C([\frac{1}{k},T];\mathbb{H}^{-1/2})$.  Then the result follows from the diagonal argument.
\end{proof}

\subsection{Standard Borel space}

We present the definitions of countably generated Borel spaces and standard
Borel spaces in the sense of Parthasarathy (cf. \cite[Chapter V, Definition 2.1 and 2.2]{Liang}).

\begin{definition}
$($Countably generated Borel space$)$ A
Borel space $(\mathbb{X}, \mathscr{B}_{\mathbb{X}})$ is said to be countably generated if there exists a denumerable class
$\mathcal{D} \subset \mathscr{B}_{\mathbb{X}}$ such that $\mathcal{D}$ generates $\mathscr{B}_{\mathbb{X}}$.
\end{definition}

\begin{definition}\label{de5}
$($Standard Borel space$)$ A countably
generated Borel space $(\mathbb{X}, \mathscr{B}_{\mathbb{X}})$ is called standard if there exists a complete separable mertic
space $(\mathbb{Y},\mathscr{B}_{\mathbb{Y}})$ such that the $\sigma$-algebras $\mathscr{B}_{\mathbb{X}}$ and $\mathscr{B}_{\mathbb{Y}}$ are $\sigma$-isomorphic.
\end{definition}

To apply the Jakubowski's version of the Skorokhod theorem, we recall the following result from \cite{Liang}.

\begin{theorem}\label{th2}$($Theorem B.4 in \cite{Liang}$)$
Let $(\mathbb{X}, \mathscr{B}_{\mathbb{X}})$ be any standard Borel space. Suppose that $\{f_m\}_{m\in\mathbb{N}}$ is an
$\mathscr{B}_{\mathbb{X}}$-measurable sequence from $\mathbb{X}$ to $\mathbb{R}$, which separate the points of $\mathbb{X}$. Denote by $\sigma_0(\mathbb{X})  $ the $\sigma$-algebra generated by  $\{f_m\}_{m\in\mathbb{N}}$. Then $\sigma_0(\mathbb{X})=\mathscr{B}_{\mathbb{X}}$.
\end{theorem}

\subsection{Stochastic Gronwall's lemma}\label{subsec3}

We recall the following stochastic Gronwall's lemma (cf. \cite[Lemma
5.3]{GZ1}).
\begin{lemma}\label{appen2}
Fix $T>0$. Assume that
$X,Y,Z,R:[0,T)\times\Omega\rightarrow\mathbb{R}$ are real-valued,
non-negative stochastic process. Let $\tau<T$ be a stopping time
such that $$\mathbf{E}\int_0^\tau(RX+Z)ds<\infty.$$ Assume,
moreover, that for some fixed constant $\kappa$,
$$\int_0^\tau{R}ds<\kappa,~~~\text{a.s.}$$
Suppose that for all stopping times $0\leq\tau_a<\tau_b\leq\tau$
\begin{eqnarray}{\mathbf{E}}
\bigg[\sup_{t\in[\tau_a,\tau_b]}X+\int_{\tau_a}^{\tau_b}Yds\bigg]\leq c_0{\mathbf{E}}\bigg[X(\tau_a)+\int_{\tau_a}^{\tau_b}(RX+Z)ds\bigg],\label{a.1}
\end{eqnarray}
where $c_0$ is a constant independent of the choice of
$\tau_a,\tau_b$. Then
\begin{eqnarray}{\mathbf{E}}\bigg[\sup_{t\in[0,\tau]}X+\int_{0}^{\tau}Yds\bigg]\leq c{\mathbf{E}}\bigg[X(0)+\int_{0}^{\tau}Zds\bigg],\label{a.2}
\end{eqnarray}
where $c=c_{c_o,T,\kappa}$.
\end{lemma}

\vspace{3mm}


%
%
%
%
%
%


\begin{thebibliography}{2}
\bibitem{AV24}
\textsc{Agresti, A.} and \textsc{Veraar, M.} (2024).
 Stochastic Navier-Stokes equations for turbulent flows in critical spaces.
\textit{ Comm. Math. Phys.}
 \textbf{405}, Paper No.~43, 57 pp.


\bibitem{APS}
\textsc{Autuori, G., Pucci, P.} and \textsc{Salvatori, M. C.} (2010).
Global nonexistence for nonlinear Kirchhoff systems.
\textit{Arch. Ration. Mech. Anal.}
 \textbf{196}, 489-516.






\bibitem{AKX25}
\textsc{Aydin, M. S., Kukavica, I.} and \textsc{Xu, F.} (2025).
Almost global existence for the stochastic Navier-Stokes equations with small $H^{1/2}$ data.
	\textit{arXiv:2501.10331.}
%
\bibitem{BBT}
\textsc{Bae, H., Biswas, A.} and  \textsc{Tadmor, E.} (2012).
Analyticity and decay estimates of the Navier-Stokes equations in critical Besov spaces.
\textit{Arch. Ration. Mech. Anal.}
 \textbf{205}, 963-991.


\bibitem{BMX}
\textsc{Bagnara, M.,  Maurelli, M.} and  \textsc{Xu, F.} (2023).
No blow-up by nonlinear It\^{o} noise for the Euler equations.
\textit{Electron. J. Probab.} \textbf{30}, Paper No. 81, 29 pp.

%
%

  \bibitem{BBP22}
\textsc{Bedrossian, J., Blumenthal, A.} and  \textsc{Punshon-Smith, S.} (2022).
Almost-sure exponential mixing of passive scalars by the stochastic Navier-Stokes equations.
\textit{Ann. Probab.} \textbf{50}, 241-303.

  \bibitem{BZPW}
\textsc{Bedrossian, J., Zelati, M., Punshon-Smith, S.} and  \textsc{Weber, F.} (2020).
Sufficient conditions for dual cascade flux laws in the stochastic 2d Navier-Stokes equations.
\textit{Arch. Ration. Mech. Anal.} \textbf{237}, 103-145.



  \bibitem{BS10}
\textsc{Benameur, J.}  and  \textsc{Selmi, R.} (2010).
Long-time behavior of periodic Navier-Stokes equations in critical spaces. Progress in analysis and its applications, 597–603, \textit{World Sci. Publ., Hackensack, NJ}.



 \bibitem{BT73}
\textsc{Bensoussan, A.} and  \textsc{Temam, R.} (1973).
\'{E}quations stochastiques du type Navier-Stokes.
\textit{J. Funct. Anal.} \textbf{13}, 195-222.
%
%
%
%
%
%
%

\bibitem{BM90}
\textsc{Borchers, W.} and \textsc{Miyakawa, T.} (1990).
 Algebraic $L^2$ decay for Navier-Stokes flows in exterior domains.
 Acta Math. \textbf{165}, no. 3-4, 189-227.


\bibitem{BM13}
\textsc{Brze\'{z}niak, Z.}  and  \textsc{Motyl, E.} (2013).
Existence of a martingale solution of the stochastic Navier-Stokes equations in unbounded 2D and 3D domains.
\textit{J. Differential Equations}
\textbf{254}, 1627-1685.


\bibitem{BO}
\textsc{Brze\'{z}niak, Z.}  and  \textsc{Ondrej\'{a}t, M.} (2013).
Stochastic geometric wave equations with values in compact Riemannian homogeneous spaces.
\textit{Ann. Probab.}
\textbf{41}, 1938-1977.













\bibitem{Bur39}
\textsc{Burgers, J.M.} (1939).
Mathematical examples illustrating relations occurring in the theory of turbulent fluid motion.
\textit{Trans. Roy. Neth. Acad. Sci. Amsterdam}
\textbf{17}, 1-53.

\bibitem{Bur48}
\textsc{Burgers, J.M.} (1948).
A mathematical model illustrating the theory of turbulence, in Advances in Applied Mechanics 1.
\textit{Academic Press, New York}
pp. 171-199.

\bibitem{CP96}
\textsc{Cannone, M.} and \textsc{Planchon, F.} (1996).
Self-similar solutions for Navier-Stokes equations in $\textbf{R}^3$,
\textit{Comm. Partial Differential Equations}
 \textbf{21} , 179-193.


\bibitem{CD19}
\textsc{Cerrai, S.} and \textsc{Debussche, A.} (2019).
Large deviations for the two-dimensional stochastic Navier-Stokes equation with vanishing noise correlation.
\textit{Ann. Inst. Henri Poincar\'{e} Probab. Stat.}
\textbf{55}, 211-236.






\bibitem{CR05}
\textsc{Consiglieri, L.} and \textsc{Rodrigues, J. F.} (2005).
On stationary flows with energy dependent nonlocal viscosities.
\textit{J. Math. Sci.}
\textbf{127}, 1875-1885



%
%
%
%
%
%
\bibitem{CL24}
\textsc{Crisan, D.} and \textsc{Lang, O.} (2025). Global solutions for stochastically controlled fluid dynamics models.
\textit{Stoch PDE: Anal Comp. } https://doi.org/10.1007/s40072-025-00396-7.



\bibitem{DD03}
\textsc{Da Prato, G.} and \textsc{Debussche, A.} (2003).
 Ergodicity for the 3D stochastic Navier-Stokes equations.
 \textit{J. Math. Pure. Appl.} \textbf{82}, 877-947.

\bibitem{DKL}
\textsc{Deng, K., Kwong, M.} and \textsc{Levine, H.} (1992).
The influence of nonlocal nonlinearities on the long time behavior of solutions of Burgers' equation.
 \textit{Quart. Appl. Math.} \textbf{50}, 173-200.



%
%
%
%
%
%
%
%
%
%
%
%

\bibitem{FRS}
\textsc{Fernando, B.P.W.,  R\"{u}diger, B.} and \textsc{Sritharan, S.S.}  (2015).
Mild solutions of stochastic Navier-Stokes equation
with jump noise in $L^p$-spaces,
\textit{Math. Nachr.}
 \textbf{288}, 1615-1621.

\bibitem{FGP10}
\textsc{Flandoli, F., Gubinelli, M.} and \textsc{Priola, E.} (2010).
Well-posedness of the transport equation by stochastic perturbation.
\textit{Invent. Math.} \textbf{180}, 1-53.

\bibitem{FL21}
\textsc{Flandoli, F.} and  \textsc{Luo, D.} (2021).
High mode transport noise improves vorticity blow-up control in 3D Navier-Stokes equations.
\textit{Probab. Theory Related Fields} \textbf{180}, 309-363.

\bibitem{FR08}
\textsc{Flandoli, F.} and  \textsc{Romito, M.} (2008).
Markov selections for the 3D stochastic Navier-Stokes equations.
\textit{Probab. Theory Related Fields} \textbf{140}, 407-458.


\bibitem{FMRT}
\textsc{Foias, C., Manley, O., Rosa, R.} and \textsc{Temam, R.} (2001).
Navier-Stokes equations and turbulence (Vol. 83).
Cambridge University Press.




\bibitem{FK}
\textsc{Fujita, H.} and \textsc{Kato, T.} (1964).
On the Navier-Stokes initial value problem. I.
\textit{Arch. Ration. Mech. Anal.}
\textbf{16}, 269-315.

\bibitem{FT25}
\textsc{Fujii, M.} and \textsc{Tsurumi H.} (2025).
Asymptotic instability for the forced Navier--Stokes equations in critical Besov spaces,
\textit{arXiv:2509.21272}.

%
%
%
%
%
%
%


\bibitem{G13}
\textsc{Ghisi M.}  (2013).
 Asymptotic limits for mildly degenerate Kirchhoff equations,
 \textit{SIAM J. Math. Anal.}
 \textbf{45}, 1886-1906..




\bibitem{GM85}
\textsc{Giga Y.}  and \textsc{Miyakawa, T.} (1985).
 Solutions in $L^r$ of the Navier-Stokes initial value problem,
 \textit{Arch. Rational Mech. Anal.}
 \textbf{89}, 267-281.

\bibitem{GV14}
\textsc{Glatt-Holtz, N.} and  \textsc{Vicol, V.C.} (2014).
Local and global existence of smooth solutions for the stochastic Euler equations with multiplicative noise.
\textit{Ann. Probab.} \textbf{42}, 80-145.

\bibitem{GZ1}
\textsc{Glatt-Holtz, N.} and  \textsc{Ziane, M.} (2009).
Strong pathwise solutions of the stochastic Navier-Stokes system.
\textit{Adv. Differential Equations} \textbf{14}, 567-600.

\bibitem{GC1}
\textsc{Goodair, D.} and \textsc{Crisan, D.} (2024).
Stochastic calculus in infinite dimensions and SPDEs.
SpringerBriefs in Mathematics. Springer, Cham.


%
%
%
%
%
\bibitem{HM06}
\textsc{Hairer, M.} and \textsc{Mattingly, J. C.} (2006).
 Ergodicity of the 2D Navier-Stokes equations with degenerate stochastic forcing.
 \textit{Ann. of Math. (2)}
 \textbf{164}, 993-1032.
%
%
%



\bibitem{HLL24}
\textsc{Hong, W., Li, S.} and \textsc{Liu, W.} (2024).
Regularization by nonlinear noise for PDEs: well-posedness and finite time extinction,
arXiv:2407.06840.

\bibitem{HO}
\textsc{Horgan, C.} and \textsc{Olmstead, W.} (1978).
Stability and uniqueness for a turbulence model of Burgers,
\textit{Quart. Appl. Math.}
\textbf{36}, 121-127.


\bibitem{HZZ}
\textsc{Hofmanov\'{a}, M., Zhu, R.-C.} and \textsc{Zhu, X.-C.} (2024).
 Nonuniqueness in law of stochastic 3D Navier-Stokes equations. J. Eur. Math. Soc. (JEMS) \textbf{26}, 163-260.





%
%
%
%
%
%
%
%
%
%
%
%
%

\bibitem{KS18}
\textsc{Kavallaris, N.} and \textsc{Suzuki, T.} (2018).
Non-Local Partial Differential Equations for Engineering and Biology.
Mathematical Modeling and Analysis. Mathematics for Industry (Tokyo), \textit{Springer, Cham.}
 \textbf{31.}   300 pp.



\bibitem{Ka84}
\textsc{Kato, T.} (1984).
Strong $L^p$-solutions of the Navier-Stokes equation in $\textbf{R}^m$, with applications to weak
solutions.
 \textit{Math. Z.}
  \textbf{187}, 471-480.

\bibitem{KP88}
\textsc{Kato, T.} and \textsc{Ponce G.} (1988).
Commutator estimates and the Euler and Navier-Stokes equations.
\textit{Comm Pure Appl Math}
 \textbf{41}, 891-907.




\bibitem{KK11}
\textsc{Kenig, C.E.} Koch, \textsc{Gabriel S.} (2011).
An alternative approach to regularity for the Navier-Stokes equations in critical spaces.
\textit{Ann. Inst. H. Poincar\'{e} C Anal. Non Lin\'{e}aire}
 \textbf{28}, 159-187.


\bibitem{Kim10}
\textsc{Kim, J.U.} (2010).
Strong solutions of the stochastic Navier-Stokes equations in $\textbf{R}^3$.
\textit{Indiana Univ. Math. J.}
\textbf{59}, 1417-1450.


\bibitem{K83}
\textsc{Kirchhoff, G.} (1883).
Vorlesungen \"{u}ber Mechanik.
\textit{Teubner, Sluttgart}.







\bibitem{KT01}
\textsc{Koch, H.} and \textsc{Tataru, D.} (2001).
 Well-posedness for the Navier-Stokes equations.
 \textit{Adv. Math.}
   \textbf{157}, 22-35.

 \bibitem{KXZ22}
\textsc{Kukavica, I., Xu, F.} and \textsc{Ziane, M.}  (2022).
Global existence for the stochastic Navier-Stokes equations with small $L^p$ data.
\textit{Stoch. Partial Differ. Equ. Anal. Comput.}
\textbf{10}, 160-189.

%





 \bibitem{KNS20}
\textsc{Kuksin, S., Nersesyan, V.} and \textsc{Shirikyan, A.} (2020).
Exponential mixing for a class of dissipative PDEs with bounded degenerate noise.
\textit{Geom. Funct. Anal.} \textbf{30}, 126-187.



\bibitem{Lad67}
\textsc{Ladyzenskaja, O.A.} (1967).
 New equations for the description of the motions of viscous incompressible fluids, and global solvability for their boundary value problems.
 \textit{Trudy Mat. Inst. Steklov.} \textbf{102}, 85-104.

\bibitem{Le34}
\textsc{Leray, J.} (1934).
 Sur le mouvement d'un liquide visqueux emplissant l'espace.
 \textit{Acta Math.}
  \textbf{63}, 193–248.



\bibitem{Liang}
\textsc{Liang, S.} (2021).
Stochastic hypodissipative
hydrodynamic equations:
well-poseness, stationary
solutions and ergodicity.
PhD Thesis, Bielefeld University.



%
\bibitem{LR1}
\textsc{Liu, W.} and \textsc{R\"{o}ckner, M.} (2015).
Stochastic Partial Differential Equations: An Introduction.
\textit{Universitext, }
 Springer.




\bibitem{MR04}
\textsc{Mikulevicius, R.} and \textsc{Rozovskii, B. L.} (2004).
Stochastic Navier-Stokes equations for turbulent flows.
\textit{SIAM J. Math. Anal.}
 \textbf{35},1250-1310.




%
%
%
%
%
%
%

\bibitem{OE00}
\textsc{Oliver, M.} and \textsc{Titi, E.S.} (2000).
Remark on the rate of decay of higher order derivatives for solutions to the Navier-Stokes equations in $R^n$.
\textit{J. Funct. Anal.}
\textbf{172}, 1-18.







\bibitem{RSZ}
\textsc{R\"{o}ckner, M., Schmuland, B.} and \textsc{Zhang, X.} (2008).
Yamada-Watanabe theorem for stochastic evolution equations in infinite dimensions.
\textit{Condens. Matter Phys.}
\textbf{54}, 247-259.
%
%
%
\bibitem{RZZ14}
\textsc{R\"{o}ckner, M., Zhu, R.-C.} and \textsc{Zhu, X.-C.} (2014).
 Local existence and non-explosion of solutions for stochastic fractional partial differential equations driven by multiplicative noise.
\textit{Stochastic Process. Appl.} \textbf{124}, 1974-2002.

\bibitem{Tem95}
\textsc{Temam, R.} (1995).
  Navier–Stokes Equations and Nonlinear Functional Analysis, second edition.
\textit{CBMS-NSF RegionalConference Series in Applied Mathematics},
  vol. 66, Society for Industrial and Applied Mathematics (SIAM), Philadelphia, PA.


\bibitem{Sc85}
\textsc{Schonbek, M.} (1985).
$L^2$ decay for weak solutions of the Navier-Stokes equations.
\textit{Arch. Rational Mech. Anal.}
 \textbf{88}, 209-222.

\bibitem{Sc91}
\textsc{Schonbek, M.} (1991).
Lower bounds of rates of decay for solutions to the Navier-Stokes equations.
\textit{J. Amer. Math. Soc.}
 \textbf{4}, 423-449.


\bibitem{Ste70} \textsc{Stein, E.} (1970).
Singular Integrals and Differentiability Properties of Functions.
\textit{Princeton, NJ: Princeton University Press}.
%
\bibitem{S3}
\textsc{Simon J.} (1987).
Compact sets in the space $L^p(0,T;B)$.
\textit{Ann. Mat. Pura Appl.}
\textbf{164}, 65-96.

\bibitem{TW22}
\textsc{Tang, H.} and \textsc{Wang, F. Y.} (2022).
A general framework for solving singular SPDEs with applications to fluid models driven by pseudo-differential noise.
\textit{arXiv:2208.08312}.

\bibitem{WK50}
\textsc{Woinowsky-Krieger S.
} (1950).
The effect of an axial force on the vibration of hinged bars.
\textit{J. Appl. Mech.}
\textbf{17}, 35-36.









%
%
%


%
%
%
%
%
%
%
%
%
\end{thebibliography}
\end{document}